\documentclass[10pt,twoside]{report}
 \usepackage{amssymb,latexsym,amsmath,mathrsfs,stmaryrd,amscd,multirow,pgf,tikz,amsthm,hyperref}
 \usepackage[all]{xy}
 \usepackage{geometry}
 \newdir{ >}{{}*!/-9pt/\dir{>}}
 
 \newtheorem{thm}{Theorem}[section]
 \newtheorem{athm}{Theorem}[chapter]
 \newtheorem{aprop}[athm]{Proposition}
 \newtheorem{alem}[athm]{Lemma}
 \newtheorem{cor}[thm]{Corollary}
 
 \newtheorem{lem}[thm]{Lemma}
 \newtheorem{prop}[thm]{Proposition}
 
 \newtheorem{conjecture}[thm]{Conjecture}
 
 \newtheorem{claim}[thm]{Claim}
 
 \theoremstyle{definition}
 \newtheorem{defn}[thm]{Definition}
 \newtheorem{altdefn}[thm]{Alternative Definition}
 
 \newtheorem{rem}[thm]{Remark}

 \newtheorem{note}[thm]{Note}

 \numberwithin{equation}{chapter}

 \DeclareMathOperator{\IM}{Im}
 \DeclareMathOperator{\Exp}{exp}
 \DeclareMathOperator{\Hom}{Hom}
 
 \DeclareMathOperator{\Map}{Map}

 \DeclareMathOperator{\Dim}{dim}

 \DeclareMathOperator{\Ker}{Ker}

 \DeclareMathOperator{\End}{End}

 \DeclareMathOperator{\inj}{Inj}
 
 \DeclareMathOperator{\hocolim}{hocolim}
 \DeclareMathOperator{\FMap}{FMap}

 \newcommand{\Real}{\mathbb{R}}
 \newcommand{\Complex}{\mathbb{C}}

 \newcommand{\redKstar}{\tilde{K}^{*}_{G}}
 \newcommand{\redKone}{\tilde{K}^{1}_{G}}
 \newcommand{\redKzero}{\tilde{K}^{0}_{G}}
 \newcommand{\Kstar}{K^{*}_{G}}
 \newcommand{\Kone}{K^{1}_{G}}
 \newcommand{\Kzero}{K^{0}_{G}}
 
 \newcommand{\pt}{\text{pt}}
 \newcommand{\st}{\tilde{\Sigma}}
 \newcommand{\sst}{\Sigma\tilde{\Sigma}}

 \newcommand{\Cat}{\mathcal{C}}
 
 \newcommand{\Top}{\mathcal{T}}
 \newcommand{\Cube}{\mathfrak{X}}
 \newcommand{\pos}{\mathcal{P}(S)}
 
 \newcommand{\tcx}{\tilde{c}\Cube}
 \newcommand{\Cubey}{\mathcal{Y}}
 \newcommand{\Cubez}{\mathcal{Z}}
 \newcommand{\pones}{\mathcal{P}_{1}(S)}

 \newcommand{\Ell}{\mathcal{L}(V_{0},V_{1})}
 \newcommand{\veen}{V_{n}(\Complex^{n+t})}
 \newcommand{\aich}{\Hom(V_{0},V_{1})}
 \newcommand{\sovo}{s_{0}(V_{0})}
 \newcommand{\svo}{s(V_{0})}

\begin{document}

\title{\huge \bf{The Equivariant Stable
Homotopy Theory Around Isometric Linear Maps}}
\author{\LARGE\bf{Harry Edward Ullman}\\\\
\large Submitted for the degree of Doctor of Philosophy}
\date{School of Mathematics and Statistics\\
University of Sheffield\\
July 2010} 
\maketitle
\newpage

\begin{abstract} The non-equivariant topology of Stiefel manifolds has been studied extensively, culminating in a result of Miller demonstrating that a Stiefel manifold splits stably to a wedge of Thom spaces over Grassmannians. Equivariantly, one can consider spaces of isometries between representations as an analogue to Stiefel manifolds. This
concept, however, yields a different theory to the non-equivariant case. We first construct a variation on the theory of the functional calculus before studying the homotopy-theoretic properties of this theory. This allows us to construct the main result; a natural tower of $G$-spectra running down from equivariant isometries which manifests the pieces of the non-equivariant splitting in the form of the homotopy cofibres of the tower. Furthermore, we detail extra topological properties and special cases of this theory, developing explicit expressions covering the tower's geometric and topological structure. We conclude with two detailed conjectures which provide an avenue for future study. Firstly we explore how our theory interacts with the splitting of Miller, proving partial results linking in our work with Miller's and conjecturing even deeper connections. Finally, we begin to calculate the equivariant $K$-theory of the tower, conjecturing and providing evidence towards the idea that the rich topological structure will be mirrored on the
$K$-theory level by a similarly deep algebraic structure.
\end{abstract}

\pagenumbering{roman} \setcounter{page}{3}
\addcontentsline{toc}{chapter}{Acknowledgments}

\section*{Acknowledgments}

Writing a thesis is never easy, and there are a number of names I wish to mention in thanks - without their support, encouragement and advice this document would probably never have been started, let alone finished!

Firstly I wish to thank my mother, Gina Ullman, who has helped me through thick and thin towards where I am today. I choose to use this thesis as a constant reminder as to just how important she has been to my studies and development. Thanks Mum! I also wish to thank the rest of my family; my dad, aunts and uncles, grandmother and everyone else who has supported me.

Great thanks must go to my supervisor, Neil Strickland, who helped and advised me throughout both the good times and the bad. His comments and guidance have been invaluable and without them my thesis would never have made it this far. I also wish to thank all the staff, researchers and postgraduate students of the Sheffield University Pure Maths Department for making it into a great working environment. Thanks go to all those who have offered advice, mathematical or otherwise, shared an office with me, shared a house with me or even just put up with me. Many people within the department have made my time of study a pleasurable one and I offer my thanks to them all. I also offer my thanks to the EPSRC for funding and supporting my study.

Non mathematically, there are probably too many names to mention but I'd like to give it a go. I wish to thank the following. Cath Howdle and C-J Donnelly for checking and re-checking my spelling, grammar and general writing style far beyond the call of duty. Rosalind Higman, Emilie Yerby and Phil Scott for letting me visit them and thus letting me spend some amazing periods of time working on this thesis abroad. Caroline Chambers for useful mathematical hints
despite being a vet who dropped all traces of maths from her life age 18. Clemency Cooper for great encouragement and motivation through the medium of self-betterment star charts. Francesca Holdrick and Laura Daly for being the first people to `publish' my results (on their fridge), and for being my co-authors when writing non-mathematics. Finally, I thank everyone who has ever considered themselves a VOLE (they know who they are); each one has been a great friend who has kept me thoroughly on the right side of sanity. I leave it open as to which side this is.

\newpage

\pdfbookmark[0]{Contents}{toc}
\tableofcontents

\pagenumbering{arabic}
\setcounter{page}{6}
\chapter{Introduction}
\label{ch:ch1}

\section{Motivation - The Miller Splitting Theorem}\label{motivationthemillersplittingtheorem}

Let $\veen$ be the Stiefel manifold of $n$-frames in $\Complex^{n+t}$. One can also think of $\veen$ as the space of
all linear isometries $\Complex^{n}\to \Complex^{n+t}$. In particular, the unitary group $U(n)$ is the Stiefel manifold $V_{n}(\Complex^{n})$. 

For an unbased space $X$ we use the notation $X_{\infty}$ for the Alexandroff one-point compactification of $X$, which we think of as a based space based at the added point at infinity. Using this terminology we can study the homotopy theory of $\veen_{\infty}$, the based version of the naturally unbased space $\veen$. The stable homotopy theory of these spaces was first studied by James, who in \cite{JamesStiefelManifolds} demonstrated that the suspension spectrum $\Sigma^{\infty}\Sigma\Complex P^{n-1}_{\infty}$ splits off stably from the spectrum $\Sigma^{\infty}U(n)_{\infty}$. 

This result prompted many interesting avenues of research, for example the development of a stable splitting of the loop space of the unitary group, $\Omega U(n)$. Moreover, the result led to calculations regarding the homotopy theory of the classifying space of the unitary group, $BU(n)$.

Further research came from the study of the part left when one splits off $\Sigma^{\infty}\Sigma\Complex P^{n-1}_{\infty}$ from $\Sigma^{\infty}U(n)_{\infty}$. This idea was studied by Miller, who in \cite{Miller} proved that Stiefel manifolds have a stable splitting. The proof of this result can also be found in the survey article \cite{Crabb}, within the paper \cite{Kitchloo} and in Appendix \ref{ch:appendixa} of this document.

In order to state the Miller splitting we first need to establish some notation. Firstly we let $G_{k}(\Complex^{n})$ denote the Grassmannian of $k$-dimensional subspaces of $\Complex^{n}$. Further, we let $T$ be the tautological bundle over $G_{k}(\Complex^{n})$:
$$T:=\{(V,v):V\in G_{k}(\Complex^{n}),v\in V\}.$$

We also wish to define one other bundle over $G_{k}(\Complex^{n})$. Firstly we note that the unitary group $U(k)$ has as its associated Lie algebra the space $\mathfrak{u}(k)$ of skew-Hermitian $k\times k$ matrices. Also note that $\mathfrak{u}(k)$ has an action of $U(k)$ via the adjoint representation. We take the bundle $T$ to have structure group $U(k)$ and we denote the principal bundle by $P$. Hence via the Borel construction we define the following bundle over $G_{k}(\Complex^{n})$:
$$\mathfrak{u}(T):=P \times_{U(k)}\mathfrak{u}(k).$$

Finally, we use the notation $X^{V}$ to denote the Thom space of the vector bundle $V\to X$.

\begin{thm}[Miller]\label{themillerthm} There is a stable homotopy equivalence:
$$\Sigma^{\infty}\veen_{\infty}\simeq\bigvee_{k=0}^{n}\Sigma^{\infty}G_{k}(\Complex^{n})^{\mathfrak{u}(T)\oplus \Hom(T,\Complex^{t})}.$$
\end{thm}

Let $I$ be a choice of inclusion $\Complex^{n}\to \Complex^{n+t}$. The proof of the Miller splitting begins with the definition of the Miller filtration:
$$F_{k}(\veen):=\{A\in \veen:\text{rank}(A-I)\leqslant k\}.$$

Miller first demonstrated that the space $F_{k}(\veen)\backslash F_{k-1}(\veen)$ is in fact diffeomorphic to the total space of the bundle $\mathfrak{u}(T)\oplus \Hom(T,\Complex^{t})$ over $G_{k}(\Complex^{n})$ and that this diffeomorphism is compatible with the projection maps down to the base space. Hence he extends this to the following based homeomorphism:
$$\frac{F_{k}(\veen)}{F_{k-1}(\veen)}_{\infty}\cong G_{k}(\Complex^{n})^{\mathfrak{u}(T)\oplus \Hom(T,\Complex^{t})}.$$

Thus there is a collapse $F_{k}(\veen)_{\infty}\to G_{k}(\Complex^{n})^{\mathfrak{u}(T)\oplus \Hom(T,\Complex^{t})}$. Miller then builds a stable map $\Sigma^{\infty}G_{k}(\Complex^{n})^{\mathfrak{u}(T)\oplus \Hom(T,\Complex^{t})}\to \Sigma^{\infty}F_{k}(\veen)_{\infty}$ such that composition with the stabilization of the collapse map gives the identity self-map of $\Sigma^{\infty}G_{k}(\Complex^{n})^{\mathfrak{u}(T)\oplus \Hom(T,\Complex^{t})}$. This map is referred to in the literature as the splitting map. This is then enough to prove the splitting via some standard results in homotopy theory.

Moreover, this splitting also appears in cohomology. The result below is well known for ordinary integral cohomology, it was proved by Kitchloo in \cite{Kitchloo} for the stated level of generality:

\begin{thm}[Kitchloo]\label{kitchloosresult} Let $E$ be a multiplicative complex oriented cohomology theory. Then $\tilde{E}^{*}(U(n)_{\infty})$ is an exterior algebra of dimension $n$ over $E^{*}(\pt)$ generated by $E^{*}(\Complex P^{n-1})$. Moreover, we have the following cohomology level composition of the splitting map with the inclusion $F_{k}(U(n))_{\infty}\to U(n)_{\infty}$:
$$\tilde{E}^{*}(U(n)_{\infty})\to\tilde{E}^{*}(F_{k}(U(n))_{\infty})\to \tilde{E}^{*}(G_{k}(\Complex^{n})^{\mathfrak{u}(T)\oplus \Hom(T,\Complex^{t})}).$$

Then the cohomology of $\tilde{E}^{*}(G_{k}(\Complex^{n})^{\mathfrak{u}(T)\oplus \Hom(T,\Complex^{t})})$ is isomorphic to the $k^{th}$ exterior power of the exterior algebra and the map above is the projection to this exterior power. Similar results hold for $\veen$.
\end{thm}

This was then extended by Strickland in Section $4$ of \cite{StricklandSubbundles} to cover unitary bundles. Overall, this gives a rather satisfying picture of the homotopy theory of $\veen$ - the space splits up stably into finite geometrically nice pieces in a natural way, further this splitting can be identified in cohomology. The goal of this document is to explore an equivariant generalization of these ideas. 

Normally with equivariant topology the first thing to do is see what happens in the na\"ive case, fix a group $G$, put $G$ in front of everything and hope things work. Within this framework it is possible to do this, however, this loses some information. The equivariant analogue of $\veen$ is not just $\veen$ with an action, rather it is something slightly different which cannot be studied simply by studying an equivariant form of the above theory. However, results similar to those depicted above can be found using different techniques, and the spirit of Miller's and Kitchloo's results motivate study of the equivariant question.

\section{How to Generalize}\label{howtogeneralize}

Throughout this section and most of the rest of the document we let $G$ be a fixed compact Lie group and work $G$-equivariantly; we take all our actions to be from the left. In order to equivariantly generalize the Miller splitting we need to come up with an equivariant analogue of $\veen$. As mentioned earlier, one can consider $\veen$ as a space of isometries. This leads us towards the following definition:

\begin{defn}\label{ell} Let $V_{0}$ and $V_{1}$ be two complex $G$-representations of finite dimension with invariant inner product. We let $d_{0}:=\Dim(V_{0})$ and $d_{1}:=\Dim(V_{1})$, then we restrict to the case where $d_{0}\leqslant d_{1}$. The other case is easily observed to be trivial. Define $\Ell$ to be the space of isometric linear maps from $V_{0}$ to $V_{1}$ with the following group action:
$$g.\theta(v)=g\theta(g^{-1}v).$$
Here $g\in G$, $\theta\in\Ell$ and $v\in V_{0}$.
\end{defn}

\begin{rem}\label{elliscompact} $\Ell$ is compact. Hence the based space $\Ell_{\infty}$ has an isolated basepoint.
\end{rem}

This gives us our equivariant analogue for $\veen$. Sadly, however, we cannot simply repeat Miller's process for $\Ell$. In the first stage of the splitting proof one built a filtration of $\veen$ using a canonical choice of inclusion $I:\Complex^{n}\to \Complex^{n+t}$. Equivariantly, however, we may well not have this choice. If $V_{0}$ and $V_{1}$ have unmatching group actions then there is no equivariant inclusion $V_{0}\to V_{1}$ and thus no starting filtration. In order to include $V_{0}$ into $V_{1}$ in any reasonable way one would require $V_{0}$ to be a subrepresentation of
$V_{1}$. In this case we have a filtration and it is easy to see that an equivariant Miller splitting follows, the proofs in the cited texts or in Appendix \ref{ch:appendixa} are natural enough to preserve any group actions. If the actions are incompatible, however, then this approach is simply unworkable and we need a new angle of attack.

In lieu of the above, $\Ell$ will not in general split stably in a fashion matching the non-equivariant case. However, equivariant analogues of the Thom spaces in the non-equivariant splitting do play a role in the equivariant stable homotopy theory around $\Ell$. Before stating the main result we first make a single technical change. 

\begin{defn}\label{svo} Let $V$ be a finite dimensional Hermitian vector space. Then by choosing an orthonormal basis for $V$ one can write an endomorphism of $V$ as a matrix $\alpha$. Set the adjoint matrix of $\alpha$ to be $\alpha^{\dag}:=\bar{\alpha}^{t}$, the conjugate-transpose of $\alpha$.

Now let $s(V)$ denote the space of selfadjoint endomorphisms of $V$:
$$s(V):=\{\alpha\in\End(V):\alpha=\alpha^{\dag}\}.$$

Elements of $s(V)$ are called positive if their eigenvalues are positive real numbers. Let $s_{+}(V)$ be the space of weakly positive (i.e. nonnegative) selfadjoint endomorphisms  of $V$ and let $s_{++}(V)$ be the space of strictly positive endomorphisms of $V$. We also note here that $s(V)$ carries the structure of a real vector space, thus we can think of $s(V)_{\infty}$ as a sphere and for convenience of notation set $S^{s(V)}:=s(V)_{\infty}$.

Finally for any Hermitian space $W$ we have the bundle $s(T)$ over $G_{k}(W)$ given by:
$$s(T):=\{(V,\alpha):V\in G_{k}(W),\alpha\in s(V)\}.$$
\end{defn}

Miller's work was phrased in terms of the unitary Lie algebra $\mathfrak{u}(V)$ and it's associated bundles; we choose to switch to phrasing things in terms of $s(V)$ as it proves to be more compatible with the functional calculus
techniques we use in Chapter \ref{ch:ch3}. This is easily done as multiplication by the imaginary number $i$ gives a homeomorphism $\mathfrak{u}(V)\cong s(V)$. Moreover, the Miller splitting can be reproved using $s(V)$ from the start rather than just using the above homeomorphism at the end of the proof; we do this in Appendix \ref{ch:appendixa}.

We note two more things before stating the main result. Firstly, we use the term `$G$-spectrum' in the below statement. We refer the reader to Section \ref{onthecategories} to detail precisely what we mean but we note here that our $G$-spectra are ones indexed on a chosen complete $G$-universe and that an equivariant suspension spectrum functor $\Sigma^{\infty}$ does exist. We also note a single convenience of notation. We use as a shorthand $\Hom(T,V_{1}-V_{0})$ for the virtual bundle $\Hom(T,V_{1})-\Hom(T,V_{0})$ over $G_{k}(V_{0})$. We can now state the theorem:

\begin{thm}\label{introresult} There is a natural tower of $G$-spectra from $\Sigma^{\infty}\Ell_{\infty}$ to the sphere spectrum $S^{0}$:
$$\Sigma^{\infty}\Ell_{\infty}\overset{\pi_{d_{0}}}{\to} X_{d_{0}-1}\overset{\pi_{d_{0}-1}}{\to}\ldots\overset{\pi_{2}}{\to}X_{1}\overset{\pi_{1}}{\to} S^{0}.$$
Further, the map $\Sigma^{\infty}\Ell_{\infty}\to S^{0}$ comes from the projection $\Ell\to\text{pt.}$ and the homotopy cofibre of the map $\pi_{k}$ is the suspension of the below Thom spectrum; equivalently in the stable category the homotopy fibre of $\pi_{k}$ is the below spectrum:
$$\Sigma^{\infty}G_{k}(V_{0})^{\Hom(T,V_{1}-V_{0})\oplus s(T)}.$$
\end{thm}

The word natural here means that all steps taken in the proof are natural enough to preserve the group action and associated structures. This gives us something similar and yet different from Miller's result. The same pieces of information as in the splitting, the Thom spectra, appear. However, they occur in a slightly different form to how they do in the non-equivariant case. The majority of this document is dedicated to constructing this tower and proving this theorem.

Following on from this, we choose to explore certain topological properties of the theorem. Our goal is to document a
complete picture of the general geometry and topology surrounding the tower, which we achieve through explicit statements regarding the maps and homotopies involved.

Finally, we consider two follow-up conjectures to our work. Firstly, if $V_{0}$ is a subrepresentation of $V_{1}$ then Miller's result does follow, so we explore our theorem in this setting:

\begin{conjecture}\label{introsplits} In the case where $V_{0}$ is a subrepresentation of $V_{1}$ the tower of Theorem \ref{introresult} suitably degenerates to produce an equivariant stable splitting which interacts well with the geometry and topology used in Miller's work.
\end{conjecture}

This claim is still unproven, however, we detail evidence strongly suggesting that the claim is correct. Further, we can prove the result in the special case where $d_{0}=2$. Finally, as Kitchloo demonstrated the Miller splitting in cohomology, so too do we wish to demonstrate that our tower's properties transfer across to equivariant cohomology. We work in reduced complex equivariant topological $K$-theory as the theory is highly calculable.

We apply $K$-theory to various portions of the tower. We note here that a calculation over the whole tower would be an extremely in-depth exercise, however, the local calculations at least suggest a way forward. The results imply the following claim, which we conjecture in more detail in Chapter \ref{ch:ch8}. If this claim were to be true then we would retrieve a suitable equivariant extension to the splitting exhibited by Kitchloo:

\begin{conjecture}\label{introconjecture} The equivariant $K$-theory of the tower resembles a Koszul complex,
with the differentials running between the equivariant $K$-theory of each $\Sigma^{\infty}G_{k}(V_{0})^{\Hom(T,V_{1}-V_{0})\oplus s(T)}$. Moreover, the equivariant $K$-theory of each $\Sigma^{\infty}G_{k}(V_{0})^{\Hom(T,V_{1}-V_{0})\oplus s(T)}$ is an exterior power
of some exterior algebra over the representation ring of $G$. Finally if $V_{0}$ is a subrepresentation of $V_{1}$ the differentials are zero and the equivariant $K$-theory of $\Ell$ is an exterior algebra.
\end{conjecture}

We conclude with a remark that we also believe that this result holds for other equivariant cohomology theories. We limit ourselves to working in $K$-theory, however, as it offers us a lot of scope for direct calculations and `brute force methods'. Moreover, in the equivariant world the concept of `any' equivariant cohomology theory is hard to pin down, as for truly universal statements in cohomology one needs to make certain restrictions and assumptions about the group involved. Studying this claim in more generality and for other cohomology theories is a future goal of the author.

\section{Chapter Layouts}\label{chapterlayouts}

The work proceeds in the following way. Firstly we conclude Chapter \ref{ch:ch1} with two sections on notation, conventions and technicalities. Section \ref{conventions} is a centralized reference point for much of our notational and other conventions. Further, we finish the chapter with a technical section compiling together many standard continuity and properness results that we will need. This section, \ref{NotesonContinuityandProperness}, is not original.

Chapter \ref{ch:ch2} is our first background chapter, all of the work is previously known. Section \ref{onthecategories} details the categories of $G$-spaces and $G$-spectra that we will work in. In particular, this section discusses the equivariant suspension spectrum functor and notes how certain space-level constructions behave when transferred to spectra. Section \ref{BasicTheory} then covers the basic theory of homotopy cofibre sequences, the topological construction we will be using as part of our main result. This is extended in Section \ref{NDRPairs} where we detail a method of building cofibre sequences via NDR pairs, another topological construction which we also define. We then spend a section detailing how cofibre sequences interact with quotients - the results in Section \ref{TotalCofibres} will allow us to simplify the proof of Theorem \ref{introresult} by taking a quotient at a certain point. This section also includes a brief exposition of the theory of cubical diagrams and total cofibres. Finally in Section \ref{SplittingsviaCofibreSequences} we cover how to retrieve stable splittings from cofibre sequences with certain properties; this theory is used in the proof of the Miller splitting and we will use it to conjecture certain results in Chapter \ref{ch:ch7}.

Chapter \ref{ch:ch3} is our final background chapter. This chapter deals with the functional calculus, a theory from functional analysis. The standard theory is found in Section \ref{TheStandardTheory}, none of this chapter is original. In Section \ref{Variations}, however, we expand and generalize the functional calculus to a more general context; this theory is prevalent throughout the rest of the document. We then use the expanded theory in Section \ref{BuildingNDRPairsUsingFunctionalCalculus} to build various NDR pairs and cofibre sequences which will be used in the proof of the main theorem. Finally, in Section \ref{TheHomotopyTypeofCertainMapsinFunctionalCalculus} we cover the homotopy classification of maps built from our functional calculus extension, proving a nice classification result which determines when two maps of this type are homotopic.

We state the main theorem in Chapter \ref{ch:ch4}. Section \ref{TheTower} details a statement of the main result as well as defining sets that will become the spectra in our tower. These sets are then topologized in Section \ref{TopologizingTheTower} - the interesting topology they hold will be used in many continuity arguments. This section also explicitly defines the spectra in our tower. We then take a detour into the topology and equivariance of bundles over Grassmannians in Section \ref{TheTopologyOfBundlesOverGrassmannians}. This section is unoriginal but we gather together the results to be complete and to be rigourous when defining certain $G$-spaces that will build the $G$-spectra that we claim are the cofibres of the tower. We construct the spectra from the spaces in Section \ref{DestabilizingtheTower}. Finally, in Section \ref{CandidateMaps} we write down unstable maps which will stabilize to become the maps in our tower. We do not prove here, however, that these maps are well-defined - this will be exhibited for one map in Chapter \ref{ch:ch5} and proved for the other map in Chapter \ref{ch:ch6}.

We prove the meat of the result - that the cofibres of the maps in our tower are as claimed - in Chapter \ref{ch:ch5}. This is done firstly for the top of the tower in Section \ref{TheTopTriangle} by proving that the top of the tower is isomorphic to a cofibre sequence built in Chapter \ref{ch:ch3}. We then explain a method of generalization in Section \ref{GeneralizingtheTopoftheTower} before using this to prove the result for the whole tower in Section \ref{TheOtherTriangles} - in this section we simplify the proof somewhat by taking a global quotient as discussed in \ref{TotalCofibres}.

Chapter \ref{ch:ch6} then covers further topological properties of the tower. Section \ref{BypassingtheQuotientbyaLift} remedies the problem of taking a quotient by giving explicit lifts of the maps involved - this section demonstrates that one of the maps in the general cofibre sequence is as described in Section \ref{CandidateMaps}. Section \ref{BypassingtheQuotientbyaLift} completes the work on defining explicit unstable maps to build the tower of Theorem \ref{introresult}, we then spend Section \ref{ExplicitNullHomotopies} explicitly writing down the null-homotopies of each composition in each cofibre sequence. Finally the bottom portion of our tower holds nice extra geometric and topological properties, this is explored in Section \ref{TheBottomoftheTower}.

Chapter \ref{ch:ch7}, our first conjectural chapter, covers the special case of our work when the Miller splitting can also be built. Section \ref{TheConjecture} conjectures how our spectra and the spectra in the Miller filtration are related, providing evidence suggesting that our claims are true. The consequences of this conjecture are explored in Section \ref{ConjectureFollowup} to derive Conjecture \ref{introsplits}, further the evidence provided in Section \ref{TheConjecture} allows us to prove Conjecture \ref{introsplits} in the dimension $2$ special case. We then state an equivalent conjecture in Section \ref{AnEquivalentConjecture}; again we provide evidence that this conjecture is true. That all these conjectures are equivalent is then proved in Section \ref{FollowuptotheSecondConjecture}.

Our final chapter, Chapter \ref{ch:ch8}, conjectures details about the equivariant $K$-theory of the tower built in Theorem \ref{introresult}. Section \ref{BasicResults} covers the basic results in equivariant $K$-theory that we will need, this section is unoriginal. Section \ref{TheKTheoryoftheBottomoftheTower} then applies this theory to partially calculate the $K$-theory of the bottom of the tower. Finally, we conjecture in Section \ref{TheKTheoryoftheTowerConjecture} a more detailed conjecture similar to Conjecture \ref{introconjecture}. To do this we use the evidence of the previous section, while also noting the obstructions that prevent the general calculation following through directly from the results concerning the bottom of the tower.

Finally, we present as an appendix a proof of the Miller splitting. Appendix \ref{ch:appendixa} is unoriginal, but we include the work to establish our conventions and to allow easy referral to the various maps and spaces involved in the proof of the Miller splitting.

\section{Conventions}\label{conventions}

We conclude this chapter with two sections on notational conventions and continuity notes that we will use throughout the document. This section supplements the notation and conventions already laid out in the proceeding sections and is gathered together to act as a centralized reference point for the reader. Firstly we note a stylistic tendency. When defining certain spaces we will occasionally use a dash as a notational convention for the unbased space, for example in defining certain based spaces $\tilde{X}_{k}$ in Chapter \ref{ch:ch4} we first define unbased spaces $\tilde{X}'_{k}$ and set $\tilde{X}_{k}:=(\tilde{X}'_{k})_{\infty}$. This notational style occurs throughout the document.

On equivariance, the one notational convention we take is writing $U\leqslant V$ to denote that $U$ is a subrepresentation of a representation $V$. We also use this notation in the non-equivariant case to denote vector subspace, we rely on context to determine which meaning we mean at any given time. 

Regarding suspensions, we use the standard notation $\Sigma$ for reduced suspension. We want to be consistent, however, on how $\Sigma$ is parameterized. As a standard we will be taking the `coordinate' in $\Sigma$ to lie in $\Real$ - we will explicitly mention when this is not the case. Regarding maps involving suspensions, we use the notation below for maps of the form $f:X\to\Sigma Y$:
$$\xymatrix{f:X\ar[r]|\bigcirc &Y.}$$

We next explicitly define two particular subspaces of homomorphisms that will occur throughout: 

\begin{defn}\label{injandnoninj} Let $V$ and $W$ be vector spaces. Define:
$$\inj(V,W):=\{\gamma\in\Hom(V,W):\gamma\text{ is injective}\},$$
$$\inj(V,W)^{c}:=\{\gamma\in\Hom(V,W):\gamma\text{ is not injective}\}.$$
\end{defn}

At various points we will both implicitly and explicitly use choices of homeomorphisms between $\Real$, $(0,\infty)$ and
$(0,1)$. We always take $\Real\cong (0,\infty)$ via the exponential map which we denote by both $\Exp(x)$ and $e^{x}$ depending on context; we use $e$ for the exponential of a number but tend to use $\Exp$ when we mean the exponential of a matrix or similar. In both cases $\log$ is our standard notation for the inverse. We also make
the below choice:

\begin{defn}\label{UnitIntervalHomeomorphism} Whenever we consider $(0,1)\cong \Real$
we do so via the function $\log(x/(1-x))$.
\end{defn}

Finally, we make a single remark on our conventions concerning categories, suspending further discussion until Section \ref{onthecategories}:

\begin{rem}\label{ignoringsigmainfty} For a space $X$ we use the notation $X$ not only to refer to the space, but also to the suspension spectrum $\Sigma^{\infty} X$. We rely on context to determine whether we mean the space or spectrum at any one time.
\end{rem}

\section{Notes on Continuity and Properness}\label{NotesonContinuityandProperness}

As we are mostly dealing with compactified spaces we need to be careful regarding a lot of continuity arguments. In particular the properness of certain maps tends to come into play. Hence we conclude Chapter \ref{ch:ch1} with a detour covering various standard results from general topology that will be used throughout the document. The first lemma is standard:

\begin{lem}\label{propernessequalsbasepoints} Let $f:X\to Y$ be a continuous map. Then $f$ has a continuous extension $f_{\infty}:X_{\infty}\to Y_{\infty}$ if and only if $f$ is proper.
\end{lem}

This indicates the importance of proper maps in the based case. A useful tool is the following lemma:

\begin{lem}\label{propercompositionisproper} Let $X$, $Y$ and $Z$ be locally compact Hausdorff spaces, $f:X\to Y$ and $g:Y\to Z$. If $g\circ f$ is proper then $f$ is proper.
\end{lem}
\begin{proof} Let $U$ be a compact subset of $Y$, we wish to show that $f^{-1}(U)$ is compact. The space $g(U)$ is compact as it is the continuous image of a compact space. The map $g\circ f$ is proper, thus $(g\circ f)^{-1}(g(U))$ is a compact subset of $X$. This, however, is $f^{-1}(g^{-1}(g(U)))$. The space $f^{-1}(U)$ is a subspace of this, we claim that $f^{-1}(U)$ is closed and thus compact as it is the closed subspace of a compact space. This is true as $f$ is continuous and $U$ is a compact subspace of a Hausdorff space, and thus closed. Hence $f^{-1}(U)$ is closed and hence compact as required.
\end{proof}

As demonstrated above, being closed is intrinsically linked to being compact for Hausdorff spaces. Hence knowledge of a map being closed is useful in certain properness arguments. The below result, Proposition $8.2$ in Chapter $1$ of
\cite{BredonTopologyandGeometry}, is useful for constructing properness arguments involving Cartesian products when combined with Tychonoff's Theorem:

\begin{lem}\label{closedprojifcompact} The projection map $\pi_{0}:X\times Y\to X$ is closed if $Y$ is compact.
\end{lem}

We now cover a nice general result which allows us to equip a space with two seemingly different equivalent topologies - this will prove useful in many of the later continuity arguments.

\begin{lem}\label{whensubspaceisquotient} Let $X$ and $Z$ be locally compact and Hausdorff and let $f:X\to Z$
be a continuous proper map. Setting $Y:=f(X)$, we have an obvious inclusion of sets $j:Y\rightarrowtail Z$ and surjection $p:X\twoheadrightarrow Y$. Moreover, as $f$ is a proper map of locally compact Hausdorff spaces we have
the continuous extension $f_{\infty}:X_{\infty}\to Z_{\infty}$. Setting $Y_{\infty}:=f_{\infty}(X_{\infty})$, we again have an obvious inclusion of sets $j_{\infty}:Y_{\infty}\rightarrowtail Z_{\infty}$ and surjection
$p:X_{\infty}\twoheadrightarrow Y_{\infty}$. These fit into the below diagram of sets:
$$\xymatrix{X\ar@{->>}[r]^{p}\ar@{ >->}[d]_{i_{X}}&Y\ar@{ >->}[r]^{j}\ar@{ >->}[d]^{i_{Y}}&Z\ar@{ >->}[d]^{i_{Z}}\\
X_{\infty}\ar@{->>}[r]_{p_{\infty}}&Y_{\infty}\ar@{ >->}[r]_{j_{\infty}}&Z_{\infty}}$$
Then there are unique topologies on $Y$ and $Y_{\infty}$ such that:
\begin{enumerate}
\item $Y$ is locally compact and Hausdorff and $Y_{\infty}$ is its one-point compactification.
\item $p$ is a proper quotient map.
\item $j$ is a proper closed inclusion.
\item $p_{\infty}$ is a quotient map.
\item $j_{\infty}$ is a closed inclusion.
\item $i_{Y}$ is an open inclusion.
\end{enumerate}
\end{lem}

\begin{proof} We first prove that the quotient topology on $Y_{\infty}$ from $p_{\infty}$ and the subspace topology on $Y_{\infty}$ from $j_{\infty}$ are equivalent. Moreover, in doing this we will demonstrate that $j_{\infty}$ is closed under this topology. Let $U\subseteq Y_{\infty}$ be closed in the quotient topology. Then $p_{\infty}^{-1}(U)$ is a closed subspace of the compact space $X_{\infty}$ and hence compact. Further, $f_{\infty}(p_{\infty}^{-1}(U))$ is the continuous image of a compact set and hence compact, hence closed in $Z_{\infty}$ as $Z_{\infty}$ is a Hausdorff space. As $f_{\infty}$ is surjective onto
$Y_{\infty}$, the above remarks imply that $j_{\infty}(U)$ is closed and thus $U$ is closed in the subspace topology.

Now let $U$ be open in the subspace topology on $Y_{\infty}$. By assumption $U=j_{\infty}^{-1}(V)$ for some open subset $V$ of $Z_{\infty}$. One can observe that $p_{\infty}^{-1}(U)=f^{-1}_{\infty}(V)$ which is open as $f$ is continuous. Thus $U$ is open in the quotient topology. We have thus proved that the quotient and subspace topologies on $Y_{\infty}$ are the same and that $j_{\infty}$ is closed.

We now equip $Y$ with three topologies - the subspace of $Y_{\infty}$ topology, the subspace of $Z$ topology and the
quotient of $X$ topology. We prove that the three topologies are equivalent. Let $\tau_{q}$ be the quotient topology, $\tau_{s}$ the subspace of $Z$ topology and $\tau_{\infty}$ the subspace of $Y_{\infty}$ topology. Similarly, let $\tau_{q}^{\infty}$ and $\tau_{s}^{\infty}$ be the quotient and subspace topologies on $Y_{\infty}$. Let $U\in \tau_{\infty}$, then $U=U_{\infty}\cap Y$ for some $U_{\infty}$ open in $Y_{\infty}$. Then $U_{\infty}\in
\tau_{q}^{\infty}$ and thus $p_{\infty}^{-1}(U_{\infty})$ is open in $X_{\infty}$. As a set, it is clear that $p^{-1}_{\infty}(U_{\infty})\cap X=p^{-1}(U)$. Moreover, $p^{-1}_{\infty}(U_{\infty})\cap X$ is easily observed to be open in $X$, hence $U\in \tau_{q}$.

Similarly, we also note that $U_{\infty}\in\tau^{\infty}_{s}$ and thus there is a $V_{\infty}$ open in $Z_{\infty}$ such that $V_{\infty}\cap Y_{\infty}=U_{\infty}$. Setting $V:=V_{\infty}\cap Z$ it is clear to see that as sets $V\cap Y=U$ and that $V$ is open in $Z$. Thus $U\in \tau_{s}$.

We now wish to show that if $U\in\tau_{q}$ then $U\in\tau_{\infty}$. As $U\in\tau_{q}$ then $p^{-1}(U)$ is open in $X$ and hence open in $X_{\infty}$. This is enough to show that $U\in\tau_{\infty}$. Similarly, if $U\in\tau_{s}$ then there exists $V$ open in $Z$ such that $V\cap Y=U$. That $V$ is open in $Z$, however, implies that $V$ is open in $Z_{\infty}$ and from here it is simple to see that $U\in\tau_{\infty}$. This proves the equality of the three topologies.

We now equip $Y$ with this topology.With it $Y$ is the quotient of a locally compact space and the subspace of a Hausdorff space and hence $Y$ is locally compact and Hausdorff. That $Y_{\infty}$ is the one-point compactification of $Y$ is easy to see from the equivalence of all the topologies. The claimed properties then all follow.
\end{proof}

Finally, many of the spaces we work with are topologized via a norm. We would like to get a handle on what compactness means in this case, which we do with the below standard lemma:

\begin{lem}\label{extendedheineborel} Let $V$ be a finite dimensional normed space. Then a subset $U$ is compact if and only if it is closed and bounded.
\end{lem}
\begin{proof} We first note that via the Heine-Borel Theorem the closed unit ball of a finite dimensional normed vector space is compact. Thus via noting that scaling is a continuous self-map we have that any closed ball in $V$ is compact. A bounded set by definition is a set contained within a closed ball, so a closed and bounded set $U$ will be a closed subset of a compact closed ball. $V$ is Hausdorff as standard, hence $U$ is compact. Conversely, let $U$ be compact. By the Heine-Borel Theorem $U$ is complete and totally bounded but this implies the result as completeness implies that $U$ is closed and $U$ being totally bounded implies that $U$ is bounded.
\end{proof}

\chapter{Cofibre Sequences and NDR Pairs}
\label{ch:ch2}

\section{On Categories of Spaces and Spectra}\label{onthecategories}

We open this chapter by remarking briefly on the categories we choose to work in. Firstly when we say space, we actually mean $G$-equivariant compactly generated weak Hausdorff space. We let $\Map(X,Y)$ denote the space of maps from $X$ to $Y$ equipped with the compact-open topology and group action given by conjugation, $g.f(x):=gf(g^{-1}x)$. One can also consider a based version of this theory by requiring $G$-fixed basepoints. We use $CGWH$ to denote the category with objects based $G$-equivariant compactly generated weak Hausdorff spaces and morphism sets $\Map(X,Y)$. The non-equivariant version of this category is built from a modification of the category $CGH$ of compactly generated
Hausdorff spaces first introduced in \cite{SteenrodCategory}; this modification was originally made by McCord in Section $2$ of \cite{McCordClassifyingSpacesAndSymmetricProducts}. This category is the standard `space-level' category, it is used as such in both \cite{LewisMaySteinberger} and \cite{mayetal} and as noted in \cite{Hovey1} it has a model structure and hence we can take the homotopy category. We use this category to exploit its rich structure, $CGWH$ is a Cartesian closed category. General information on Cartesian closed categories can be found in Chapter $4$ Section $6$ of \cite{maclane}. In particular we have the below adjunction:
$$\Map(X\wedge Y,Z)\cong \Map (X,\Map(Y,Z)).$$ 

This bijection will be used often in continuity arguments - we will demonstrate that a map is continuous by demonstrating that what would be its adjoint is continuous. How the adjoint works in this particular case is described in great detail in \cite{SteenrodCategory}.

We finally note one other space category we work in. We denote by $GCW$ the subcategory of $CGWH$ whose objects are based finite $G$-$CW$-complexes. We make this distinction in certain places to ease the transition from $G$-space to $G$-spectrum. We finally note here that $GCW$ has the standard model structure and hence we can consider the associated homotopy category.

Let $\mathcal{U}$ be a choice of complete $G$-universe. We now define the homotopy category $\mathcal{F}_{G}$ of finite $G$-$CW$-spectra indexed over $\mathcal{U}$. Objects of $\mathcal{F}_{G}$ consist of expressions of the form $\Sigma^{-U}X$ for $U$ a finite dimensional subrepresentation of $\mathcal{U}$ and $X$ a based finite $G$-$CW$-complex. We then define morphisms as below, recalling $[A,B]_{G}$ to denote based $G$-homotopy classes of maps from $A$ to $B$ and recalling that for $U$ a subrepresentation of $W$ we have an honest representation $W\ominus U$:
$$\mathcal{F}_{G}(\Sigma^{-U} X,\Sigma^{-V}Y)=\lim_{\overset{\longrightarrow}{U,V\leqslant W\leqslant\mathcal{U},\dim(W)<\infty}}\left[\Sigma^{W\ominus U} X,\Sigma^{W\ominus V}Y\right]_{G}.$$

This gives us a well-defined category of finite $G$-$CW$-spectra to work from which is defined relatively easily due to the finiteness conditions. Moreover, this category has a well-defined smash product given by $\Sigma^{-U} X\wedge \Sigma^{-V}Y:=\Sigma^{-(U\oplus V)}X\wedge Y$. There is also a suspension spectrum functor given in the obvious way that allows us to pass from $G$-spaces to $G$-spectra in this case:
$$\Sigma^{\infty}:\text{Ho}(GCW)\to\mathcal{F}_{G}.$$

As the category $\mathcal{F}_{G}$ has morphisms arising from homotopy classes any constructions that apply on the space-level homotopy category also exist in $\mathcal{F}_{G}$. Moreover one can transfer a homotopy construction to $\mathcal{F}_{G}$ via $\Sigma^{\infty}$. In particular the homotopy-invariant constructions built throughout this chapter also appear in $\mathcal{F}_{G}$; there is a well-defined notion of cofibre sequence in this category. Finally, we remark that we could have chosen to work in the more detailed categories of general $G$-spectra, the category in \cite{LewisMaySteinberger} or one of the more highly structured categories, for example the category of EKMM spectra or the category of orthogonal spectra. We do not do so, however, for simplicities sake, instead noting here that $\mathcal{F}_{G}$ is a subcategory of the homotopy category of any reasonable category of $G$-spectra. Moreover, for a chosen $G$-universe upon which to index, all the stable categories are equivalent. The stable category has cofibre sequences which are constructed by the same method as we take in this chapter for spaces, this is noted at the beginning of Chapter $1$ Section $6$ of \cite{LewisMaySteinberger}.
Moreover, cofibre sequences and fibre sequences dualize in the stable category, as noted in Chapter $3$ Section $2$ of \cite{LewisMaySteinberger} and generate the distinguished triangles of the equivariant stable homotopy category. That this category is triangulated was explicitly proved as Theorem $9.4.3$ in \cite{HoveyPalmieriStrickland}.

\section{The Basic Theory of Cofibre Sequences}\label{BasicTheory}

We now cover some of the standard theory of cofibre sequences (occasionally referred to in the literature as Puppe or cofibration sequences). The results are standard and can be found in many basic sources, for example there is a detailed account of the space-level theory in Chapter $8$ of \cite{ConciseMay}. We include the detail to achieve consistency in our conventions and for completeness. Most of the theory we give in this chapter is for spaces, but can equally be applied in the categories $CGWH$, $GCW$ or $\mathcal{F}_{G}$ from Section \ref{onthecategories}.

Let $f:X\to Y$ be a map of based spaces and consider the unit interval $[0,1]$ as a based space with basepoint $\{0\}$.

\begin{defn}\label{cofibre} Let $C(X):=[0,1]\wedge X$ be the cone on $X$. The mapping cone or (homotopy) cofibre of $f$ is given by $C_{f}:=C(X)\vee Y/\sim$ where $\sim$ is the equivalence relation identifying $(1,x)$ with $f(x)$. We also use the notation $C(X)\cup_{f} Y$ for $C_{f}$.
\end{defn}

\begin{rem}\label{differentbasepointstothecitedsource} We remark here that in the cited text \cite{ConciseMay} the basepoint of $[0,1]$ is taken to be at $\{1\}$ rather than $\{0\}$; our choice only makes aesthetic changes to the theory.  
\end{rem}

\begin{rem}\label{cofibresdifferentcategories} Although we used spaces in the above definition we could equally use $G$-spaces and $G$-maps, equipping $[0,1]$ with the trivial action. Using this set-up all of the following theory applies in $CGWH$ or $GCW$. We can also make this definition in $\mathcal{F}_{G}$ but we have to be careful as our morphisms in this category are in some sense homotopy classes so we also need certain homotopy invariance properties which we cover below.
\end{rem}

The main aim of this document is to construct certain cofibre sequences. A cofibre sequence is going to be a sequence of spaces of the form $X\to Y\to C_{f}\to\ldots$. In order to build them we need the following lemma. We defer notes on the proof as it will follow from a later stated result but we choose to present this result first as we feel it provides more motivation for the later topological constructions. For the below statement we take suspension to have coordinate values in $(0,1)$ using the homeomorphism \ref{UnitIntervalHomeomorphism}.

\begin{lem}\label{cofibreslemma} Let $f:X\to Y$ be a map of based spaces, and let $i:Y\to C_{f}$ be the obvious inclusion map. Then the map $\beta:C_{i}=C(Y)\cup_{i} C_{f}:\to \Sigma X$ collapsing $C_{f}$ to $\Sigma X$ and collapsing $C(Y)$ to the basepoint is a homotopy equivalence $\beta:C_{i}\simeq \Sigma X$.
\end{lem}

The above lemma suggests that we should be able to construct sequences of the following form:
$$X\overset{f}{\to} Y\overset{i_{f}}{\to} C_{f}\to\Sigma X\to \Sigma Y\to\ldots.$$

To do this we can extend the sequence $X\overset{f}{\to} Y\overset{i_{f}}{\to} C_{f}\to\ldots$ in the obvious way:
$$X\overset{f}{\to} Y\overset{i_{f}}{\to} C_{f}\overset{i_{i_{f}}}{\to}C_{i_{f}}\to C_{i_{i_{f}}}\to\ldots.$$

We have the above lemma giving homotopy equivalences $C_{i_{f}}\simeq\Sigma X$, $C_{i_{i_{f}}}\simeq\Sigma Y$ and similar. However, if we wish to match the two sequences up via these homotopy equivalences we need to insert a twist into the sequence. This can occur in different places, we will take the twist as indicated by the following proposition.

\begin{defn}\label{thecofibrestwist} Let $f:X\to Y$ be a map of based spaces. Taking suspension coordinates in $(0,1)$, let $-\Sigma f:\Sigma X\to \Sigma Y$ be the map given by $(-\Sigma f)(t\wedge x)=(1-t)\wedge f(x)$.
\end{defn}

\begin{prop}\label{howtobuildacofibseq} For any based map $f:X\to Y$ the following diagram commutes up to homotopy:
$$\xymatrix
{X\ar[r]^{f}&Y\ar[r]^{i_{f}}&C_{f}\ar[dr]_{d_{f}}\ar[r]^{i_{i_{f}}}&C_{i_{f}}\ar[d]^{\beta_{X}}\ar[dr]^{d_{i_{f}}}\ar[r]^{i_{i_{i_{f}}}}&C_{i_{i_{f}}\ar[d]^{\beta_{Y}}}\\
&&&\Sigma X\ar[r]_{-\Sigma f}&\Sigma Y}
$$
Here $\beta_{X}$ and $\beta_{Y}$ are the homotopy equivalences from Lemma \ref{cofibreslemma} and the maps $d_{f}$ and $d_{i_{f}}$ are the obvious collapse maps.
\end{prop}
\begin{proof} It is clear that the left-hand triangle and right-hand triangle strictly commute. Thus the only issue is showing that the middle triangle commutes up to homotopy, i.e. that $d_{i_{f}}\simeq -\Sigma f\circ \beta_{X}$. These two maps factor as below:
$$d_{i_{f}}:C_{i_{f}}= C(Y)\cup_{i_{f}}(C_{f})\overset{coll.}{\to} \frac{C_{i_{f}}}{C_{f}}\cong \Sigma Y$$
$$-\Sigma f\circ \beta_{X}:C_{i_{f}}=C(Y)\cup_{i_{f}}(C(X)\cup_{f} Y)\overset{coll.}{\to}\frac{C_{i_{f}}}{C(Y)}\cong \Sigma X\overset{-\Sigma f}{\to}\Sigma Y.$$

It is easy to observe that $-\Sigma f\circ \beta_{X}$ is equivalent to applying $-Cf$ to $C(X)$ before collapsing out. As $C(X)$ is collapsed out, precomposing by $-Cf$ on $C(X)$ does not affect the map $d_{i_{f}}$. Thus we can consider the space $Z$ given by applying $-Cf$ to $C_{i_{f}}$ - it is trivial to see that $Z$ is homeomorphic to the space $[0,2]\times Y/\sim$ where here $\sim$ quotients out the two end copies of $Y$ and the copy of $[0,2]$ affixed to the basepoint. We note that the twist in the map $-Cf$ is needed to do this. Thus finding a homotopy between $d_{i_{f}}$ and $-\Sigma f\circ \beta_{X}$ is equivalent to finding a homotopy between two maps $Z\to \Sigma Y$, the first collapsing out $[1,2]\times Y$ and the second collapsing out $[0,1]\times Y$. This homotopy, however, is simply the homotopy that slides the copy of the product of the unit interval with $Y$ that will remain along the interval $[0,2]$,
starting at $[0,1]$ and ending at $[1,2]$. This then produces a homotopy between $d_{i_{f}}$ and $-\Sigma f\circ \beta_{X}$ as required.
\end{proof}

This gives us the following sequence:
$$X\overset{f}{\to} Y\overset{i_{f}}{\to}
C_{f}\overset{d_{f}}{\to}\Sigma X\overset{-\Sigma f}{\to} \Sigma
Y\overset{-\Sigma i_{f}}{\to}\Sigma C_{f}\overset{-\Sigma
d_{f}}{\to}\Sigma^{2}X\overset{\Sigma^{2}f}{\to}\Sigma^{2}
Y\overset{\Sigma^{2}i_{f}}{\to}\ldots.$$

Note that we only need to consider the $X\overset{f}{\to} Y\overset{i_{f}}{\to} C_{f}\overset{d_{f}}{\to}\Sigma X$ portion of the sequence as it contains all the necessary information. We can also denote a sequence as a triangle:
$$\xymatrix{Y\ar[d]& X\ar[l]_{f}\\
C_{f}\ar[ur]|\bigcirc}$$

The next definition covers what we mean when we say something `is' a cofibre sequence; the concept is fairly clear on the space level but importantly it implies that cofibre sequences are actually constructions in the homotopy category. Hence the analogous construction in $\mathcal{F}_{G}$ is well-defined. Moreover, this definition makes it clear that a cofibre sequence of $CW$-complexes is a cofibre sequence in $\mathcal{F}_{G}$ after one applies $\Sigma^{\infty}$. Also, all the constructions and modifications of cofibre sequences proved for spaces throughout the rest of the document hold in $\mathcal{F}_{G}$ by simply replacing the word `space' with `spectrum'; this definition of isomorphic sequences allows this. Finally, we note that this definition also hints to the fact that in the stable category cofibre sequences are distinguished; the definition resembles one of the axioms of a triangulated category.

\begin{defn}\label{isomorphistoacofibreseq} We say $X\to Y\to Z\to \Sigma X$ is isomorphic to a cofibre sequence $A\overset{f}{\to} B\to C_{f}\to\Sigma A$ if we have the following homotopy commutative diagram, where the downward maps are all homotopy equivalences: 
$$\xymatrix{X\ar[d]_{\sim}\ar[r]&Y\ar[d]_{\sim}\ar[r]&Z\ar[d]_{\sim}\ar[r]&\Sigma X\ar[d]_{\sim}\\
A\ar[r]_{f}& B\ar[r]& C_{f}\ar[r]& \Sigma A}$$
\end{defn}

\begin{lem}\label{rotatingcofibresequences} Using the above definition, Lemma \ref{cofibreslemma} and Proposition
\ref{howtobuildacofibseq} we note that if $X\to Y\to Z\to\Sigma X$ is a cofibre sequence, then so is any portion of the long sequence $X\to Y\to Z\to\Sigma X\to\ldots$. For example if $f:X\to Y$ then $Y\overset{i_{f}}{\to} C_{f}\overset{d_{f}}{\to}\Sigma X\overset{-\Sigma f}{\to} \Sigma Y$ is a cofibre sequence.
\end{lem}

We also add two lemmas to give us a little more flexibility in building cofibre sequences.

\begin{lem}\label{smashingalongacofibseq} Let $X\to Y\to Z\to \Sigma X$ be a cofibre sequence. Then $X\wedge A\to Y\wedge A\to Z\wedge A\to \Sigma X\wedge A$ is a cofibre sequence.
\end{lem}
\begin{proof} We have a map $f:X\to Y$ such that $Z\simeq C_{f}$. Consider the map $f\wedge 1:X\wedge A\to Y\wedge A$. The below sequence of equivalences is then sufficient to prove the result:
\begin{align*} C_{f\wedge 1} &= \frac{([0,1]\times X)\wedge A}{\sim}\cup_{f\wedge 1} (Y\wedge A) \\
&= \frac{[0,1]\times X}{\sim}\wedge A\cup_{f\wedge 1} (Y\wedge A)\\
&= \left(\frac{[0,1]\times X}{\sim}\cup_{f} Y\right)\wedge A\\
&= C_{f}\wedge A\\
&\simeq Z\wedge A.
\end{align*}
\end{proof}

\begin{lem}\label{cofibseqfibrebundles} Let $A$ be a space and let $X_{a}$, $Y_{a}$ and $Z_{a}$ be families of based spaces parameterized over all $a\in A$ equipped with the following structure:
\begin{itemize}
\item Total spaces $X:=\bigcup_{a\in A} X_{a}$, $Y:=\bigcup_{a\in A} Y_{a}$ and $Z:=\bigcup_{a\in A} Z_{a}$.
\item Projections $X\overset{\pi_{1}}{\to} A$, $Y\overset{\pi_{2}}{\to} A$ and $Z\overset{\pi_{3}}{\to} A$ that send points in each $X_{a}$, $Y_{a}$ and $Z_{a}$ to $a$.
\item Sections $A\overset{\sigma_{1}}{\to} X$, $A\overset{\sigma_{2}}{\to} Y$ and $A\overset{\sigma_{3}}{\to} Z$ sending $a$ to the basepoints in $X_{a}$, $Y_{a}$ and $Z_{a}$ respectively.
\end{itemize}
Let $\Sigma_{A}X:= \bigcup_{a\in A}\Sigma X_{a}$ be the union of all spaces $\Sigma X_{a}$. Suppose also that we have a sequence of continuous maps $X\overset{f}{\to}Y\overset{g}{\to} Z\overset{h}{\to}\Sigma_{A} X$ arising from fibrewise sequences $X_{a}\overset{f_{a}}{\to} Y_{a}\overset{g_{a}}{\to} Z_{a}\overset{h_{a}}{\to} \Sigma X_{a}$. Moreover, assume that each fibrewise sequence is a cofibre sequence. Then we can add basepoints via our sections, giving us the following based sequence:
$$X/\sigma_{1}(A)\to Y/\sigma_{2}(A)\to Z/\sigma_{3}(A)\to \Sigma X/\sigma_{1}(A).$$

This is a cofibre sequence.
\end{lem}
\begin{proof} Firstly note that $\Sigma_{A}X$ contains a copy of $A$ via the section $\sigma$ sending $a$ to the basepoint of $\Sigma X_{a}$. From here it is pretty simple to observe that $\Sigma X/\sigma_{1}(A)\cong\Sigma_{A}X/\sigma(A)$ and thus we can take sections and build the based sequence. That it's a cofibre sequence will then follow from demonstrating that if $f':X/\sigma_{1}(A)\to Y/\sigma_{2}(A)$
then $C_{f'}\cong Z/\sigma_{3}(A)$. This, however, follows from the following set of equivalences; the key issues involve noting that the Cartesian product distributes over union and keeping track that the basepoints collapse out as expected:
\begin{align*} C_{f'} &= \frac{[0,1]\times X/\sigma_{1}(A)}{\sim}\cup_{f'} Y/\sigma_{2}(A) \\
&= \frac{[0,1]\times \left(\bigcup_{a\in A}X_{a}\right)/\sigma_{1}(A)}{\sim}\cup_{f'} \left(\bigcup_{a\in A}Y_{a}\right)/\sigma_{2}(A)\\
&\simeq \left(\bigcup_{a\in A}Z_{a}\right)/\sigma_{3}(A)\\
&= Z/\sigma_{3}(A).
\end{align*}
\end{proof}

This gives us enough theory to build various cofibre sequences from a given sequence. It does not, however, give us a reasonable method of building cofibre sequences from scratch, barring first principles. To give us a way to do this, we turn to the theory of NDR pairs.

\section{NDR Pairs}\label{NDRPairs}

One standard method of building cofibre sequences is via neighbourhood deformation retract pairs, or NDR's. We now detail
the method of construction; in Chapter \ref{ch:ch3} we will use this to build an explicit cofibre sequences used in the proof of the main theorem. The definition and results are standard, and can be found in Chapter $6$, Section $4$ of \cite{ConciseMay} for example.

\begin{defn}\label{NDR} Let $X$ be a space and $A$ a closed subspace. We say $A$ is an NDR of $X$ or that $(u,h)$ represents $(X,A)$ as an NDR pair if:
\begin{enumerate}
\item $u:X\to [0,1]$ is continuous.
\item $h:[0,1]\times X\to X$ is continuous.
\item $h_{1}(x)=x$ for all $x\in X$.
\item $h_{t}(a)=a$ for all $t$ and for all $a\in A$.
\item $h_{0}(x)\in A$ for all $x$ such that $u(x)<1$.
\item $u^{-1}(0)=A$.
\end{enumerate}
\end{defn}

\begin{prop}\label{NDRCofibreseq} Let $A\overset{i}{\to} X$ be an NDR pair represented by $(u,h)$. Let $\tilde{r}$ be the following map:
$$\tilde{r}:X\to [0,1]\times X$$
$$x\mapsto(u(x),h_{0}(x)).$$ 

Then quotienting by $A$ throughout gives a continuous map $r:X/A\to C_{i}$. We thus also have the following sequence:
$$A\overset{i}{\to}X\overset{p}{\to}\frac{X}{A}\overset{e}{\to}\Sigma A.$$
Here $p$ is the evident collapse and $e$ is the composition $X/A\overset{r}{\to}C_{i}\overset{d}{\to}\Sigma A$ with $d$ defined to be the standard collapse map. This is a cofibre sequence.
\end{prop}
\begin{proof} Referring to the NDR definition in \ref{NDR}, we note that condition $5$ implies that $\tilde{r}(X)\subseteq [0,1]\times A\cup_{i}
\{1\}\times X$. We also note that $\tilde{r}(A)\subseteq\{0\}\times A$ by conditions $4$ and $6$. This is enough to show that the map $r$ exists as claimed.

We now construct the following homotopy commutative diagram:
$$\xymatrix{A\ar[r]^{i}\ar[d]_{=}&X\ar[r]^{f}\ar[d]_{=}&C_{i}\ar[d]^{q}\ar[r]^{d}&\Sigma A\ar[d]^{=}\\
A\ar[r]_{i}&X\ar[r]_{p}&\frac{X}{A}\ar[r]_{e}&\Sigma A }$$

Here $f$ is the standard inclusion and $q$ is the quotient map $C_{i}\to C_{i}/C(A)\cong X/A$. We claim that $q$ is a
homotopy equivalence with homotopy inverse $r$ - this will then prove the claim as with this fact all squares will trivially commute up to homotopy.

We first show that $q\circ r$ is homotopy equivalent to the identity. Note that the map comes from the map $X\to X$ given by $x\mapsto h_{0}(x)$. Thus by noting that $h_{t}(A)\subseteq A$ by condition $4$ we have a well-defined map $\tilde{h}_{t}:X/A\to X/A$. This gives a homotopy between $h_{0}=q\circ r$ and $h_{1}=\text{id}_{X}$ by condition $3$.

We now prove that $r\circ q$ is homotopy equivalent to the identity by constructing an explicit homotopy. We do this by building a family of maps out of the mapping cylinder, which we define as $M_{i}:=([0,1]\times A)\cup (\{1\}\times X)$ equipped with the appropriate basepoint identifications. We note here that we can collapse $M_{i}$ into $C_{i}$. Now define a map $\tilde{k}:[0,1]\times M_{i}\to [0,1]\times X$ as follows, this will be the basis for our homotopy:
$$
\tilde{k}(s,1,x)=\left\{\begin{array}{ll}
(u(x)+2s,h_{0}(x))& \quad 0\leqslant s\leqslant \frac{1}{2}(1-u(x))\\
(1,h(\frac{2s-1+u(x)}{1+u(x)},x))& \quad
\frac{1}{2}(1-u(x))\leqslant s\leqslant 1
\end{array}
\right.
$$
$$\tilde{k}(s,t,a)=\left\{\begin{array}{ll}(2st,a)& \quad 0\leqslant s\leqslant \frac{1}{2}\\
(t,a)&\quad \frac{1}{2}\leqslant s\leqslant 1. \end{array}\right.$$

We first check that this definition is internally consistent and well-defined. Firstly it is easy to see that
$\tilde{k}(\frac{1}{2}(1-u(x)),1,x)$ always outputs $(1,h_{0}(x))$ whether whether we use the first clause of the definition or the second clause. It is also clear that $\tilde{k}(\frac{1}{2},t,a)=(t,a)$ when considering either the third or fourth clause. Now consider $\tilde{k}(s,1,a)$ for $a\in A$. This is $(\min(2s,1),a)$ according to the second half of the definition, we wish to check that this is true according to the first half of the definition. As $a\in A$ we note that $u(a)=0$ from condition $6$ of \ref{NDR}. Thus the first half of the definition degenerates to the below special case:
$$
\tilde{k}(s,1,a)=\left\{\begin{array}{ll}
(2s,h_{0}(a))& \quad 0\leqslant s\leqslant \frac{1}{2}\\
(1,h(2s-1,a))& \quad \frac{1}{2}\leqslant s\leqslant 1.
\end{array}
\right.
$$

Condition $4$ of the NDR definition clarifies that this gives $(\min(2s,1),a)$ as required. This checks that $\tilde{k}$ is well-defined and internally consistent.

We now additionally claim that the image of $\tilde{k}$ in fact lies in $M_{i}$. To check this, we first note that the bottom half of the definition has image in $[0,1]\times A$ and the second clause in the top half of the definition has image in $\{1\}\times X$. Thus the only possible issue is the first clause of the definition. From condition $5$ of \ref{NDR} we are fine if $u(x)$ is less than $1$ as then $h_{0}(x)$ is in $A$. Letting $u(x)=1$ we consider the first cause of the definition precisely when $s=0$, this is observable by looking at the inequality subdivision. Finally, noting that in this case the image is $(1,h_{0}(x))$ is enough to show that the image of $\tilde{k}$ is in $M_{i}$.

We now further claim that the map restricts to a map $k:[0,1]\times C_{i}\to C_{i}$. To see this observe that $\tilde{k}(s,0,a)=(0,a)$ - this allows us to suitably quotient out. We now take $k$ as our candidate homotopy, to observe that $k_{0}$ is $r\circ q$ and that $k_{1}$ is the identity it is enough to observe the behaviour of
$\tilde{k}$ at $0$ and $1$ - the below four properties of $\tilde{k}$ are enough to show that the homotopy is the one we want:
$$\begin{array}{rcl}
\tilde{k}(0,1,x)&=&(u(x),h_{0}(x))\\
\tilde{k}(0,t,a)&=&(0,a)=(u(a),h_{0}(a))\\
\tilde{k}(1,1,x)&=&(1,h_{1}(x))=(1,x)\\
\tilde{k}(1,t,a)&=&(t,a).
\end{array}
$$

Finally we note that all of the above work interacts well with basepoints. Thus $k$ provides a homotopy between $r\circ q$ and the identity; combining this with the result that $q\circ r$ is homotopic to the identity and the earlier homotopy commutative diagram is enough to prove the claim.
\end{proof}

We can also build cofibre sequences from inclusions which are also classical cofibrations by noting that they are equivalent to NDR pairs. We take the definition of a cofibration here to be as follows:

\begin{defn}\label{cofibsandHEP} Let $\text{inc}:A\to X$ be a continuous inclusion. We say $\text{inc}$ has the
homotopy extension property if given any space $Y$ and any two maps $f:X\to Y$ and $g:A\to Y$ such that $f|_{A}\simeq g$ then there is an extension $\tilde{g}:X\to Y$ of $g$ which is homotopic to $f$ and such that $\tilde{g}|_{A}=g$. We say that a map $A\to X$ is a cofibration if it has the homotopy extension property.
\end{defn}

The next result is then standard, a proof is included in Chapter $6$ Section $4$ of \cite{ConciseMay}.

\begin{prop}\label{classicalcofibrationcofibres} $(X,A)$ is an NDR pair if and only if $\text{inc}:A\to X$ is a
cofibration. In particular this condition implies that the standard collapse $C_{\text{inc}}\to C_{\text{inc}}/C(A)\cong X/A$ is a homotopy equivalence.
\end{prop}

We remark here that this result proves Lemma \ref{cofibreslemma}. This gives us a reasonable method of building cofibres which we will use liberally in proving the main theorem.

\section{Taking Quotients Throughout Cofibre Sequences}\label{TotalCofibres}

In order to prove the main theorem we wish to construct certain cofibre sequences; however, it will turn out to be more convenient to construct simpler, related cofibre sequences by globally taking a quotient. In this section we detail how to do this. We firstly prove the following proposition:

\begin{prop}\label{takingquotientswithcofibrations} Let $X$, $Y$ and $Z$ be such that $Z$ includes into both $X$ and $Y$
and that the inclusions are cofibrations. Moreover let $f:X\to Y$ be such that $f$ factors as $\bar{f}:X/Z\to Y/Z$. Then the homotopy cofibre of $f$ is homotopy equivalent to the homotopy cofibre of $\bar{f}$.
\end{prop}
\begin{proof} We have the below diagram of cofibre sequences:
$$\xymatrix{Z\ar@{ >->}[d]\ar[r]^{1}&Z\ar@{ >->}[d]\ar[r]&C(Z)\\
X\ar[d]\ar[r]^{f}&Y\ar[d]\ar[r]&C_{f}\\
\frac{X}{Z}\ar[r]_{\bar{f}}&\frac{Y}{Z}\ar[r]&C_{\bar{f}}}$$

It is easy to see that the cofibre of $\bar{f}$ is naturally homeomorphic to $C_{f}/C(Z)$. Moreover the map $C(Z)\to C_{f}$ is a cofibration; this is due to the assumption that the two inclusions are cofibrations. Hence by \ref{classicalcofibrationcofibres} the cofibre of the map $C(Z)\to C_{f}$ is $C_{f}/C(Z)\cong C_{\bar{f}}$.
Thus the map $C_{f}\to C_{\bar{f}}$ fits in a cofibre sequence with $C(Z)$. As $C(Z)$ is contractible we get that the map $C(Z)\to C_{f}$ is null-homotopic and thus its cofibre is $C_{f}\wedge\Sigma C(Z)$. We again observe that $C(Z)$ is contractible and it follows that $C_{f}\to C_{\bar{f}}$ is a homotopy equivalence.
\end{proof}

The above result demonstrates that it is possible to take a quotient throughout a cofibre sequence without affecting the cofibre. We briefly note here that there is also a shortcut for this proof when working in the stable category; this uses the octahedral axiom of a triangulated category. We have the following commutative diagram, the edge maps defined by the axiom:
$$\xymatrix{&&C_{f}\ar[dl]|\bigcirc\ar@/_1pc/[ddll]|\bigcirc&&\\
&X\ar[rr]\ar[dl]&&Y\ar[dr]\ar[ul]&\\
X/Z\ar[rr]|\bigcirc\ar@/_1pc/[rrrr]&&Z\ar[ur]\ar[ul]&&Y/Z\ar[ll]|\bigcirc\ar@/_1pc/[uull]}$$

By the axiom the outer maps form a cofibre sequence. We now extend the above result to remove the cofibration conditions.
This will be accomplished via a brief discussion of a generalization of the notion of a cofibre, that of the total cofibre. This mirrors a result surveyed by Goodwillie in Section $1$ of \cite{Goodwillie2} - we prove dual results to those given for total fibres.

Let $S$ be a finite set. In most usual cases $S$ is going to be the set $\underline{n}=\{1,2,\ldots,n\}$. We have the power set of $S$, $\pos$, partially ordered under inclusion. Now let $\Cat$ be any category.

\begin{defn}\label{cubicaldiagram} A cubical diagram is a functor $\Cube:\pos\to\Cat$. We say that $\Cube$ is as an $S$-cube, or if $S=\underline{n}$ an $n$-cube.
\end{defn}

\begin{rem}\label{cubicaldecomposition} An $n$-cube $\Cube$ can be thought of as a map $\Cubey\to\Cubez$ of $(n-1)$-cubes $\Cubey$ and $\Cubez$ - within each $n$-cube one can choose two disjoint $(n-1)$-cubes and there
are $n$ different ways making this choice. This thus gives us $n$ different decompositions of an $n$-cube into a map of $(n-1)$-cubes.
\end{rem}

Homotopy cofibres generalize in a natural way to objects of $\Cat$ called total cofibres, wherein you associate to each cube $\Cube$ a total cofibre $\tcx$. A lemma proved below then allows you to transfer work with total cofibres up and down different sizes of cubes.

Now take $\Cat$ to be $\Top$, the category of topological spaces, though we note here that we could substitute $\Top$ with $CGWH$ or any model category of spectra with no change to the theory. There are many equivalent definitions of total cofibres, the first we give is defined through the theory of homotopy colimits defined in Chapter $12$ of \cite{BousfieldKan}; note that these are referred to as homotopy direct limits in the reference.

Firstly let $\Cube$ be an $S$-cube of spaces and let $\pones$ be the poset of all $T\subset S$ such that $T\neq S$. Then we have the composed functor $\pones\to\pos\to\Top$. Define $h_{1}(\Cube)$ to be the homotopy colimit of this composed functor, $h_{1}(\Cube):=\hocolim(\Cube|_{\pones})$.

We have the inclusion $h_{1}(\Cube)\to\hocolim(\Cube)$ and this inclusion is also a cofibration. Now note that $S$ is terminal in $\pos$. Thus the homotopy colimit of $\Cube:\pos\to\Top$ is homotopy equivalent to $\Cube(S)$. Combining these two results allows us to construct the following composition: $$h_{1}(\Cube)=\hocolim(\Cube|_{\pones})\overset{\text{cofib.}}{\to}\hocolim(\Cube)\overset{\sim}{\to}\text{colim}(\Cube)=\Cube(S).$$

Denote this map by $b(\Cube)$.

\begin{defn}\label{totalcofibre1} The total cofibre of $\Cube$, $\tcx$, is defined to be the homotopy cofibre of the map $b(\Cube)$ above, $\tcx:=C_{b(\Cube)}$.
\end{defn}

There is an alternate definition of total cofibre for when $\Cube$ is an $n$-cube expressed as a map of $(n-1)$-cubes
$\Cubey\to\Cubez$. We have the following $2$-cube $\mathcal{A}$, using $\Cubey(n-1)$ and $\Cubez(n-1)$ to denote the objects in $\Top$ arising from applying the functors $\Cubey$ and $\Cubez$ to the terminal object of $\mathcal{P}(n-1)$:
$$\xymatrix{h_{1}(\Cubey)\ar[r]\ar[d]_{b(\Cubey)}&h_{1}(\Cubez)\ar[d]^{b(\Cubez)}\\
\Cubey(n-1)\ar[r]&\Cubez(n-1)}$$

\begin{altdefn}\label{totalcofibre2} The map $b(\Cube)$ is the map $b$ for the above $2$-cube $\mathcal{A}$. Thus $\tcx:=\tilde{c}\mathcal{A}$.
\end{altdefn}

These two definitions are easily observed to be the same: when considering the map $b(\mathcal{A})$, one must look at $h_{1}(\mathcal{A})$ and $\mathcal{A}(2)$. Observe that $h_{1}(\mathcal{A})$ is the homotopy colimit of the following diagram:
$$\Cubey(n-1)\leftarrow h_{1}(\mathcal{Y})\rightarrow h_{1}(\mathcal{Z}).$$
The observation that $\Cube(n)=\Cubez(n-1)$ allows us to see that the homotopy colimit of the above diagram will naturally be the homotopy colimit of $\Cube|_{\mathcal{P}_{1}(n)}$, which by definition is $h_{1}(\Cube)$. Also note that $\mathcal{A}(2)$ is just $\Cubez(n-1)$, which is $\Cube(n)$ by the same observation. It is trivial from here to observe that $b(\mathcal{A})$ is the same map as $b(\Cube)$.

This alternate construction allows us to prove the following lemma:
\begin{lem}\label{totalcofibreslemma} Let $\Cube$ be an $n$-cube and let $\Cubey$ and $\Cubez$ be $(n-1)$-cubes such that $\Cube$ is built from a map $\Cubey\to\Cubez$. Then $\tcx\simeq\text{cof}(\tilde{c}\Cubey\to\tilde{c}\Cubez)$, where
$\text{cof}$ indicates the homotopy cofibre of the map of spaces.
\end{lem}
\begin{proof} This result essentially follows from the two equivalent constructions of total cofibres given above. Build $\tilde{c}\Cube$ from Definition \ref{totalcofibre2}, giving:
\begin{eqnarray*}
\tilde{c}\Cube&\simeq &C(h_{1}(\mathcal{A}))\cup_{b}\Cubez(n-1)\\
&\simeq&C\left(\frac{\Cubey(n-1)\sqcup h_{1}(\Cubey)\times I\sqcup
h_{1}(\Cubez)}{\sim}\right)\cup_{b}\Cubez(n-1).
\end{eqnarray*}
Here $(x,0)\sim b(\Cubey)(x)$ and $h_{1}(\Cubey)\times\{1\}$ is adjoined to $h_{1}(\Cubez)$ via $\Cube$. Now $\tilde{c}\Cubey=C(h_{1}(\Cubey))\cup_{b}\Cubey(n-1)$ and $\tilde{c}\Cubez=C(h_{1}(\Cubez))\cup_{b}\Cubez(n-1)$. This gives:
$$\text{cof}(\tilde{c}\Cubey\to\tilde{c}\Cubez)=C(C(h_{1}(\Cubey))\cup_{b}\Cubey(n-1))\cup_{\Cube}(C(h_{1}(\Cubez))\cup_{b}\Cubez(n-1)).$$
In these forms the cofibre and total cofibre are easily observed to be the same.
\end{proof}

This allows us to calculate homotopy cofibres using total cofibres - there are $n-1$ ways of building an $n$-cube from two $(n-1)$-cubes which will allow us to `work backwards' somewhat, calculating total cofibres via the homotopy cofibres we know and then calculating homotopy cofibres we don't know via the new total cofibre information. In particular we prove the following lemma:

\begin{lem}\label{quotientcofibresequences} Let $X$ and $Y$ be connected and simply connected based $CW$-complexes such that a simply connected $CW$-complex $Z$ includes into both. Further, let $f:X\to Y$ be such that there is an associated map $\bar{f}:X/Z\to Y/Z$. Then the homotopy cofibre of $f$ is equivalent to the homotopy cofibre of $\bar{f}$.
\end{lem}
\begin{proof} Let $\text{pt.}$ denote the single point space. We have a $3$-cube:
$$\xymatrix{Z\ar[rr]\ar[dd]\ar[dr]&&\text{pt.}\ar[dr]\ar[dd]&\\
&X\ar[rr]\ar[dd]&&X/Z\ar[dd]\\
Z\ar[rr]\ar[dr]&&\text{pt.}\ar[dr]&\\
&Y\ar[rr]&&Y/Z}
$$

The top and bottom faces are both homotopy pushouts and thus have zero total cofibre. Thus by Lemma \ref{totalcofibreslemma} the total cofibre of the front face is equal to the total cofibre of the back face. The maps $\text{pt.}\to\text{pt.}$ and $Z\to Z$ are identity maps and have zero cofibres, giving the rear face zero total cofibre. By the lemma the front face must now have zero total cofibre; thus another application of the lemma gives us that the cofibre of the natural map $C_{f}\to C_{\bar{f}}$ is zero. We thus have as standard the following exact sequence of homotopy classes for any space $A$, this is covered in more detail in Chapter $8$ section $4$ of \cite{ConciseMay}:
$$\pt\to[C_{\bar{f}},A]\to[C_{f},A].$$
It follows that $C_{f}$ and $C_{\bar{f}}$ have the same integral cohomology and hence by the Universal Coefficient Theorem the same integral homology. Due to the assumptions both $C_{f}$ and $C_{\bar{f}}$ will be connected and simply connected so by the Whitehead Theorem for spaces (we refer the reader to Theorem $9$ in $7.5$ of \cite{spanier} for the exact formulation we want) the natural map between $C_{f}$ and $C_{\bar{f}}$ gives an isomorphism of homotopy groups, and due to the connectedness assumption there is no issue with connected components. Both our cofibres are $CW$-type by assumption, hence by another application of the Whitehead Theorem (this time the unstable and nonequivariant form of the result stated in \ref{thewhiteheadthm}) the natural map is a homotopy equivalence.
\end{proof}

We note here that it is a triviality to extend this result to $G$-$CW$-complexes. This now gives us the means to take quotients in cofibre sequences, which will prove useful in the proof of the main theorem.

\section{Splittings via Cofibre Sequences}\label{SplittingsviaCofibreSequences}

We finally detail how certain cofibre sequences stably split, allowing us to derive results of the form $B\simeq A\vee C$ from cofibre sequences $A\to B\to C\to \Sigma A$ with certain properties. Ideas of this kind are standard and will be used in Appendix \ref{ch:appendixa} to build the splitting of Miller. Moreover, we will use the theory in Chapter
\ref{ch:ch7} to explore the consequences should certain conjectures hold. For this section we restrain ourselves to working with $G$-spectra.

\begin{prop}\label{howacofibseqsplits} Let $A\overset{f}{\to}B\overset{g}{\to} C\overset{h}{\to}\Sigma A$ be a cofibre sequence of finite $G$-$CW$-spectra such that either of the following holds:
\begin{itemize}
\item There is a map $f':B\to A$ such that $f'\circ f$ is an isomorphism in $\mathcal{F}_{G}$ on $A$.
\item There is a map $g':C\to B$ such that $g\circ g'$ is an isomorphism in $\mathcal{F}_{G}$ on $C$.
\end{itemize}
Then $B$ is isomorphic in $\mathcal{F}_{G}$ to $A\vee C$.
\end{prop}

\begin{rem}\label{unstablenearlysplitting} In particular, if $A\overset{f}{\to}B\overset{g}{\to} C\overset{h}{\to}\Sigma A$ is a cofibre sequence of finite $G$-$CW$-complexes such that, for example, there exists a map $f':B\to A$ such that $f'\circ f\simeq \text{Id}_{A}$ then the associated map of spectra gleaned by applying $\Sigma^{\infty}$ would be an isomorphism and hence there would be a stable splitting.
\end{rem}

In order to prove this we first prove the below lemma:

\begin{lem}\label{hisnullhomotopic} Under one of the conditions above the map $h$ is null.
\end{lem}
\begin{proof} First consider the case where $f'$ exists. From the detail of $\S$\ref{BasicTheory} it is clear that the composition $\Sigma f\circ h:C\to \Sigma B$ is null. Thus $\Sigma f'\circ\Sigma f\circ h$ is null, but as $f'\circ f$ is an isomorphism this implies that $h$ is null.

For the case where $g'$ exists, we note that similar to above the composite $h\circ g$ is null. Thus $h\circ g\circ g'$ is null and $g\circ g'$ being being an isomorphism implies that $h$ is null.
\end{proof}

We now move to a homotopy groups argument. As noted in Chapter $3$, Section $2$ of \cite{LewisMaySteinberger} for any subgroup $H$ we have the below long-exact sequence of homotopy groups. This only applies in the stable case:
$$\ldots\to \pi_{n}^{H}(A)\overset{f_{*}}{\to}\pi_{n}^{H}(B)\overset{g_{*}}{\to}\pi_{n}^{H}(C)\overset{h_{*}}{\to}\pi_{n-1}^{H}(A)\to\ldots.$$
As $h$ is null the maps $h_{*}$ are all zero. Thus we can replace this sequence with the following short exact sequences:
$$0\to \pi_{n}^{H}(A)\overset{f_{*}}{\to}\pi_{n}^{H}(B)\overset{g_{*}}{\to}\pi_{n}^{H}(C)\to 0.$$
We now recall the Splitting Lemma, a basic result found in many sources, for example it is stated in Section $2.2$ of \cite{Hatcher}:

\begin{lem}\label{thesplittinglemma} Let $0\to X\overset{a}{\to} Y\overset{b}{\to} Z\to 0$ be a short exact sequence of
abelian groups. Then the following are equivalent:
\begin{enumerate}
\item There exists $a':Y\to X$ such that $a'\circ a:X\to X$ is the
identity.
\item There exists $b':Z\to Y$ such that $b\circ b':Z\to Z$ is the
identity.
\item There exists an isomorphism $f:Y\to X\oplus Z$ such that $f(a(x))=(x,0)$ and $b(f^{-1}(x,z))=z$ for all $x\in X$ and $z\in Z$.
\end{enumerate}
\end{lem}
We apply this to the short exact sequences of homotopy groups to retrieve result $3$, noting that we can do this because the either/or condition in the statement of Proposition \ref{howacofibseqsplits} gives us either one or the other of the inverses needed. We now apply an equivariant form of the Whitehead Theorem, which we state below. Non-equivariantly this can be found in many standard reference books, the result is $3.5$ of Part III in \cite{Adams1} for example, while an equivariant form is included as $5.10$ in Chapter $1$ of \cite{LewisMaySteinberger}:

\begin{prop}\label{thewhiteheadthm} Let $X$ and $Y$ be finite $G$-$CW$-spectra. Then $f:X \to Y$ is an isomorphism in $\mathcal{F}_{G}$ if and only if $\pi_{*}^{H}(f):\pi_{*}^{H}(X)\cong \pi_{*}^{H}(Y)$ for all $H\leqslant G$.
\end{prop}

An application of this theorem to the split exact sequences above thus produces a proof of Proposition \ref{howacofibseqsplits}. We also note here that using similar techniques we can prove another useful result. It is clear that the cofibre of $A\overset{\text{id.}}{\to} A$ is isomorphic to a point in $\mathcal{F}_{G}$, we prove the converse:

\begin{prop}\label{zerocofibre} Let $A$ and $B$ be two finite $G$-$CW$-spectra and $f:A\to B$ be such that
the map has zero cofibre. Then $A$ is isomorphic to $B$ via $f$.
\end{prop}
\begin{proof} We have a long exact sequence of homotopy groups:
$$\ldots\pi_{n+1}^{H}(C_{f})\to \pi_{n}^{H}(A)\overset{f_{*}}{\to}\pi_{n}^{H}(B)\to\pi_{n}^{H}(C_{f})\to\ldots.$$
As $C_{f}$ is zero the sequence thus gives us the following short exact sequences:
$$0\to \pi_{n}^{H}(A)\overset{f_{*}}{\to}\pi_{n}^{H}(B)\to0.$$
Thus $A$ and $B$ have matching homotopy groups and the result follows from the Whitehead Theorem \ref{thewhiteheadthm}.
\end{proof}

We will be able to use these results throughout the rest of the document to derive splittings and equality from the homotopy cofibre sequences we build.
\chapter{Functional Calculus}
\label{ch:ch3}

\section{The Standard Theory}\label{TheStandardTheory}

The functional calculus is a technique from analysis used to study maps from one vector space to another. Morally, it allows us to build new maps from old ones by applying functions to spaces of eigenvalues. In this section we state the
basic result and discuss its implications for our goal - all of this is standard and can be found throughout the literature.

Throughout this chapter and beyond $V$ and $W$ will tend to refer to finite dimensional vector spaces equipped with a Hermitian inner product from which we also retrieve norms given by $\|v\|:=\sqrt{\langle v,v\rangle}$. This structure also allows us to define an adjoint $\gamma^{\dagger}\in \Hom(W,V)$ from $\gamma\in\Hom(V,W)$; by the Riesz representation theorem this construction is characterized by $\langle \gamma(v),w\rangle=\langle v,\gamma^{\dag}(w)\rangle$ for all $v\in V$ and $w\in W$. In the most part we will be using the functional calculus to study the spaces $\Hom(V,W)$ and $\End(V)=\Hom(V,V)$. We equip these spaces with the following norm:

\begin{defn}\label{operatornorm} The operator norm on $\Hom(V,W)$ is given by:
$$\|\gamma\|=\sup\{\|\gamma(v)\|:\|v\|=1\}.$$
\end{defn}

\begin{defn}\label{analyticspectrum} Let $V$ be Hermitian and let $\gamma\in \End(V)$. The (analytic) spectrum of $\gamma$ is  given as follows:
$$\sigma_{\End(V)}(\gamma):=\{\lambda\in\Complex:\lambda-\gamma\text{ is not invertible}\}.$$
For $V$ finite dimensional the spectrum of $\gamma$ precisely corresponds to the set of eigenvalues of $\gamma$.
\end{defn}

Manipulation of the analytic spectrum of an endomorphism of certain type leads to many interesting and fundamental results within functional analysis.

\begin{defn}\label{normalelement} We say that an endomorphism $\gamma$ is normal if it commutes with its
adjoint. In particular we recall $s(V)$ from Definition \ref{svo} to be the space of all selfadjoint endomorphisms of $V$, noting that our old and new definitions for adjoint do coincide. We note that every $\alpha\in s(V)$ is normal. 
\end{defn}

It is a standard result that $\gamma$ is normal if and only if it is diagonalizable. Study of the spectral theory of normal maps leads towards a deep and interesting facet of analysis framed around a group of results known as the functional calculus. One of these results is stated below, variations of the Lemma can be found in many basic sources on functional analysis. The opening chapters of the book \cite{DavidsonC*Algebras}, for example, detail the general theory; appendix $A$ of \cite{StricklandSubbundles} demonstrates the result applied to a topological context.

\begin{lem}\label{fnalcalclemma} Let $\gamma\in\End(V)$ be normal, let $X\subseteq \Complex$ be such that $\sigma_{\End(V)}(\gamma)\subseteq X$ and let $f:X\to \Complex$ be a continuous function. Then there exists a normal endomorphism $f(\gamma)$ of $V$ determined by $f$ as discussed in the remark below and such that $\text{id}(\gamma)=\gamma$, $(f+g)(\gamma)=f(\gamma)+g(\gamma)$, $(fg)(\gamma)=f(\gamma)g(\gamma)$, $f\circ g(\gamma)=f(g(\gamma))$, $c(\gamma)=c$ for $c$ a constant map and $\bar{f}(\gamma)=f(\gamma)^{\dag}$ where bar is complex conjugation. Moreover, this construction is continuous in the following sense. Let $K$ be a closed subset of $\Complex$. Denote by $\mathcal{C}(K,\Complex)$ the space of continuous functions from $K$ to $\Complex$. Then the below map is continuous:
$$\mathcal{C}(K,\Complex)\times\{\gamma\in\End(V):\gamma^{\dag}\gamma=\gamma\gamma^{\dag},\sigma_{\End(V)}(\gamma)\subseteq K\}\to \End(V)$$
$$(f,\gamma)\mapsto f(\gamma).$$
\end{lem}

\begin{rem}\label{howfnalcalcworks} If $\gamma$ is an endomorphism with eigenvalues $\{\lambda\}$ and corresponding eigenspaces $\{\langle v\rangle\}$ then $f(\gamma)$ is the endomorphism with eigenvalues $\{f(\lambda)\}$ and eigenspaces $\{\langle v\rangle\}$. For example $\Exp(\gamma)$ has eigenvalues $\{e^{\lambda}\}$ and $\log(\gamma)$ has eigenvalues $\{\log(\lambda)\}$. Moreover it is easy to see that $\Exp(\log(\gamma))=\log(\Exp(\gamma))=\gamma$.
\end{rem}

Although we have defined the above for general normal operators, we will tend to only consider selfadjoint elements of $\End(V)$. Doing this specializes our results via the below lemma:

\begin{lem}\label{selfadjointeigenspaces} Let $\gamma\in \End(V)$. Then $\gamma\in s(V)$ if and only if the eigenvalues of $\gamma$ are real. Now let $\gamma\in \Hom(V,W)$ and recall the notation defined in \ref{svo} for selfadjoint endomorphisms with positive eigenvalues. Then $\gamma^{\dag}\gamma\in s_{+}(V)$.
\end{lem}

Thus if we want to use Lemma \ref{fnalcalclemma} to build new selfadjoint endomorphisms from old we need to restrict to the case where the domain of the function $f$ is a suitable subset of $\Real$. Moreover if we want $f(\gamma)$ to be selfadjoint then by Lemma \ref{selfadjointeigenspaces} and Remark \ref{howfnalcalcworks} we need $f$ to have codomain $\Real$. Similarly we can further restrict ourselves to looking at functions $f:\Real^{+}\to \Real^{+}$ to give us maps $s_{+}(V)\to s_{+}(V)$, for example if $\gamma$ is selfadjoint and positive then there is a well-defined positive selfadjoint endomorphism given by $\gamma^{\frac{1}{2}}$. Two immediate corollaries of this restriction are stated below:

\begin{cor}\label{shomeotosplus} For any Hermitian space $V$ we have $s(V)\cong s_{++}(V)$ via the following maps:
$$\Exp:s(V)\to s_{++}(V)$$
$$\gamma\mapsto \Exp(\gamma)$$
$$s(V)\leftarrow s_{++}(V):\log$$
$$\log(\gamma)\mapsfrom \gamma.$$
\end{cor}

\begin{cor}\label{themaprho} We have a well-defined continuous map $\rho:\Hom(V,W)\to s_{+}(V)$ for any Hermitian spaces $V$ and $W$ given by $\rho(\gamma)=(\gamma^{\dag}\gamma)^{\frac{1}{2}}$.
\end{cor}

We now note a few standard properties of the map $\rho$ above. We first observe that for all $v\in V$ we have $\|\rho(\gamma)(v)\|=\|\gamma(v)\|$. This makes it trivial to see that $\Ker(\rho(\gamma))=\Ker(\gamma)$. It is also easy to see that $\IM(\rho(\gamma))$ is contained within $(\Ker(\gamma))^{\bot}$ and thus by a simple application of rank-nullity that $\IM(\rho(\gamma))=(\Ker(\gamma))^{\bot}$. Restricting $\rho(\gamma)$ to $(\Ker(\gamma))^{\bot}$ then clearly gives an isomorphism $(\Ker(\gamma))^{\bot}\to(\Ker(\gamma))^{\bot}$ via the first isomorphism theorem. This crucially gives us a domain upon which $\rho(\gamma)$ has an inverse and also allows us to define the following map
$\sigma$:

\begin{prop}\label{themapsigma} For each $\gamma\in\Hom(V,W)$ there is a well-defined continuous map
$\sigma(\gamma):(\Ker(\gamma))^{\bot}\to W$ given by:
$$\sigma(\gamma):=\left((\Ker(\gamma))^{\bot}\overset{\rho(\gamma)^{-1}}{\longrightarrow}(\Ker(\gamma))^{\bot}\overset{\gamma}{\longrightarrow} W\right).$$
Moreover, $\sigma(\gamma)$ is a linear isometry and $\gamma=\sigma(\gamma)\circ\rho(\gamma)$.
\end{prop}
\begin{proof} That the map is well-defined follows from the above remarks on the existence of an inverse to $\rho(\gamma)|_{(\Ker(\gamma))^{\bot}}$. Continuity of $\sigma(\gamma)$ will then follow from the observation that this
restriction of $\rho(\gamma)$ is a homeomorphism onto $(\Ker(\gamma))^{\bot}$. This follows for functional calculus reasons from the observation that the map below is a homeomorphism:
$$f:(0,\infty)\to (0,\infty)$$
$$\iota\mapsto \iota^{1/2}.$$

To show $\sigma(\gamma)$ is an isometry we consider $(\sigma(\gamma))^{\dag}\sigma(\gamma)$; we wish to show that this is the identity. Naively this can be thought of as $\rho(\gamma)^{-1}\circ \gamma^{\dag}\circ\gamma\circ \rho(\gamma)^{-1}|_{(\Ker(\gamma))^{\bot}}$, however, this is not as stated correct. We need to be clear about the domains and codomains of the left-hand $\rho(\gamma)^{-1}$ and the middle $\gamma$ and $\gamma^{\dag}$, specifying them exactly in order to make this expression correct.

Firstly, we let $i:(\Ker(\gamma))^{\bot}\to V$ be the inclusion, then we have $i^{\dag}:V\to(\Ker(\gamma))^{\bot}$ the projection. We let $\rho_{1}(\gamma)$ be the unique restriction filling in the following commutative square:
$$\xymatrix{V\ar@{->>}[d]_{i^{\dag}}\ar[r]^{\rho(\gamma)}&V\\
(\Ker(\gamma))^{\bot}\ar[r]_{\rho_{1}(\gamma)}^{\cong}&(\Ker(\gamma))^{\bot}\ar@{ >->}[u]_{i}}$$

Thus $\rho(\gamma)=i\circ\rho_{1}(\gamma)\circ i^{\dag}$. It is also easy to see that $\sigma(\gamma)$ is $\gamma\circ i\circ\rho_{1}(\gamma)^{-1}$. We thus check the following calculation:
$$(\sigma(\gamma))^{\dag}\sigma(\gamma)=\rho_{1}(\gamma)^{-1}\circ i^{\dag}\circ
\rho(\gamma)^{2}\circ i\circ \rho_{1}(\gamma)^{-1}=\rho_{1}(\gamma)^{-1}\circ\rho_{1}(\gamma)^{2}\circ\rho_{1}(\gamma)^{-1}=\text{id}.$$

This demonstrates that $\sigma(\gamma)$ is an isometry. Finally, the last claim is trivial to observe.
\end{proof}

There is an immediate corollary to the above in the special case when $\gamma^{\dag}\gamma$ is strictly positive. In these circumstances we can define $\rho(\gamma)^{-1}=(\gamma^{\dag}\gamma)^{-\frac{1}{2}}$ globally using Lemma \ref{fnalcalclemma}; the problem point of the spectrum potentially containing $0$ is avoided. This allows us to define the map $\sigma(\gamma)$ globally if $\gamma^{\dag}\gamma\in s_{++}(V)$.

Combining some of the ideas above we can build the following homeomorphism; this will be used in the construction of the main result. We note here that there appears to be a trivial minus sign in the formulation. This exists for technical
reasons - it is needed so that we can link the result back in with the Miller splitting and induce an identity homotopy later. We recall the notation $S^{V}$ for the one-point compactification of a vector space $V$, thought of as a sphere.

\begin{prop}\label{thetaEalpha} Let $V$ and $W$ be Hermitian and such that $\Dim(V)\leqslant \Dim(W)$. We have the following homeomorphism:
$$\kappa':s(V)\times\mathcal{L}(V,W)\cong\inj(V,W)$$
$$(\alpha,\theta)\mapsto -\theta\circ \Exp(\alpha).$$

We thus have a continuous extension $\kappa$ giving a homeomorphism on the one-point compactifications, $\kappa:S^{s(V)}\wedge\mathcal{L}(V,W)_{\infty}\cong\inj(V,W)_{\infty}$. Finally we also have a collapse map $\kappa^{!}:S^{\Hom(V,W)}\to
S^{s(V)}\wedge\mathcal{L}(V,W)_{\infty}$.
\end{prop}
\begin{proof} We build this homeomorphism in two stages. Firstly from Corollary \ref{shomeotosplus} we have a homeomorphism $s(V)\cong s_{++}(V)$ given by $\alpha\mapsto \Exp(\alpha)$. To build $\kappa'$ we compose with the map $s_{++}(V)\times \mathcal{L}(V,W)\to\Hom(V,W)$ given by $(\beta,\theta)\mapsto-\theta\circ\beta$. We wish to check that the image of this map is in $\inj(V,W)$ and that $\kappa'$ has a continuous inverse.

Now let $\gamma\in\text{Im}(\kappa')$, i.e. $\gamma=-\theta\circ \Exp(\alpha)$ for some $\theta\in\mathcal{L}(V,W)$ and $\alpha\in s(V)$. Then $\gamma^{\dag}\gamma=(\Exp(\alpha))^{2}$ which is invertible, thus showing that $\gamma\in\inj(V,W)$ as required.

We proceed to show $\kappa'$ is a homeomorphism by building a continuous inverse. $\gamma^{\dag}\gamma$ is selfadjoint and strictly positive and thus we can use Corollary \ref{shomeotosplus}, Corollary \ref{themaprho} and Proposition \ref{themapsigma} to build the below continuous map:
$$\inj(V,W)\to s(V)\times\mathcal{L}(V,W)$$
$$\gamma\mapsto
(\log(\rho(\gamma)),-\sigma(\gamma) ).$$

We now check that this map is well-defined. This, however, follows from earlier remarks. We have already noted that $\sigma(\gamma)$ will be an isometry if it can be built as a map out of $V$, this is fine as $\gamma$ is injective.  That $\log(\rho(\gamma))\in s(V)$ follows from Corollary \ref{shomeotosplus} and from noting that $\rho(\gamma)\in s_{++}(V)$.

Observing that the map is an inverse to $\kappa'$ is also simple. We have one composition given by $\sigma(\gamma)\circ \Exp(\log(\rho(\gamma)))$ and another given by $(\log(\rho(-\theta\circ \Exp(\alpha))),-\sigma(-\theta\circ
\Exp(\alpha)))$ - recalling from above that $\rho(\gamma)=(\gamma^{\dag}\gamma)^{1/2}$, that $\sigma(\gamma)=\gamma\circ(\gamma^{\dag}\gamma)^{-1/2}$ and that if $\gamma=-\theta\circ \Exp(\alpha)$ then
$\gamma^{\dag}\gamma=(\Exp(\alpha))^{2}$ allows one to check that both compositions degenerate to the identity.

The second claim then follows trivially from the first. Finally observe that there is an open embedding $\inj(V,W)\hookrightarrow\Hom(V,W)$ and hence $\kappa'$ gives an embedding $s(V)\times\mathcal{L}(V,W)\hookrightarrow\Hom(V,W)$. Thus the required collapse map $\kappa^{!}$ can be constructed.
\end{proof}

There is another method of building new maps from selfadjoint endomorphisms, this time using the selfadjoint property
rather than relying on the weaker concept of normal operators. Let $\alpha$ be a selfadjoint endomorphism of a finite dimensional Hermitian space $V$, then it has real eigenvalues which can be ordered via the standard $\leqslant$ ordering on $\Real$. Set $e_{j}(\alpha)$ as the $(j+1)^{th}$ eigenvalue under this ordering, $j$ running from $0$ to $\dim(V)-1$. We also use $e_{top}(\alpha)$ as a notation of convenience for the top eigenvalue $e_{\Dim(V)-1}(\alpha)$ under this ordering. Then define the function $e_{j}:s(V)\to \Real$ to be the map which sends $\alpha$ to $e_{j}(\alpha)$. Furthermore, denote by $\eta':s(V)\to \Real^{n}$ the map sending $\alpha$ to its eigenvalues $(e_{0}(\alpha),e_{1}(\alpha),\ldots,e_{top}(\alpha))$. We claim that these eigenvalue functions are continuous, though the result is standard we choose to detail the proof for completeness. We first need the lemma below, called the Courant-Fischer Min-Max Theorem, which can be found in \cite{ReedSimonMethModMathPhyVol4} as Theorem $XIII.1$, for
example. We recall the notation $S(V)$ for the unit sphere of a vector space $V$ and $G_{j}(V)$ for the $j^{th}$ dimensional Grassmannian on $V$.

\begin{lem}\label{courantfischer}  Let $V$ be Hermitian of dimension $d$, then for $\alpha\in s(V)$:
$$e_{j}(\alpha)=\max_{W\in G_{d-j}(V)}\min_{w\in S(W)}\langle\alpha(w),w\rangle.$$
\end{lem}

\begin{lem}\label{continuouseigenvalues} The eigenvalue functions $e_{j}$ are continuous.
\end{lem}
\begin{proof} From first principles of the continuity of metric spaces $e_{j}$ is continuous if for every $\alpha$ selfadjoint and every $\epsilon>0$ there exists $\delta>0$ such that if $\|\alpha-\beta\|<\delta$ then
$|e_{j}(\alpha)-e_{j}(\beta)|<\epsilon$. Fix $\epsilon>0$ and let $\|\alpha-\beta\|<\epsilon$. Then continuity follows if we can show that $|e_{j}(\alpha)-e_{j}(\beta)|<\|\alpha-\beta\|$.

To show this, fix $W,W'\in G_{d-j}(V)$. Then for any fixed $\epsilon'>0$ one can choose a $w'\in S(W')$ such that $\langle\alpha(w'),w'\rangle-\min_{w\in S(W)}\langle\alpha(w),w\rangle<\epsilon'$, i.e. that $\langle\alpha(w'),w'\rangle<\epsilon'+\min_{w\in S(W)}\langle\alpha(w),w\rangle$. Now consider $\min_{w\in S(W')}\langle\beta(w),w\rangle$. For our choice of $w'$ we have the below inequality:
$$\min_{w\in S(W')}\langle\beta(w),w\rangle\leqslant \langle\beta(w'),w'\rangle=\langle(\alpha+\beta-\alpha)(w'),w'\rangle.$$
It is standard from the Cauchy-Schwartz inequality that for any $v$ such that $\|v\|=1$ then $\langle \gamma(v),v\rangle\leqslant \|\gamma\|$. Using this and the above inequality we retrieve the following inequality:
$$\min_{w\in S(W')}\langle\beta(w),w\rangle\leqslant \min_{w\in S(W)}\langle\alpha(w),w\rangle+\|\alpha-\beta\|+\epsilon'.$$

This is true for any fixed $\epsilon'$. Hence the following statement holds:
$$\min_{w\in S(W')}\langle\beta(w),w\rangle\leqslant \min_{w\in S(W)}\langle\alpha(w),w\rangle+\|\alpha-\beta\|.$$
The work above is symmetric in $\alpha$ and $\beta$. Thus the below statement is true:
$$\left|\min_{w\in S(W)}\langle\alpha(w),w\rangle-\min_{w\in S(W')}\langle\beta(w),w\rangle\right|\leqslant\|\alpha-\beta\|.$$
Taking maximums and re-equating the eigenvalues via the Courant-Fischer Theorem retrieves that $|e_{j}(\alpha)-e_{j}(\beta)|<\|\alpha-\beta\|$. This implies continuity in the eigenvalue functions.
\end{proof}

There is one other result we wish to note regarding the behaviour of the eigenvalues of selfadjoint
endomorphisms. This can be found throughout the literature, for example in \cite{PryceFunctionalAnalysis}, and details how the operator norm simplifies using some of the machinery we have covered.

\begin{lem}\label{operatornormiseigenvalues} Let $\gamma\in\Hom(V,W)$. Then $\|\gamma\|=e_{top}(\rho(\gamma))$. In particular if $\gamma$ is a selfadjoint endomorphism then $\|\gamma\|$ is equal to the spectral radius of
$\gamma$, i.e. $\|\gamma\|=\text{max}(|e_{i}|)$ for $e_{i}$ the $i^{th}$ eigenvalue of $\gamma$.
\end{lem}

This machinery is all we need to now, for example, write down continuous maps of the form $\alpha\mapsto \alpha- e_{0}(\alpha)$ or $\alpha \mapsto \alpha^{-1/2}$. Note the principle behind these constructions; we apply functions to spaces of eigenvalues to build new maps. We now take this principle and construct a variation on the functional
calculus theory.

\section{Variations}\label{Variations}

In order to build a variation of the functional calculus we first need a model of a space of eigenvalues.

\begin{defn}\label{DeeAndFaces} We make the following definitions:
\begin{itemize}
\item Set $D'(d):=\{(t_{0},\ldots,t_{d-1})\in\Real^{d}:t_{0}\leqslant t_{1} \leqslant\ldots\leqslant t_{d-1}\}$. Moreover denote $D'(d)_{\infty}$ by $D(d)$.
\item Set $F_{i}(D'(d)):=\{t\in D'(d):t_{i}=t_{i+1}\}$ and $F_{i}(D(d)):=(F_{i}(D'(d)))_{\infty}$ for $i\in\{0,\ldots,d-2\}$; call these subsets the faces of $D(d)$.
\item A based map $f:D(d)\to D(d)$ is facial if it preserves faces, let $F(d)$ be the space of facial maps $D(d)\to
D(d)$ equipped with the compact-open topology.
\item Set $D'_{+}(d)$ to be the subspace of $D'(d)$ that takes non-negative values in all coordinates, set $D_{+}(d):=D_{+}'(d)_{\infty}$.
\item  Set $D_{0}'(d):=\{(t_{0},\ldots,t_{d-1})\in D_{+}'(d):t_{0}=0\}$ and put $D_{0}(d):= D_{0}'(d)_{\infty}$. Similar to the above we also have faces $F_{i}(D_{+}(d))$ for $i\in\{0,\ldots,d-2\}$.
\item A based map $f:D_{+}(d)\to D_{+}(d)$ is facial if it preserves faces and preserves $D_{0}(d)$. Denote by $F_{+}(d)$ the space of facial maps $D_{+}(d)\to D_{+}(d)$ with the compact-open topology.
\item For any based space $X$ let $\text{FMap}(D(d),D(d)\wedge X)$ be the space of maps $f:D(d)\to D(d)\wedge X$ such that $f(F_{i}(D(d)))\subseteq F_{i}(D(d))\wedge X$; we equip this space with the compact-open topology. Such maps $f:D(d)\to D(d)\wedge X$ are referred to as facial. In more generality, $\text{FMap}(Y,Z)$ will be used in similar situations to refer to spaces of `facial' maps from $Y$ to $Z$; $Y$ and $Z$ in this case are any based spaces such that consistent with the definitions above there is a reasonable notion for a map $Y\to Z$ to be termed facial. Further for any based space $Y$ with a definition of a `facial' self-map $Y\to Y$ we set $\text{FMap}(Y)$ to be the space of facial self maps of $Y$ with the compact-open topology. We rely on the context of each situation to indicate what we mean by facial map at any given time.
\end{itemize}
\end{defn}

\begin{rem}\label{whyDandnotSimplices} We note here that the above definitions resemble simplicial complexes somewhat. In fact the work below can be rephrased in terms of simplicial complexes rather than these spaces $D(d)$. However, we choose not to do this for three reasons. Firstly if we were to rephrase this in simplicial language we don't quite get ordinary simplices. In fact $D(d)$ is the following, where $\st$ is unreduced suspension based at an endpoint and $\Delta_{n}$ is the standard $n$-simplex:
$$D(d)\cong\sst \Delta_{d-2}.$$
Although we do retrieve that the faces of the simplex correspond to our definition of faces above, that we have these extra suspensions cropping up when we phrase the work in terms of simplicial complexes is undesirable. Secondly it is harder to see how universal the work is in this language. Many of the results proved using these definitions rely on rather simple face-preserving homeomorphisms, for example \ref{wereactuallyworkingwithd}. When written in terms of simplices these homeomorphisms are far harder to write down. Finally, the simplicial phrasing looks less obviously like a space modeling eigenvalues. Although this is a mostly aesthetic issue it is still one that we feel is important.
\end{rem}

This definition mirrors the idea of the analytic spectrum from Section \ref{TheStandardTheory}, $D(d)$ is going to in some way model eigenvalues in our variation.

\begin{lem}\label{propereigenvectors} Let $V$ be Hermitian of dimension $d$ and let $\eta':s(V)\to D'(d)$ be the eigenvalue map. Then $\eta'$ is proper and hence we have an extension $\eta:s(V)_{\infty}\to D(d)$. Moreover, $\eta'$ is surjective and hence $\eta$ is a quotient map.
\end{lem}
\begin{proof} Take $C$ a compact subset of $D'(d)$ - as we are working within Euclidean space this is a closed
and bounded set by the Heine-Borel Theorem. Its inverse image is closed due to the continuity of $\eta'$ inherited from the continuity in its factors. Moreover, the inverse image is also clearly bounded - by Lemma \ref{operatornormiseigenvalues} if $\alpha\in s(V)$ then $\|\alpha\|=\|\eta'(\alpha)\|$ and hence if $\eta'(\alpha)\in C$ then there is a bound on $\|\alpha\|$. Thus the preimage is compact due to Lemma \ref{extendedheineborel} and $\eta'$ is proper. The rest of the result follows from the standard claim that a continuous surjection from a compact space to a Hausdorff space is a quotient map.
\end{proof}

\begin{lem}\label{themapnu} For $t\in D'(d)$ set $\Delta(t)$ to be the diagonal matrix with entries $t$. Let $V$ again be Hermitian and of dimension $d$ and let $\nu'$ be the following map:
$$\nu':\mathcal{L}(\Complex^{d},V)\times D'(d)\to s(V)$$
$$(\alpha,t)\mapsto \alpha\Delta(t)\alpha^{\dag}.$$
Then $\nu'$ is proper and hence there is a map $\nu:\mathcal{L}(\Complex^{d},V)_{\infty}\wedge D(d)\to s(V)_{\infty}$. Moreover, $\nu'$ is surjective and hence $\nu$ is a quotient map.
\end{lem}
\begin{proof} The composition $\nu' \circ \eta': \mathcal{L}(\Complex^{d},V)\times D'(d)\to D'(d)$ is patently the
projection, and as previously observed $\mathcal{L}(\Complex^{d},V)$ is compact. Hence the inverse image of a compact set $C\subseteq D'(d)$ under this projection will be $\mathcal{L}(\Complex^{d},V)\times C$ which is compact. This is enough to show that the composition is proper and thus via Lemma \ref{propercompositionisproper} the map $\nu'$ is
proper. That $\nu'$ is surjective is a standard result of linear algebra following from the diagonalizability of all normal and hence all selfadjoint matrices. The rest of the result follows.
\end{proof}

We are now in a position to construct our first functional calculus variation.

\begin{prop}\label{fnalvarA} Let $X$ be a based space, let $f:D(d)\to  D(d)\wedge X$ be facial and let $V$ be a Hermitian
space of dimension $d$. There exists a unique map $\mathfrak{A}_{f}:s(V)_{\infty}\to s(V)_{\infty}\wedge X$ making the diagram below commute:
$$\xymatrix@C=1.5cm{\mathcal{L}(\Complex^{d},V)_{\infty}\wedge D(d)\ar[d]_{1\wedge f}\ar[r]^{\phantom{xxxx}\nu}&s(V)_{\infty}\ar[d]^{\mathfrak{A}_{f}}\ar[r]^{\eta}&D(d)\ar[d]^{f}\\
\mathcal{L}(\Complex^{d},V)_{\infty}\wedge D(d)\wedge
X\ar[r]_{\phantom{xxxx}\nu\wedge1}& s(V)_{\infty}\wedge
X\ar[r]_{\eta\wedge 1}&D(d)\wedge X}$$
Moreover:
\begin{itemize}
\item The map $\mathfrak{A}:\FMap (D(d),D(d)\wedge X)\to \Map (s(V)_{\infty},s(V)_{\infty}\wedge X)$ given by $f\mapsto \mathfrak{A}_{f}$ is continuous.
\item $\mathfrak{A}_{f}(\alpha)$ can be expressed explicitly for a given $\alpha$. Choose an orthonormal basis of
eigenvectors $v_{0},\ldots,v_{d-1}$ of $\alpha$ with eigenvalues $e_{0}\leqslant\ldots\leqslant e_{d-1}$; then if
$f(e_{0},\ldots,e_{d-1})=(s_{0},\ldots,s_{d-1})\wedge x$ we have $\mathfrak{A}_{f}(\alpha)=f(\alpha)\wedge x$ where $f(\alpha)$ denotes the endomorphism with eigenvectors $v_{i}$ and eigenvalues $s_{i}$.
\end{itemize}
\end{prop}
\begin{proof} We have a characterization of a possible map $\mathfrak{A}_{f}$ above, we first check that this characterization is reasonable, i.e. that it is possible to put a map in the place of $\mathfrak{A}_{f}$ to fill in the square. To do this we need to check that if $\nu(\alpha,t)=\nu(\alpha',t')$ then $\nu(\alpha,f(t))=\nu(\alpha',f(t'))$. From the condition that $\nu(\alpha,t)=\nu(\alpha',t')$ we have that $(\alpha^{\dag}\alpha')^{-1}\Delta(t)\alpha^{\dag}\alpha'=\Delta(t')$ which demonstrates that $t=t'$ and $\alpha^{\dag}\alpha'$ is in the centralizer of $\Delta(t)$. We now note that the centralizer of $\Delta(t)$ depends only on repeating values of $t_{i}$ in $t$. Let $f(t)=s\wedge x$, then $f$ is facial if and only if the centralizer of $\Delta(t)$ is a subgroup of the centralizer of $\Delta(s)$. Hence $\alpha^{\dag}\alpha'$ centralizes $\Delta(s)$ and from this it is easy to show that $\nu(\alpha,f(t))=\nu(\alpha',f(t'))$ as required. Thus having a map $\mathfrak{A}_{f}$ filling in the square is reasonable, and it is simple to observe that the characterization of $\mathfrak{A}_{f}$ in the statement of the proposition makes the diagram commute. We now wish to show that $\mathfrak{A}_{f}$ is continuous, that it is unique and that $f\mapsto \mathfrak{A}_{f}$ is continuous.

That $\mathfrak{A}_{f}$ is unique follows from $\nu$ being surjective, as noted in \ref{themapnu}. To show $\mathfrak{A}_{f}$ is continuous, we simply observe that $\mathfrak{A}_{f}\circ\nu$ is the continuous map $\nu\circ(1\wedge f)$ by commutativity. This is enough to prove that $\mathfrak{A}_{f}$ is continuous as $\nu$ is a quotient map - continuous compositions where one precomposes with a quotient automatically induce continuity in the postcomposition function.

Finally we wish to show that the map $\mathfrak{A}$ below is continuous:
$$\mathfrak{A}:\FMap (D(d),D(d)\wedge
X)\to \Map (s(V)_{\infty},s(V)_{\infty}\wedge X)$$
$$f\mapsto\mathfrak{A}_{f}.$$

We recall that we are working in $CGWH$ as detailed in $\S$\ref{onthecategories}. We have the below
adjunction: 
$$\xymatrix@R=0.001cm{\Map(\FMap (D(d),D(d)\wedge X), \Map (s(V)_{\infty},s(V)_{\infty}\wedge X))\\
\cong\\
\Map(\FMap (D(d),D(d)\wedge X)\wedge s(V)_{\infty},s(V)_{\infty}\wedge X)}$$

Therefore it is enough to show that the adjoint map $\mathfrak{A}^{\#}:\FMap (D(d),D(d)\wedge X)\wedge s(V)_{\infty}\to
s(V)_{\infty}\wedge X$ is continuous - here $\mathfrak{A}^{\#}$ is characterized by $\mathfrak{A}^{\#}(f,g)= \mathfrak{A}_{f}(g)$, how to build the adjoint in this case is explicitly detailed in \cite{SteenrodCategory}.

Let $\text{eval.}$ be the continuous map below:
$$\text{eval.}:\FMap (D(d),D(d)\wedge X)\wedge \mathcal{L}(\Complex^{d},V)_{\infty}\wedge D(d)\to \mathcal{L}(\Complex^{d},V)_{\infty}\wedge D(d)\wedge X$$
$$(f,\alpha,t)\mapsto (\alpha,f(t)).$$

We have the following commutative diagram:
$$\xymatrix{\FMap (D(d),D(d)\wedge X)\wedge \mathcal{L}(\Complex^{d},V)_{\infty}\wedge D(d)\ar[r]^{\phantom{xxxxxxxx}\text{eval.}}\ar[d]_{1\wedge\nu}&\mathcal{L}(\Complex^{d},V)_{\infty}\wedge D(d)\wedge X\ar[d]^{\nu\wedge 1}\\
\FMap (D(d),D(d)\wedge X)\wedge s(V)_{\infty}\ar[r]_{\phantom{xxxxxxxx}\mathfrak{A}^{\#}}&s(V)_{\infty}\wedge X}$$

This demonstrates that $\mathfrak{A}^{\#}\circ (1\wedge \nu)$ is continuous as it is equal to the continuous composition $(\nu\wedge 1)\circ\text{eval.}$. It follows that $\mathfrak{A}^{\#}$ is continuous as $(1\wedge \nu)$ is a quotient map. This proves that $\mathfrak{A}$ is continuous and hence proves the proposition.
\end{proof}

We can further extend this result by a simple restriction. The proof of the below is identical to the result above:

\begin{cor}\label{fnvarApositive} The maps given in the above lemmas \ref{propereigenvectors} and \ref{themapnu} factor in the following way:
$$\eta':s_{+}(V)\to D'_{+}(d)$$
$$\nu':\mathcal{L}(\Complex^{d},V)\times D'_{+}(d)\to s_{+}(V).$$

Thus if $f:D_{+}(d)\to D_{+}(d)\wedge X$ is facial then we have a unique map $\mathfrak{A}_{f}$ making the below diagram commute, such that $\mathfrak{A}$ is continuous and such that the characterization of $\mathfrak{A}_{f}$ coincides with the one given in Proposition \ref{fnalvarA}:
$$\xymatrix@C=1.5cm{\mathcal{L}(\Complex^{d},V)_{\infty}\wedge D_{+}(d)\ar[d]_{1\wedge f}\ar[r]^{\phantom{xxxx}\nu}&s_{+}(V)_{\infty}\ar[d]^{\mathfrak{A}_{f}}\ar[r]^{\eta}&D_{+}(d)\ar[d]^{f}\\
\mathcal{L}(\Complex^{d},V)_{\infty}\wedge D_{+}(d)\wedge
X\ar[r]_{\phantom{xxxx}\nu\wedge1}& s_{+}(V)_{\infty}\wedge
X\ar[r]_{\eta\wedge 1}&D_{+}(d)\wedge X}$$
\end{cor}

Thus we can build maps out of $s_{+}(V)_{\infty}$ by considering maps on a space modeling eigenvalues of selfadjoint maps. We now wish to extend this further one stage, building self-maps of spaces of homomorphisms.

\begin{lem}\label{rhoisproper} Let $W$ be Hermitian of dimension $d$ and $V$ Hermitian and such that $\Dim(V)\geqslant d$. The map $\rho:\Hom(W,V)\to s_{+}(V)$ as defined in \ref{themaprho} is a proper surjection and hence extends to a quotient map on compactifications.
\end{lem}
\begin{proof} The only issue is showing that $\rho$ is proper, but this follows from \ref{operatornormiseigenvalues} and the characterization of compact sets as closed and bounded sets from \ref{extendedheineborel}.
\end{proof}

\begin{lem}\label{themapmu} Again, let $W$ be Hermitian of dimension $d$ and $V$ Hermitian and such that $\Dim(V)\geqslant d$. We have the map $\mu'$ below:
$$\mu':\mathcal{L}(W,V)\times s_{+}(W)\to \Hom(W,V)$$
$$(\theta,\alpha)\mapsto -\theta\circ\alpha.$$
Then $\mu'$ is proper and thus there exists an extension $\mu:\mathcal{L}(W,V)_{\infty}\wedge s_{+}(W)_{\infty}\to S^{\Hom(W,V)}$. Moreover $\mu'$ is surjective and hence $\mu$ is a quotient map.
\end{lem}
\begin{proof} We note that $\mu'\circ\rho:\mathcal{L}(W,V)\times s_{+}(W)\to s_{+}(W)$ is the projection. This is proper due to $\mathcal{L}(W,V)$ being compact and by the same techniques used in the proof of \ref{themapnu}. Thus by Lemma \ref{propercompositionisproper} $\mu'$ is a proper map. The only other issue is its surjectivity. Let $\gamma\in\Hom(W,V)$. Then recalling $\sigma$ from \ref{themapsigma} we can extend $\sigma(\gamma)$ to an isometry $\varsigma(\gamma):W\to V$ by taking any choice of isometry $i:\Ker(\gamma)\to V\backslash \IM(\gamma)$ and setting $\varsigma(\gamma):=i\oplus\sigma(\gamma)$. Taking $(-\varsigma(\gamma),\rho(\gamma))\in\mathcal{L}(W,V)\times s_{+}(W)$, we note that by remarks in \ref{themapsigma} we will have $\mu'(-\varsigma(\gamma),\rho(\gamma))=\gamma$ and hence $\mu'$ is surjective. The rest of the result follows.
\end{proof}

\begin{prop}\label{fnalvarB} Let $X$ be a based space, $f:D_{+}(d)\to D_{+}(d)\wedge X$ be facial, $W$ Hermitian of dimension $d$ and $V$ Hermitian and such that $\Dim(V)\geqslant d$. Then there uniquely exists a map $\mathfrak{B}_{f}:S^{\Hom(W,V)}\to  S^{\Hom(W,V)}\wedge X$ making the diagram below commute:
$$\xymatrix@C=1.5cm{\mathcal{L}(W,V)_{\infty}\wedge s_{+}(W)_{\infty}\ar[d]_{1\times \mathfrak{A}_{f}}\ar[r]^{\phantom{xxxx}\mu}&S^{\Hom(W,V)}\ar[d]^{\mathfrak{B}_{f}}\ar[r]^{\rho}&s_{+}(W)_{\infty}\ar[d]^{\mathfrak{A}_{f}}\\
\mathcal{L}(W,V)_{\infty}\wedge s_{+}(W)_{\infty}\wedge
X\ar[r]_{\phantom{xxxx}\mu\wedge1}& S^{\Hom(W,V)}\wedge
X\ar[r]_{\rho\wedge 1}&s_{+}(W)_{\infty}\wedge X}$$
Moreover:
\begin{itemize}
\item  Let $F(\mathfrak{A})$ be the space of all maps $s_{+}(W)_{\infty}\to s_{+}(W)_{\infty}\wedge X$ of the form $\mathfrak{A}_{f}$ as defined in \ref{fnvarApositive}. Then $\mathfrak{B}:F(\mathfrak{A})\to \Map(S^{\Hom(W,V)},S^{\Hom(W,V)}\wedge X)$ given by $f\mapsto \mathfrak{B}_{f}$ is continuous.
\item $\mathfrak{B}_{f}(\gamma)$ can be expressed explicitly for a given $\gamma$. Choose an orthonormal basis of eigenvectors $v_{0},\ldots,v_{d-1}$ for $\gamma^{\dag}\gamma$ with eigenvalues $e_{0}^{2}\leqslant\ldots\leqslant e_{d-1}^{2}$ and such that $\gamma(v_{i})=e_{i}m_{i}$ with the $m_{i}$ being orthonormal in $V$. Then if $f(e_{0},\ldots,e_{d-1})=(s_{0},\ldots,s_{d-1})\wedge x$ we have $\mathfrak{B}_{f}(\gamma)=f(\gamma)\wedge x$, where $f(\gamma)$ denotes the homomorphism sending each $v_{i}$ to $s_{i}m_{i}$.
\end{itemize}
\end{prop}
\begin{proof} Let $\mu(\theta,\alpha)=\mu(\theta',\alpha')$. We wish to show that adding a $\mathfrak{B}_{f}$ would be internally consistent, i.e. that $\mu(\theta,\mathfrak{A}_{f}(\alpha))=\mu(\theta',\mathfrak{A}_{f}(\alpha'))$. First note that if we restrict ourselves to looking at the case of strictly positive elements of $s(W)$ by restricting $\mu$ to a map out of $s_{++}(W)$, then we have that $\mu$ is injective for the same reasons as to why injectivity holds in \ref{thetaEalpha}. This thus means there are no issues, $\mu(\theta,\mathfrak{A}_{f}(\alpha))=\mu(\theta',\mathfrak{A}_{f}(\alpha'))$ as required. Hence we turn our attention to the case where $\alpha$ and $\alpha'$ both have eigenvalues equal to zero. It is clear that in this case $\mu(\theta,\alpha)$ being equal to $\mu(\theta',\alpha')$ indicates that $\alpha=\alpha'$ but also that $\theta$ and $\theta'$ may not agree on $\Ker(\alpha)$ - they need only match on the image of $\alpha$. This is all we need, however,
as $f$ being facial implies in particular that $\Ker(\alpha)\subseteq\Ker(\mathfrak{A}_{f}(\alpha))$ and thus that
$\mu(\theta,\mathfrak{A}_{f}(\alpha))=\mu(\theta',\mathfrak{A}_{f}(\alpha'))$. This implies that a map $\mathfrak{B}_{f}$ would therefore be reasonable and internally consistent.

We now define a map $\mathfrak{B}_{f}$ given by the description above - this commutes when fitted into the diagram. As in the earlier proof $\mathfrak{B}_{f}$ is unique by the surjectivity of $\mu$. Also as before $\mathfrak{B}_{f}$ fitting into the commutative diagram implies continuity as $\mu$ is a quotient map, thus $\mathfrak{B}_{f}\circ\mu$ being continuous implies $\mathfrak{B}_{f}$ is.

Finally, to prove $\mathfrak{B}$ is continuous we again work with the adjunction in $CGWH$. We wish to show $\mathfrak{B}:F(\mathfrak{A})\to \Map(S^{\Hom(W,V)},S^{\Hom(W,V)}\wedge X)$ is continuous. We have an adjunction:
$$\xymatrix@R=0.001cm{\Map(F(\mathfrak{A}),\Map(S^{\Hom(W,V)},S^{\Hom(W,V)}\wedge X))\\
\cong \\
\Map(F(\mathfrak{A})\wedge S^{\Hom(W,V)},S^{\Hom(W,V)}\wedge X)}$$

Hence continuity of $\mathfrak{B}$ follows from continuity of the adjoint $\mathfrak{B}^{\#}:F(\mathfrak{A})\wedge S^{\Hom(W,V)}\to S^{\Hom(W,V)}\wedge X$. Let $\text{eval.}$ be the continuous map below:
$$\text{eval.}:F(\mathfrak{A})\wedge\mathcal{L}(W,V)_{\infty}\wedge s_{+}(W)_{\infty}\to\mathcal{L}(W,V)_{\infty}\wedge s_{+}(W)_{\infty}\wedge X$$
$$(\mathfrak{A}_{f},\theta,\alpha)\mapsto
(\theta,\mathfrak{A}_{f}(\alpha)).$$

We have the following commutative diagram:
$$\xymatrix{F(\mathfrak{A})\wedge\mathcal{L}(W,V)_{\infty}\wedge s_{+}(W)_{\infty}\ar[r]^{\text{eval.}}\ar[d]_{1\wedge\mu}&\mathcal{L}(W,V)_{\infty}\wedge s_{+}(W)_{\infty}\wedge X\ar[d]^{\mu\wedge 1}\\
F(\mathfrak{A})\wedge S^{\Hom(W,V)}\ar[r]_{\mathfrak{B}^{\#}}&S^{\Hom(W,V)}\wedge X}$$

Thus by the diagram $\mathfrak{B}^{\#}\circ (1\wedge\mu)$ is continuous. The map $\mu$ is a quotient and hence so is
$(1\wedge\mu)$. It follows that $\mathfrak{B}^{\#}$ and hence $\mathfrak{B}$ is continuous - this is enough to give us the result.
\end{proof}

This builds the required variation of functional calculus - we will tend to take the standard case where $X$ is $S^{0}$, thus letting us look at self maps of $D(d)$ or $D_{+}(d)$. We now mention one further construction as a way of building maps from the very simple starting level of self-maps of $D_{+}(2)$. Let $f:D_{+}(2)\to D_{+}(2)$ be facial. Then $f(t_{0},t_{1})$ can be expressed in the form $(g(t_{0},t_{1}),g(t_{0},t_{1})+h(t_{0},t_{1}))$ for some $g,h$
landing in the positive reals and such that $h(t,t)=0$. We set $\hat{f}:D_{+}(d+1)\to D_{+}(d+1)$ to be given by
$\hat{f}(t_{0},\ldots,t_{d})_{i}=g(t_{0},t_{d})+\frac{t_{i}-t_{0}}{t_{d}-t_{0}}h(t_{0},t_{d})$ where this makes sense, i.e. where $t_{0}<t_{d}$ and $\hat{f}(t_{0},\ldots,t_{d})_{i}=g(t_{0},t_{d})$ when $t_{0}=t_{d}$; we map $\hat{f}(\infty)$ to $\infty$. To make this work we need the following lemma:

\begin{lem}\label{MakingHatWork} $\hat{f}$ is continuous, $f\mapsto \hat{f}$ gives a continuous map from facial maps on $D_{+}(2)$ to facial maps on $D_{+}(d+1)$.
\end{lem}
\begin{proof} We first check that if $f$ is facial then so is $\hat{f}$. We check to see that $\hat{f}$ preserves $D_{0}(d)$, i.e. that $\hat{f}(0,\ldots,t_{d})_{0}=0$. This is equivalent to checking that $g(0,t_{d})+\frac{0-0}{t_{d}-0}h(0,t_{d})$ is zero - the other case is simply $g(0,0)=0$ as required. This, however, is $g(0,t_{d})$ which is $0$ as $f$ is facial. Next we require that $\hat{f}$ preserves $F_{i}(D_{+}(d))$. Let $t_{i}=t_{i+1}=t$, we now wish to check that $g(t_{0},t_{d})+\frac{t_{i}-t_{0}}{t_{d}-t_{0}}h(t_{0},t_{d})$ is
equal to $g(t_{0},t_{d})+\frac{t_{i+1}-t_{0}}{t_{d}-t_{0}}h(t_{0},t_{d})$ but this is trivial, as is the other case where $g(t_{0},t_{d})$ is patently equal to $g(t_{0},t_{d})$. It follows that $\hat{f}$ is facial.

Continuity of $\hat{f}$ is clear away from $t_{d}-t_{0}=0$, it is inherited from the continuity of $f$ and by extension of $g$ and $h$. We have $(t_{i}-t_{0})/(t_{d}-t_{0})\in[0,1]$, near $t_{d}-t_{0}=0$ we also have $h(t_{0},t_{d})$ approaching $0$ and hence a limiting argument demonstrates continuity at the point $t_{d}-t_{0}=0$.

We now wish to check that the map $\text{hat}:F_{+}(2)\to F_{+}(d+1)$ is continuous. We have an adjunction given
below:
$$\xymatrix@R=0.001cm{\Map(F_{+}(2),\Map(D_{+}(d+1),D_{+}(d+1)))\\
\cong\\
\Map(F_{+}(2)\wedge D_{+}(d+1),D_{+}(d+1))}$$

We observe that $F_{+}(d+1)$ is actually a subspace of $\Map(D_{+}(d+1),D_{+}(d+1))$, thus from the above adjunction we can show $\text{hat}$ is continuous if we show the adjoint $\text{hat}^{\#}$ is. Let $\Delta_{d-1}$ be the standard $(d-1)$-simplex which we take to be parameterized by $d-2$ increasing coordinates in $[0,1]$. Define $\lambda'$ as follows:
$$\lambda':D_{+}'(2)\times \Delta_{d-1}\to D_{+}'(d)$$
$$(t_{0},t_{1},s_{1},\ldots,s_{d-1})\mapsto (t_{0},t_{0}+s_{1}(t_{1}-t_{0}),\ldots,t_{0}+s_{d-1}(t_{1}-t_{0}),t_{1}).$$

The map $\lambda'$ is easily demonstrated to be proper and a surjection, hence $\lambda$, the map on compactifications, is a quotient. Let $\text{eval.}$ be the map below:
$$\text{eval.}:F_{+}(2)\wedge D_{+}(2)\wedge (\Delta_{d-1})_{\infty}\to D_{+}(2)\wedge (\Delta_{d-1})_{\infty}$$
$$(f,t,s)\mapsto (f(t),s).$$
We have the following commutative diagram:
$$\xymatrix{F_{+}(2)\wedge D_{+}(2)\wedge (\Delta_{d-1})_{\infty}\ar[r]^{\phantom{xx}\text{eval.}}\ar[d]_{(1\wedge \lambda)}&D_{+}(2)\wedge (\Delta_{d-1})_{\infty}\ar[d]^{\lambda}\\
F_{+}(2)\wedge D_{+}(d+1)\ar[r]_{{\phantom{xx}\text{hat}^{\#}}}&D_{+}(d+1)}$$
As before, this gives us that $\text{hat}^{\#}\circ (1\wedge \lambda)$ is continuous, hence so is $\text{hat}^{\#}$ and hence so is $\text{hat}$. This completes the proof.
\end{proof}

We now wish to use this variation to build certain NDR pairs that will be used in the later work - we do this in the next section.

\section{Building NDR Pairs Using Functional Calculus}\label{BuildingNDRPairsUsingFunctionalCalculus}

We now use the theory from the above section to build $(S^{\Hom(W,V)},\inj(W,V)^{c}_{\infty})$ into an NDR pair for
Hermitian spaces $V$ and $W$, with $\Dim(W)=d+1$ and $\Dim(V)\geqslant d+1$. This will be of relevance in Chapter \ref{ch:ch5} when we come to calculate the homotopy cofibres needed to prove the main theorem; we shall demonstrate that the sequences we build are isomorphic to the cofibre sequence generated from $(S^{\Hom(W,V)},\inj(W,V)^{c}_{\infty})$. We start with the following two results:

\begin{prop}\label{unitdiscNDR} Let $X$ be the upper half disc $\{z\in\Complex:|z|\leqslant 1, \IM(z)\geqslant 0\}$ and let $Y$ be  the upper semicircle $\{z\in X:|z|=1\}$. Then $(X,Y)$ is an NDR pair via:
$$u''(re^{i\theta}):=\min(1,2-2r)$$
$$h_{t}''(re^{i\theta}):=\min(1,(2-t)r)e^{i\theta}.$$
\end{prop}
\begin{proof} We check conditions $1$ to $6$ from Definition \ref{NDR}. The continuity of $u''$ and $h_{t}''$ is clear and hence $1$ and $2$ hold. The map $h_{1}''$ is evaluated as $h_{1}''(re^{i\theta})=\min(1,(2-1)r)e^{i\theta}=re^{i\theta}$, from this it is clear that $3$ holds. If $z\in Y$ then $z$ is of the form $e^{i\theta}$. Thus $h_{t}''(z)=\min(1,2-t)e^{i\theta}=e^{i\theta}$ and we observe that $4$ holds. Regarding condition $5$, if $u''(z)<1$ then $2-2r<1$, i.e. $r>1/2$. In this case $h_{0}''(z)=\min(1,2r)e^{i\theta}=e^{i\theta}$, thus $h_{0}''(z)$ is in $Y$ as required. Finally, let $z$ be such that $u''(z)=0$, i.e. $2-2r=0$. This occurs if and only if $r=1$, showing that $z$ is in $Y$ as required for condition $6$ to be satisfied; thus $(X,Y)$ is an NDR pair.
\end{proof}

\begin{prop}\label{theconformalmapphi} The map $\phi$ as defined below gives us a relative homeomorphism $(D_{+}(2),D_{0}(2))\cong(X,Y)$:
$$\phi(t_{0},t_{1})=\frac{i-(t_{1}+it_{0})^{2}}{i+(t_{1}+it_{0})^{2}}.$$
\end{prop}
\begin{proof} This map can be built up using basic complex analysis. We first consider the map out of $D_{+}(2)$ given by $(t_{0},t_{1})\mapsto (t_{1},t_{0})$. Denoting this map by $z\mapsto \dot{z}$, this gives us the homeomorphism below, the bold line representing $D_{0}(2)$:

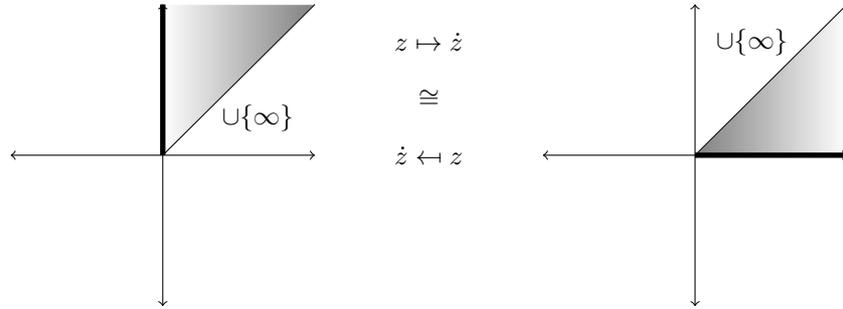
\begin{tikzpicture}
\shadedraw[left color=white, right color=gray, draw=white] (0,0) -- (2,2) -- (0,2) -- cycle;
\draw[<->](-2,0) -- (2,0);
\draw[<->](0,-2) -- (0,2);
\draw[line width=2pt](0,0) -- (0,2);
\draw(0,0) -- (2,2);
\draw(1.25,0.5) node {$\cup\{\infty\}$};
\shadedraw[left color=gray, right color=white, draw=white] (7,0) -- (9,2) -- (9,0) -- cycle;
\draw[<->](5,0) -- (9,0);
\draw[<->](7,-2) -- (7,2);
\draw[line width=2pt](7,0) -- (9,0);
\draw(7,0) -- (9,2);
\draw(7.75,1.5) node {$\cup\{\infty\}$};
\draw(3.5,0.75) node {$\cong$};
\draw(3.5,1.5) node {$z\mapsto \dot{z}$};
\draw(3.5,0) node {$\dot{z} \mapsfrom z$};
\end{tikzpicture}

We now apply the map $z\mapsto z^{2}$ to build another homeomorphism:

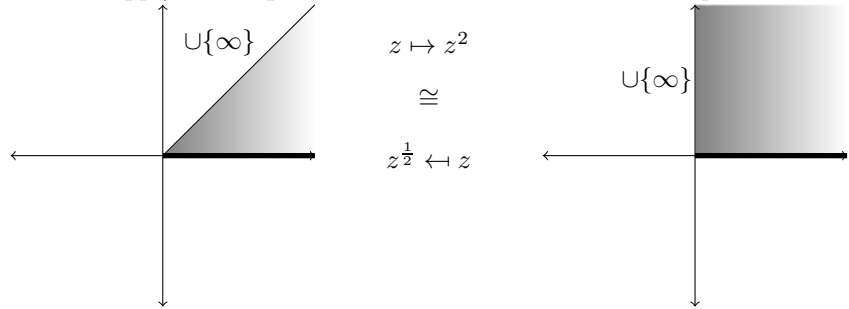
\begin{tikzpicture}
\shadedraw[left color=gray, right color=white, draw=white] (0,0) -- (2,2) -- (2,0) -- cycle;
\draw[<->](-2,0) -- (2,0);
\draw[<->](0,-2) -- (0,2);
\draw[line width=2pt](0,0) -- (2,0);
\draw(0,0) -- (2,2);
\draw(0.75,1.5) node {$\cup\{\infty\}$};
\shadedraw[left color=gray, right color=white, draw=white] (7,0) -- (9,0) -- (9,2) -- (7,2) -- cycle;
\draw[<->](5,0) -- (9,0);
\draw[<->](7,-2) -- (7,2);
\draw[line width=2pt](7,0) -- (9,0);
\draw(6.5,1) node {$\cup\{\infty\}$};
\draw(3.5,0.75) node {$\cong$};
\draw(3.5,1.5) node {$z\mapsto z^{2}$};
\draw(3.5,0) node {$z^{\frac{1}{2}} \mapsfrom z$};
\end{tikzpicture}

Finally, we apply the map sending $z\mapsto (i-z)/(i+z)$ and the basepoint $\infty$ to $-1$. This is a homeomorphism with codomain $X$ and inverse $z\mapsto i(1-z)/(1+z)$; note this sends $-1$ to $\infty$. The bold line is sent to $Y$, as
observed below.

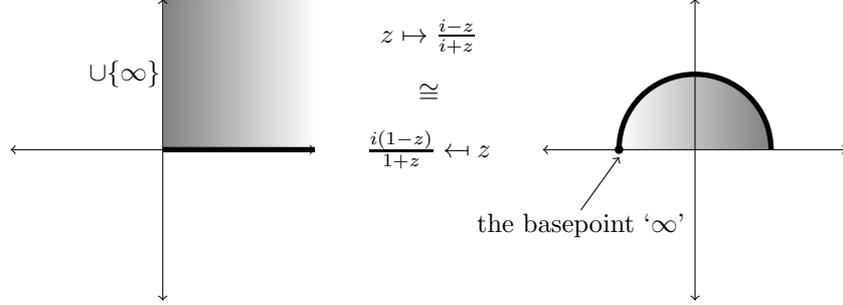
\begin{tikzpicture}
\shadedraw[left color=gray, right color=white, draw=white] (0,0) -- (2,0) -- (2,2) -- (0,2) -- cycle;
\draw[<->](-2,0) -- (2,0);
\draw[<->](0,-2) -- (0,2);
\draw[line width=2pt](0,0) -- (2,0);
\draw(-0.5,1) node {$\cup\{\infty\}$};
\shadedraw[left color=white, right color=gray, draw=white] (8,0) arc (0:180:1) -- cycle;
\draw[<->](5,0) -- (9,0);
\draw[<->](7,-2) -- (7,2);
\draw[line width=2pt](8,0) arc (0:180:1);
\filldraw[black](6,0) circle (0.5mm);
\draw(5.5,-1) node {the basepoint `$\infty$'};
\draw[->](5.5,-0.8) -- (6,-0.1);
\draw(3.5,0.75) node {$\cong$};
\draw(3.5,1.5) node {$z\mapsto \frac{i-z}{i+z}$};
\draw(3.5,0) node {$\frac{i(1-z)}{1+z} \mapsfrom z$};
\end{tikzpicture}

The above composition gives a homeomorphism $(D_{+}(2),D_{0}(2))\cong(X,Y)$. It is trivial to observe that this
composition can be explicitly written down as $\phi$ stated above. Finally, we note that this maps $F_{0}(D_{+}(2))$ homeomorphically onto the base of $X$, the set of $z\in X$ such that $\IM(z)=0$.
\end{proof}

This now allows us to build an NDR pair out of $(D_{+}(2),D_{0}(2))$ in an obvious way:

\begin{prop}\label{D2NDRpair} $(D_{+}(2),D_{0}(2))$ is an NDR pair via:
$$u'(t_{0},t_{1}):=u''\circ\phi(t_{0},t_{1})$$
$$h_{t}'(t_{0},t_{1}):=\phi^{-1}\circ h_{t}''\circ \phi(t_{0},t_{1}).$$
\end{prop}
\begin{proof} We again check the points from Definition \ref{NDR}. As $\phi$ is a homeomorphism we are considering continuous maps and thus $1$ and $2$ hold. That $h_{1}''$ is the identity on $X$ trivially implies that $h_{1}'$ is the identity on $D_{+}(2)$ and thus condition $3$ also holds. If $y\in D_{0}(2)$ then from above $\phi(y)\in Y$ and thus
$h_{t}''\circ\phi(y)\in Y$. This then shows $h_{t}'(y)\in D_{0}(2)$ as $\phi^{-1}$ maps $Y$ homeomorphically to
$D_{0}(2)$, proving that $4$ holds. If $u'(t)<1$ then $2-2|\phi(t)|<1$ or $|\phi(t)|>1/2$; thus $h_{0}''\circ\phi(t)$ is in $Y$ and thus $h_{0}'(t)\in D_{0}(2)$ so $5$ is satisfied. Finally, if $u'(t)=0$ then $|\phi(t)|=1$, i.e. $\phi(t)\in Y$ and thus $t\in D_{0}(2)$ as required for condition $6$ to hold. Thus $(D_{+}(2),D_{0}(2))$ is an NDR pair.
\end{proof}

Finally, we use this combined with the functional calculus variation above to prove the following, constructing the needed NDR pair and cofibre sequence:

\begin{prop}\label{homNDRpair} $(S^{\Hom(W,V)},\inj(W,V)^{c}_{\infty})$ is an NDR pair via:
$$u(\gamma):=u'(e_{0}(\rho(\gamma)),e_{d}(\rho(\gamma)))$$
$$h_{t}:=\mathfrak{B}_{\widehat{h'_{t}}}.$$
\end{prop}
\begin{proof} Again we refer to Definition \ref{NDR}. Continuity of $h_{t}$ follows from the continuity of the constructions \ref{fnalvarB} and \ref{MakingHatWork}, while continuity of $u$ follows from Lemma \ref{continuouseigenvalues}, Corollary \ref{themaprho} and the continuity of $u'$; thus $1$ and $2$ hold. Now let $\gamma:W\to V$ be a homomorphism, we wish to show that $\mathfrak{B}_{\widehat{h'_{1}}}(\gamma)=\gamma$. Note that 
$h_{1}'(t_{0},t_{1})=(t_{0},t_{1})$. Thus applying construction \ref{MakingHatWork} gives us $\widehat{h'_{1}}(t_{0},\ldots,t_{d})_{i}=t_{0}+(\frac{t_{i}-t_{0}}{t_{d}-t_{0}})(t_{d}-t_{0})=t_{i}$. Thus by the explicit expression of $\mathfrak{B}_{\widehat{h'_{1}}}$ from Proposition \ref{fnalvarB} we have $\mathfrak{B}_{\widehat{h'_{1}}}(\gamma)=\gamma$, and condition $3$ is satisfied.

For condition $4$, let $\gamma$ be non-injective. We wish to show that $\mathfrak{B}_{\widehat{h'_{t}}}(\gamma)=\gamma$ for all $t$. By Proposition \ref{fnalvarB} there is an orthonormal basis of eigenvectors $v_{0},\ldots,v_{d-1}$ for $\gamma^{\dag}\gamma$ with eigenvalues $e_{i}^{2}$ such that $\gamma(v_{i})=e_{i}m_{i}$ for some $m_{i}$ orthonormal in $W$. As $\gamma$ is non-injective it follows that $e_{0}=0$. Then if $\widehat{h'_{t}}(\underline{e})=\underline{s}$ we have $\mathfrak{B}_{\widehat{h'_{t}}}(\gamma)(v_{i})=s_{i}m_{i}$. We claim that $\widehat{h'_{t}}(\underline{e})=\underline{e}$, which would be sufficient to prove that $4$ holds.

We can write $h_{t}'(t_{0},t_{1})$ in the form $(g(t_{0},t_{1}),g(t_{0},t_{1})+h(t_{0},t_{1}))$ for some $g$ and $h$. First consider the case where $e_{i}=0$ for all $i$. If so, then $(\widehat{h'_{t}})_{i}(\underline{0})=g(0,0)$ by definition. We have $\phi(0,0)=1$ and $h_{t}''(1)=1$ so $h_{t}'(0,0)=(0,0)$. Thus $g(0,0)=0$ and we have $(\widehat{h'_{t}})_{i}(\underline{0})=0$.

Now consider the case where $e_{d}\neq0$. Then $(\widehat{h'_{t}})_{i}(\underline{e})=g(0,e_{d})+h(0,e_{d})e_{i}/e_{d}$.
We claim that $g(0,e_{d})=0$ and $h(0,e_{d})=e_{d}$, i.e. that $h_{t}'(0,e_{d})=(0,e_{d})$. As $\phi$ maps $D_{0}(2)$ to $Y$ we have that $h_{t}''(\phi(0,e_{d}))=\phi(0,e_{d})$ and thus that $h_{t}'(0,e_{d})=(0,e_{d})$ as required, proving that $4$ holds.

For condition $5$, let $\gamma$ be such that $u(\gamma)<1$. We wish to show that $\mathfrak{B}_{\widehat{h'_{0}}}(\gamma)$ is non-injective, i.e. that if $v_{i}$ is an orthonormal basis for $\gamma^{\dag}\gamma$ with eigenvalues $e_{i}^{2}$ then $(\widehat{h'_{0}})_{0}(\underline{e})=0$. If all $e_{i}=0$ this is trivial from above, if not then observe that $e_{0}=e_{0}(\rho(\gamma))$ and $e_{d}=e_{d}(\rho(\gamma))$ so by the assumption that $u(\gamma)<1$ we have $u'(e_{0},e_{d})<1$ and therefore $h'_{0}(e_{0},e_{d})\in D_{0}(2)$ by Proposition \ref{D2NDRpair}. Thus $h_{0}'(e_{0},e_{d})=(0,h(e_{0},e_{d}))$ for some $h$, so $(\widehat{h'_{0}})_{0}(\underline{e})=0+h(e_{0},e_{d})(e_{0}-e_{0})/(e_{d}-e_{0})=0$ as required, hence condition $5$ holds.

Finally we consider condition $6$. Let $\gamma$ be such that $u(\gamma)=0$. We wish to show that $\gamma$ is non-injective. Note $u(\gamma)=u'(e_{0}(\rho(\gamma)),e_{d}(\rho(\gamma)))=u''\circ\phi(e_{0}(\rho(\gamma)),e_{d}(\rho(\gamma)))$. This is zero if and only if $2-2|\phi(e_{0}(\rho(\gamma)),e_{d}(\rho(\gamma)))|=0$ which occurs when $|\phi(e_{0}(\rho(\gamma)),e_{d}(\rho(\gamma)))|=1$. By the definition of $\phi$, this implies that $(e_{0}(\rho(\gamma)),e_{d}(\rho(\gamma)))\in D_{0}(2)$ which in turn shows that $e_{0}(\rho(\gamma))=0$ which proves that $\gamma$ is non-injective as required. Thus $(S^{\Hom(W,V)},\inj(W,V)^{c}_{\infty})$ is an NDR pair.
\end{proof}

\begin{cor}\label{homcofibseq} The following is a cofibre sequence, with $i$, $p$ and $e$ being as defined from Proposition \ref{NDRCofibreseq}:
$$\inj(W,V)^{c}_{\infty}\overset{i}{\to}S^{\Hom(W,V)}\overset{p}{\to}\inj(W,V)_{\infty}\overset{e}{\to}\Sigma\inj(W,V)^{c}_{\infty}.$$
\end{cor}

We will use this sequence throughout the rest of the document.

\section{The Homotopy Type of Spaces of Maps in Functional Calculus} \label{TheHomotopyTypeofCertainMapsinFunctionalCalculus}

To prove the theorem stated at the beginning of the next chapter we will show that other sequences we construct are
isomorphic to the cofibre sequence constructed at the end of the previous section; key to this will be showing that certain other maps are homotopic to the map $e$ in Corollary \ref{homcofibseq}. These other maps, however, will also turn out to arise from the construction $\mathfrak{B}_{f}$, thus we wish to explore the homotopy type of maps arising from our functional calculus constructions. We start out by making the following observation, the proof is clear:

\begin{lem}\label{homotopiesthroughfnalcalc} Let $f$ and $g$ be facial maps such that $f\simeq g$ through a facial homotopy $h_{t}$. Then $\mathfrak{A}_{f}\simeq\mathfrak{A}_{g}$ via $\mathfrak{A}_{h_{t}}$ and $\mathfrak{B}_{f}\simeq\mathfrak{B}_{g}$ via $\mathfrak{B}_{h_{t}}$.
\end{lem}

Thus we want to look at the facial homotopy type of facial maps. We look at the basic case, that of facial self-maps $D(d)\to D(d)$. We first easily observe the below homeomorphisms:

\begin{lem}\label{itsalljustspheresandballs} $D(1)\cong S^{1}$ and $D(d)\cong B^{d}$ for $d>1$. Moreover $F_{i}(D(d))\cong D(d-1)$ via the homeomorphism $(t_{0},\ldots,t_{d-1})\mapsto (t_{0},\ldots,t_{i-1},t_{i+1},\ldots,t_{d-1})$ and we deduce that each face is also homeomorphic to either $B^{n}$ for some $n$ or $S^{1}$.
\end{lem}

We now consider this more generally, looking at intersections of faces.

\begin{defn}\label{faceintersectionBsigma} Let $\sigma\subseteq \{0,\ldots,d-2\}$. Then set $\bar{B}_{\sigma}$ to be the intersection of faces $\bigcap_{i\notin \sigma}F_{i}(D(d))$.
\end{defn}

\begin{rem}\label{remarksonintersectionsoffaces} If $\sigma\cup\{d-1\}=\{i_{0},\ldots,i_{m-1}\}$ then the map $(t_{0},\ldots,t_{d-1})\mapsto (t_{i_{0}},\ldots,t_{i_{m-1}})$ is a homeomorphism between $\bar{B}_{\sigma}$ and $D(m)$. Hence $\bar{B}_{\emptyset}\cong S^{1}$ and $\bar{B}_{\sigma}\cong B^{|\sigma|+1}$ for $\sigma$ nonempty.
\end{rem}

Thus intersections of faces are balls unless we consider the intersection of all faces, which is a sphere embedded in $D(d)$. Moreover, facial maps by definition preserve $\bar{B}_{\sigma}$ for each $\sigma$. This suggests the following claim:

\begin{prop}\label{facialhomotopiesviadegrees} Let $f,g:D(d)\to D(d)$ be facial and such that $f$ and $g$ have the same
degree on $\bar{B}_{\emptyset}$. Then $f\simeq g$ through facial maps.
\end{prop}

To prove this we choose to induct on some concept of dimension. To make this rigourous we define the following subspaces of $D(d)$:

\begin{defn}\label{bk} Let $\bar{B}[k]$ be the union of all $\bar{B}_{\sigma}$ with $|\sigma|\leqslant k$. We say that a self map of $\bar{B}[k]$ is facial if it preserves each $\bar{B}_{\sigma}$.
\end{defn}

Note that $\bar{B}[0]=\bar{B}_{\emptyset}$ and $\bar{B}[d-1]=D(d)$. We now proceed to prove the below lemma. Note here that should this result hold then Proposition \ref{facialhomotopiesviadegrees} will follow by induction on the dimension of $\bar{B}[k]$ - assumption that $f$ and $g$ have the same degree on $\bar{B}_{\emptyset}$ gives us the assumption that they are homotopic on $\bar{B}[0]$ and thus homotopic via facial maps on $\bar{B}[0]$ as the only facial
intersection contained in $\bar{B}[0]$ is $\bar{B}_{\emptyset}=\bar{B}[0]$ itself.

\begin{lem}\label{facialhomotopiesinduction} Let $f,g:D(d)\to D(d)$ be facial maps such that $f|_{\bar{B}[k]}$ is homotopic to $g|_{\bar{B}[k]}$ through facial maps. Then this homotopy can be extended to a facial homotopy $f|_{\bar{B}[k+1]}\simeq g|_{\bar{B}[k+1]}$.
\end{lem}

To prove this, we need another small lemma and brief corollary.

\begin{lem}\label{extendingboundariesofballs} Suppose that $X$ is homeomorphic to a ball of dimension $n+1$ and $Y$ homeomorphic to a ball of dimension $n$, $X\cong B^{n+1}$ and $Y\cong B^{n}$, and let $p$ be a map $\partial X\to Y$. Then there exists an extension $\tilde{p}:X\to Y$ of $p$.
\end{lem}
\begin{proof} Without loss of generality we can assume that $X=B^{n+1}$ and $Y=B^{n}$. In this case, we can thus parameterize $X$ by coordinates $(x,t)$ for $x$ a point on the boundary and $t$ a scalar. Under
this, we build the required extension $\tilde{p}$ via $\tilde{p}(x,t):=tp(x)$.
\end{proof}

\begin{cor}\label{extendingballshomotopies} Suppose $Y\cong B^{n}$ and $f,g:Y\to Y$ are maps such that $h:[0,1]\times\partial Y\to Y$ gives a homotopy from $f|_{\partial Y}$ to $g|_{\partial Y}$. Then there exists an extension $\tilde{h}:[0,1]\times Y\to Y$ that provides a homotopy between $f$ and $g$.
\end{cor}
\begin{proof} Set $X:=[0,1]\times Y$ then it is a trivial exercise to show that $X$ is homeomorphic to $B^{n+1}$. Now define a map $p:\partial X\to Y$ in the following way. Note that $\partial X$ is $([0,1]\times \partial Y)\cup
(\{0,1\}\times Y)$ - set $p$ therefore to be $h$ on $[0,1]\times \partial Y$, $f$ on $\{0\}\times Y$ and $g$ on $\{1\}\times Y$; note this is consistent. Then use Lemma \ref{extendingboundariesofballs} to extend $p$ to a map
$\tilde{h}:[0,1]\times Y\to Y$ giving the required homotopy.
\end{proof}

We now have enough machinery to prove Lemma \ref{facialhomotopiesinduction}. Let $f$ and $g$ be the two maps and
set $h_{k}$ to be the facial homotopy $[0,1]\times \bar{B}[k]\to \bar{B}[k]$ agreeing with $f$ on $0$ and $g$ on $1$. Now let $\bar{B}_{\sigma}$ be such that $|\sigma|=k+1$. Then we have that $\bar{B}_{\sigma}$, a $(k+2)$-ball, is contained within $\bar{B}[k+1]$ while its boundary $\partial \bar{B}_{\sigma}$, a $(k+1)$-sphere, is contained
within $\bar{B}[k]$. Restrict $f$ and $g$ to two maps $f|_{\bar{B}_{\sigma}}$ and $g|_{\bar{B}_{\sigma}}$ and restrict $h_{k}$ to a homotopy $f|_{\partial \bar{B}_{\sigma}}\simeq g|_{\partial \bar{B}_{\sigma}}$. This extends to give a homotopy $h_{k+1,\sigma}:[0,1]\times \bar{B}_{\sigma}\to \bar{B}_{\sigma}$ via Corollary \ref{extendingballshomotopies} which agrees with $h_{k}$ on the boundary, $f|_{\bar{B}_{\sigma}}$ on $0$ and $g|_{\bar{B}_{\sigma}}$ on $1$. Moreover, performing this for all face intersections of dimension $k+1$ gives us a family of maps $h_{k+1,\sigma}$. If $\bar{B}_{\sigma}$ and $\bar{B}_{\tau}$ are two different face intersections of this type, then note that $\bar{B}_{\sigma}\cap \bar{B}_{\tau}\subset \bar{B}[k]$. Thus the two homotopies $h_{k+1,\sigma}$ and $h_{k+1,\tau}$ agree on the intersection as they are both extensions of $h_{k}$ out of $\bar{B}[k]$. Therefore we can patch these maps together to get a homotopy $h_{k+1}:[0,1]\times \bigcup \bar{B}_{\sigma}\to \bigcup \bar{B}_{\sigma}$. Noticing
that $\bigcup \bar{B}_{\sigma}=\bar{B}[k+1]$ allows us to write this as $h_{k+1}:[0,1]\times \bar{B}[k+1]\to \bar{B}[k+1]$, a homotopy between $f$ and $g$ extending $h_{k}$. Finally, we need to check that this is facial, but this follows instantly for face intersections of dimension up to $k$ as $h_{k}$ is facial and furthermore follows for face intersections of dimension $k+1$ via the method of construction of patching together homotopies $h_{k+1,\sigma}$ which satisfy $h_{k+1,\sigma}([0,1]\times \bar{B}_{\sigma})\subseteq \bar{B}_{\sigma}$. Thus we have extended the homotopy, proving Lemma \ref{facialhomotopiesinduction} as required and thus proving Proposition \ref{facialhomotopiesviadegrees} via induction.

Hence we now have a criterion for when $f\simeq g:D(d)\to D(d)$. There are induced maps $f':S^{1}\to S^{1}$ and $g':S^{1}\to S^{1}$ given by $t\mapsto f(t,\ldots,t)$ and $t\mapsto g(t,\ldots,t)$. If these maps have the same degree then $f\simeq g$ through facial maps. Moreover, we can apply this to more specific cases.

In the last section we defined a map $e:\inj(W,V)_{\infty}\to\Sigma \inj(W,V)^{c}_{\infty}$ as $\mathfrak{B}_{f}$ for a certain facial map $f:D_{+}(d)\to \Sigma D_{+}(d)$. We want to apply the homotopical information above to this case. We start with the below lemma, the proof is clear:

\begin{lem}\label{factoringthroughinfnalcalc} Let $f:D_{+}(d)\to \Sigma D_{+}(d)$ be facial, consider the induced map $\mathfrak{B}_{f}:S^{\Hom(W,V)}\to S^{\Real\oplus\Hom(W,V)}$. Then $\mathfrak{B}_{f}$ factors through
$\inj(W,V)_{\infty}\to\Sigma \inj(W,V)^{c}_{\infty}$ if and only if $f$ factors through $D_{+}(d)/D_{0}(d)\to\Sigma
D_{0}(d)$.
\end{lem}

Hence we want to look at facial maps $D_{+}(d)/D_{0}(d)\to\Sigma D_{0}(d)$. This is equivalent to the above work via the two homeomorphisms below, the proof is standard to check:

\begin{lem}\label{wereactuallyworkingwithd} We have the below face-preserving homeomorphisms:
\begin{itemize}
\item $D_{+}(d)/D_{0}(d)\cong D(d)$ via the map $t\mapsto \log (t)$.
\item $\Sigma D_{0}(d)\cong D(d)$ via the map $(s,t_{0}=0,\ldots,t_{d-1})\mapsto (s+t_{0},s+t_{1},\ldots,s+t_{d-1})$.
\end{itemize}
\end{lem}

Hence the facial homotopy type of facial maps $D_{+}(d)/D_{0}(d)\to\Sigma D_{0}(d)$ is determined by degree. This works as follows:

\begin{rem}\label{whatdegreesmean} Via the identifications above, the $S^{1}$ in $D_{+}(d)/D_{0}(d)$ is given by the intersection of faces and thus is achieved when considering points in $D_{+}(d)/D_{0}(d)$ of the form $(t,\ldots,t)$
for $t$ in the strictly positive real numbers. The copy of $S^{1}$ in $\Sigma D_{0}(d)$ also arises from the intersection of faces - in this case the $S^{1}$ is from points $(s,0,\ldots,0)$ for $s$ the suspension coordinate taken over the real numbers. Thus observing the degree of the map $f:D_{+}(d)/D_{0}(d)\to \Sigma D_{0}(d)$ is given by considering the map $S^{1}\to S^{1}$ given by $f(t,\ldots,t)$, with the initial $S^{1}$ running over the positive real numbers and the final $S^{1}$ running over $\Real$.
\end{rem}
\chapter{Building the Tower}
\label{ch:ch4}

\section{The Tower}\label{TheTower}

We now state Theorem \ref{introresult} in more detail. Most of this chapter is then dedicated to defining the spectra and maps used in the construction of the tower; the topological claims in the result are proved in Chapter \ref{ch:ch5}. We recall most of the notation laid out throughout Chapter \ref{ch:ch1}.

\begin{thm}\label{themaintheorem}
There is a natural tower of finite $G$-$CW$-spectra:
$$
\xymatrix@C=1.5cm
{\Ell_{\infty}=X_{d_{0}}\ar[d]_{\pi_{d_{0}}}&G_{d_{0}}(V_{0})^{\Hom(T,V_{1} - V_{0})\oplus s(T)}\ar[l]_{\phi_{d_{0}}}\\
X_{d_{0}-1}\ar[ur]|\bigcirc^{\delta_{d_{0}}}\ar[d]_{\pi_{d_{0}-1}}&G_{d_{0}-1}(V_{0})^{\Hom(T,V_{1} - V_{0})\oplus s(T)}\ar[l]_{\phi_{d_{0}-1}}\\
X_{d_{0}-2}\ar[ur]|\bigcirc^{\delta_{d_{0}-1}}\ar@{.}[d]&G_{d_{0}-2}(V_{0})^{\Hom(T,V_{1} - V_{0})\oplus s(T)}\ar[l]_{\phi_{d_{0}-2}}\\
X_{1}\ar[d]_{\pi_{1}}&G_{1}(V_{0})^{\Hom(T,V_{1} - V_{0})\oplus s(T)}\ar[l]_{\phi_{1}}\\
S^{0}=X_{0}\ar[ur]|\bigcirc^{\delta_{1}}}
$$
This tower has the following properties:
\begin{enumerate}
\item The composition $\pi:\Ell_{\infty}\to S^{0}$ of the maps $\pi_{k}$ is the projection.
\item Each triangle is a cofibre triangle.
\end{enumerate}
\end{thm}

We have not specified in the above statement what any of the spectra or maps actually are. We spend this section defining sets that will become the $X_{k}$ after a desuspension before topologizing these sets in $\S$\ref{TopologizingTheTower}. We then spend the rest of this chapter showing that the Thom spaces in the above statement are well-defined virtual $G$-vector bundles that lie in $\mathcal{F}_{G}$ before concluding the section by defining the maps. The first step in proving this theorem is building a tower of $G$-spaces from $S^{\svo}\wedge\Ell_{\infty}$ to $S^{\svo}$.

Let $\alpha\in\svo$. Then as discussed in $\S$\ref{TheStandardTheory} $\alpha$ has real eigenvalues which can
thus be ordered. Hence this ordering passes to the eigenspaces of $\alpha$.

\begin{note}\label{whatareeigenspaces} Recall $e_{k}(\alpha)$ to be the $k^{\text{th}}$ eigenvalue of $\alpha$. Then the eigenspace associated to $e_{k}(\alpha)$ can be thought of as $\Ker(\alpha-e_{k}(\alpha))$; this is a vector subspace of $V_{0}$. Hence the $k^{\text{th}}$ eigenspace of $\alpha$ is the space $\Ker(\alpha-e_{k}(\alpha))$.
\end{note}

\begin{defn}\label{pkdefn} Define $P_{k}(\alpha)$ - a vector subspace of $V_{0}$ - to be the orthogonal complement in $V_{0}$ of the direct sum of the first $d_{0}-k$ eigenspaces of $\alpha$ under the eigenspace ordering:
$$P_{k}(\alpha):=\left[\bigoplus_{j+k<d_{0}}(\Ker(\alpha-e_{j}(\alpha)))\right]^{\bot}.$$
\end{defn}

\begin{rem}\label{whatpkcouldmean} Let $e_{d_{0}-k-1}(\alpha)\neq e_{d_{0}-k}(\alpha)$. Then in this case $P_{k}(\alpha)$ is the direct sum of the top $k$ eigenspaces of $\alpha$. In general, $P_{k}(\alpha)$ is the largest direct sum of top eigenspaces of $\alpha$ such that $\Dim(P_{k}(\alpha))\leqslant k$.
\end{rem}

It is easy to see from the definition that $P_{0}(\alpha)=0$, and $P_{d_{0}}(\alpha)=V_{0}$. We also note here that $P_{d_{0}-1}(\alpha)=(\Ker(\alpha-e_{0}(\alpha)))^{\bot}$ and that $P_{k}(\alpha)$ includes into $P_{k+1}(\alpha)$ in the obvious way. We now define sets $\tilde{X}_{k}'$ using this $P_{k}(\alpha)$ construction:

\begin{defn}\label{tildeXk'} The set $\tilde{X}_{k}'$ is defined as follows:
$$\tilde{X}_{k}':=\{(\alpha,\theta):\alpha\in\svo, \theta:P_{k}(\alpha)\to V_{1}\text{ is isometric}\}.$$
\end{defn}

\begin{rem}\label{itsagset} For any representation $V$ the space $s(V)$ is a $G$-space with conjugation action $g.\alpha (v):=g\alpha(g^{-1}v)$. Hence we can equip $\tilde{X}_{k}'$ with the action $g.(\alpha,\theta):=(g.\alpha,g.\theta)$. We need to check that this action is well-defined, i.e. that $g.\theta$ is an isometry from $P_{k}(g.\alpha)$ to $V_{1}$. This boils down to checking that $g^{-1}(P_{k}(g.\alpha))=P_{k}(\alpha)$ but this easily follows from the definition of the action on $\alpha$. Thus $\tilde{X}_{k}'$ is a $G$-set.
\end{rem}

It can be immediately noted from the remarks concerning $P_{0}(\alpha)$ and $P_{d_{0}}(\alpha)$ that $\tilde{X}_{d_{0}}'=\svo\times\Ell$ and that $\tilde{X}_{0}'=\svo$. We now define the following map of sets:

\begin{defn}\label{themapspi} We have a surjection of $G$-sets as follows:
$$\tilde{\pi}_{k}':\tilde{X}_{k}'\to\tilde{X}_{k-1}'$$
$$(\alpha,\theta)\mapsto (\alpha,\theta|_{P_{k-1}(\alpha)}).$$
This map is constructed from the inclusion $P_{k-1}(\alpha)\to P_{k}(\alpha)$.
\end{defn}

The $G$-set $\tilde{X}_{k}'$ will eventually be built into the $G$-spectrum $X_{k}$. There is one thing we have not mentioned, however, and that is the topology on $\tilde{X}_{k}'$. We now choose to discuss this in detail.

\section{Topologizing the Tower}\label{TopologizingTheTower}

We wish to suitably topologize the sets $\tilde{X}_{k}'$. The topology is clear for the cases $k=d_{0}$ and $k=0$ as they are the classical spaces $\svo\times\Ell$ and $\svo$ which have a standard topology; moreover equipped with this topology the two spaces become $G$-spaces. For the rest of the tower, however, we have to be more subtle. When topologized, we wish for the maps $\tilde{\pi}_{k}'$ to be continuous. This suggests that we should take the following topology:

\begin{rem}\label{quotientthingy} There is a surjection $\svo\times\Ell\to \tilde{X}'_{k}$ given by $(\alpha,\theta)\mapsto (\alpha,\theta|_{P_{k}(\alpha)})$. This allows us to
topologize $\tilde{X}'_{k}$ as a quotient of $\svo\times\Ell$.
\end{rem}

Under this topology the sets $\tilde{X}'_{k}$ become $G$-spaces and the maps $\tilde{\pi}_{k}'$ are continuous quotient $G$-maps. Further, we have the following standard result:

\begin{lem}\label{thetowerisGCW} The spaces $\svo$ and $\Ell$ have the homotopy type of finite $G$-$CW$-complexes. Hence each $G$-space $\tilde{X}_{k}'$ has the homotopy type of a finite $G$-$CW$-complex as the quotient of a finite complex.
\end{lem}

We now have a way of topologizing $\tilde{X}_{k}'$, however there is another way to do so. We choose to present another, equivalent, topology for $\tilde{X}_{k}'$ as this alternate topology eases many of the continuity arguments used later; as it is a subspace topology we use it for continuity proofs for functions mapping into the space.

Let $0\leqslant k<d_{0}$, then for any $\alpha\in\svo$ we have a defined $e_{d_{0}-k-1}(\alpha)$. Via \ref{continuouseigenvalues} we have the below continuous map:
$$\svo\to \svo$$
$$\alpha\mapsto \alpha-e_{d_{0}-k-1}(\alpha).$$

Further, the following is a continuous function:
$$f:\Real \to [0,\infty)$$
$$x\mapsto \text{max}(0,x).$$

Thus the below construction follows by the functional calculus \ref{fnalcalclemma}:

\begin{prop}\label{lambdak} We have a continuous map $\lambda_{k}:\svo\to s_{+}(V_{0})$ given by $\lambda_{k}(\alpha):=f(\alpha-e_{d_{0}-k-1}(\alpha))$ for the function $f$ above. The map $\lambda_{k}(\alpha):V_{0}\to V_{0}$ is zero on the bottom $d_{0}-k$ eigenspaces of $\alpha$ and $\alpha-e_{d_{0}-k-1}(\alpha)$ on the other eigenspaces.
\end{prop}

When restricted to $P_{k}(\alpha)$ the map $\lambda_{k}(\alpha)$ will provide an automorphism of $P_{k}(\alpha)$ as standard. Away from $P_{k}(\alpha)$, however, $\lambda_{k}(\alpha)$ is zero by definition. Hence the image of
$\lambda_{k}(\alpha)$ is $P_{k}(\alpha)$. We can use this construction to topologize $\tilde{X}_{k}'$ as follows:

\begin{prop}\label{topologisingthingy} Recalling $\rho$ from \ref{themaprho}, we have a bijection of sets:
$$\tilde{X_{k}}'\longrightarrow\{(\alpha,\beta):\alpha\in\svo, \beta:V_{0}\to V_{1}, \rho(\beta)=\lambda_{k}(\alpha)\}$$
$$(\alpha,\theta)\mapsto (\alpha,-\theta\circ \lambda_{k}(\alpha)).$$
\end{prop}
\begin{proof} First we check that the map is well-defined. The image of $\lambda_{k}(\alpha)$ is $P_{k}(\alpha)$ so
$-\theta\circ\lambda_{k}(\alpha)$ is reasonable; $\theta$ never ends up mapping out of too big a domain. Now we check that $\rho(\beta)$ is $\lambda_{k}(\alpha)$ as specified - noting that $\beta^{\dag}\beta=\lambda_{k}(\alpha)^{\dag}\circ\theta^{\dag}\circ\theta\circ\lambda_{k}(\alpha)=\lambda_{k}(\alpha)^{\dag}\circ\lambda_{k}(\alpha)$
and that $\lambda_{k}(\alpha)$ is as proved above selfadjoint and nonnegative is all that is required here. This checks that the map is well-defined.

Regarding injectivity, let $(\alpha,\theta)$ and $(\alpha',\theta')$ have the same image. Clearly $\alpha=\alpha'$ in this case and as $-\theta\circ\lambda_{k}(\alpha)=-\theta'\circ\lambda_{k}(\alpha)$ then $\theta$ and $\theta'$ must agree on the image of $\lambda_{k}(\alpha)$ which is $P_{k}(\alpha)$; thus $\theta$ and $\theta'$ agree and the map is injective.

To prove surjectivity we consider the pair $(\alpha,\beta)$ in the target set; we wish to build a suitable isometry $\theta$ from them. Our candidate isometry in this case is $-\sigma(\beta)$ as defined in \ref{themapsigma}. This has
already been shown to be a well-defined isometry, therefore we only need to demonstrate that its domain is $P_{k}(\alpha)$ and that $\sigma(\beta)\circ\lambda_{k}(\alpha)$ is $\beta$. For the first point, we wish to show that
$(\Ker(\beta))^{\bot}=P_{k}(\alpha)$. This, however, follows from the identity
$(\Ker(\beta))^{\bot}=(\Ker(\rho(\beta)))^{\bot}=\IM(\rho(\beta))=\IM(\lambda_{k}(\alpha))=P_{k}(\alpha)$; this result was discussed in the paragraph following \ref{themaprho}. Regarding the second point, recalling again that $\lambda_{k}(\alpha)=\rho(\beta)$ and that $\sigma(\beta)\circ\rho(\beta)=\beta$ is enough to prove the claim.
\end{proof}

\begin{rem}\label{wehaveasubspacetopology} The bijection of Proposition \ref{topologisingthingy} allows us to topologize $\tilde{X}_{k}'$ as a subspace of $\svo\times\aich$ by demanding that the bijection be a homeomorphism. Further, it is clear that the action on $\tilde{X}_{k}'$ defined earlier matches up with the standard diagonal-conjugation action on the subspace of $\svo\times\aich$.
\end{rem}

We now have two topologies on $\tilde{X}_{k}'$, we claim these are the same:

\begin{lem}\label{topologisinglemma} The following map is continuous and proper:
$$\svo\times \Ell\to \svo\times\aich$$
$$f:(\alpha,\theta)\mapsto (\alpha, -\theta\circ\lambda_{k}(\alpha)).$$
Hence $\tilde{X}_{k}'$ has the same topology as a quotient of $\svo\times\Ell$ or as a subspace of $\svo\times\aich$.
\end{lem}
\begin{proof} The map as stated is continuous because of the continuity of $\lambda_{k}(\alpha)$; thus we only need to show that it is proper. Let $U$ be a compact subset of $\svo\times\aich$. We claim that $f^{-1}(U)$ is a closed subset of a compact space and hence compact. Let $\pi:\svo\times\aich\to \svo$ be the projection to the first factor. $\pi(U)$ is the continuous image of a compact set and hence compact. It is easy to observe from here that $f^{-1}(U)\subseteq \pi(U)\times \Ell$ and that $\pi(U)\times \Ell$ is a compact space by Tychonoff's Theorem. Hence we only need to demonstrate that $f^{-1}(U)$ is closed. This is standard, however, as $U$ is a compact subspace of a Hausdorff space and thus closed. The second claim then follows immediately from the first claim and from Lemma \ref{whensubspaceisquotient} in an obvious way.
\end{proof}

We now have $G$-spaces $\tilde{X}_{k}'$. One thing to note here, however, is that the $X_{k}$ and $\pi_{k}$ we are required to define are based - thus we should check that what we've defined is compatible with adding a basepoint. Let $\tilde{X}_{k}$ be the one-point compactification of $\tilde{X}'_{k}$.

\begin{lem}\label{propermapslemma} The maps $\tilde{\pi}_{k}'$ are proper.
\end{lem}
\begin{proof} We claim that the projection $\tilde{X}_{k}'\to \svo$ is proper, then the result will follow from Lemma
\ref{propercompositionisproper}. Let $U$ be a compact subset of $\svo$, as $\svo$ is Hausdorff we note that this set is closed. The pre-image of $U$ in $\svo\times\Ell$ is $U\times\Ell$ which is compact. The continuous image of $U\times \Ell$ in $\tilde{X}_{k}'$ under the quotient map is thus also compact, and finally it is easy to observe that this
space is the same as the preimage of $U$ in $\tilde{X}_{k}'$. This proves the result.
\end{proof}

This lemma means we have extensions $\tilde{\pi}_{k}:\tilde{X}_{k}\to \tilde{X}_{k-1}$. This allows us to construct the following unstable tower of $G$-spaces with the homotopy type of finite based $G$-$CW$-complexes:
$$S^{\svo}\wedge\Ell_{\infty}\to\tilde{X}_{d_{0}-1}\to\tilde{X}_{d_{0}-2}\to\ldots\to S^{\svo}.$$

It is clear that the composition down this tower is just the projection $S^{\svo}\wedge\Ell_{\infty}\to S^{\svo}$. Recalling the categorical notes in $\S$\ref{onthecategories} we can apply $\Sigma^{\infty}$ and pass to the category $\mathcal{F}_{G}$. There is a finite $G$-$CW$-spectrum $S^{-\svo}:=\Sigma^{-\svo} S^{0}$. Thus we smash by $S^{-\svo}$ throughout and making the definition $X_{k}:=S^{-\svo}\wedge \tilde{X}_{k}$ we get the following tower in $\mathcal{F}_{G}$:
$$\Ell_{\infty}\to X_{d_{0}-1}\to X_{d_{0}-2}\to\ldots\to S^{0}.$$

We have already noted that $S^{\svo}\wedge\Ell_{\infty}\to S^{\svo}$ is just the projection, hence the map $\Ell_{\infty}\to S^{0}$ is the projection. Part $1$ of Theorem \ref{themaintheorem} then follows; we have constructed a suitable stable tower such that the map running down the whole tower is the projection.

\section{The Topology of Bundles Over Grassmannians}\label{TheTopologyOfBundlesOverGrassmannians}

We now turn our attention to defining the $G$-spectra that will become the cofibres of the maps $\pi_{k}$. To do this, we need to take a detour into the topology of Grassmannians and of bundles over Grassmannians. The
general theory is standard, for example the non-equivariant theory is covered in \cite{CharacteristicClassesMilnorStasheff} in the setting of real Grassmannians. We include the detail, however, to cover results specific to our purpose - we wish to be clear on the subjects of equivariance, topology and continuity when it comes to the bundles we use. 

\begin{defn}\label{whatisanequigrassmannian} We take $G_{k}(V_{0})$ to be the Grassmannian of all $k$-dimensional vector subspaces of $V_{0}$. We then equip $G_{k}(V_{0})$ with the standard $G$-action inherited from the action on $V_{0}$. 
\end{defn}

In order to topologize $G_{k}(V_{0})$ effectively, we first make the following definition:

\begin{defn}\label{GVandGPrime} Set $G(V_{0}):=\coprod_{k=0}^{d_{0}} G_{k}(V_{0})$. We also define the space $G'(V_{0})$ as $\{\pi\in\End(V_{0}):\pi^{2}=\pi, \pi^{\dag}=\pi\}$. Finally set $G'_{k}(V_{0})$ to be the subspace of $\pi\in G'(V_{0})$ with trace equal to $k$.
\end{defn}

$G'(V_{0})$ and its associated subspaces can then be topologized as closed subspaces of $s(V_{0})$. We now note that there are bijections between $G_{k}(V_{0})$ and $G_{k}'(V_{0})$ given by $W\mapsto 1_{W}\oplus 0_{W^{\bot}}$ and $\IM(\pi)\mapsfrom\pi$; these are mutually inverse. This allows us to topologize $G_{k}(V_{0})$ by making the bijections into homeomorphisms and setting that $U$ is open in $G_{k}(V_{0})$ if and only if its image is open in $G_{k}'(V_{0})$. Under this topology $G_{k}(V_{0})$ is a $G$-space.

\begin{lem}\label{compactgrassmannian} $G_{k}(V_{0})$ is compact.
\end{lem}
\begin{proof} If $\pi\in G_{k}'(V_{0})$ then $\pi$ is of the form $1_{W}\oplus 0_{W^{\bot}}$ for some $W$. If $k\neq 0$ then it is simple to see that $\|\pi\|=1$ and thus that $G_{k}'(V_{0})$ is closed and bounded in $s(V_{0})$, hence compact by Lemma \ref{extendedheineborel}. For $k=0$ we have $\|\pi\|=0$, the proof then follows.
\end{proof}

Similar to results in the previous section we also note that there is another way of topologizing $G_{k}(V_{0})$, this time via quotients. Let $i:\mathcal{L}(\Complex^{k},V_{0})\to G_{k}(V_{0})$ be the map given by $i(\alpha):=\IM(\alpha)$. This is continuous as it corresponds to the continuous $i':\mathcal{L}(\Complex^{k},V_{0})\to G_{k}'(V_{0})$ given by $i'(\alpha):=\alpha\alpha^{\dag}$. The map is also surjective via a standard linear algebra argument. Thus as $i$ is a continuous surjection from a compact space to a Hausdorff space it is a quotient map and we have a different way of equipping $G_{k}(V_{0})$ with the above topology.

We now show that the map $P_{k}:\alpha\mapsto P_{k}(\alpha)$ mapping into $G_{k}(V_{0})$ is continuous under this
topology. This will be one of the cornerstones of the maps we will build in the next chapter. To do this we need the following lemma:

\begin{lem}\label{poseigenvaluescontinuouslybuildplanes} Let $\svo^{\times}$ be the space of injective selfadjoint
endomorphisms, or equivalently $\{\alpha\in\svo:e_{i}(\alpha)\in \Real^{\times}\forall i\}$. Then the map $f:\svo^{\times}\to G(V_{0})$ sending $\alpha$ to the sum of the positive eigenspaces of $\alpha$ is continuous.
\end{lem}
\begin{proof} Define $f':\Real^{\times}\to \{0,1\}$ sending all negative numbers to $0$ and all positive numbers to $1$. This is clearly continuous under any topology one can give $\{0,1\}$, self-conjugate and such that $f'(t)^{2}=f'(t)$ for all $t$. Thus by Lemma \ref{fnalcalclemma} we have a map $f:\svo^{\times}\to G'(V)$ and it is simple to check that this map behaves as stated in the claim.
\end{proof}

\begin{cor}\label{Pkiscont} Let $s_{k}(V_{0})$ be the subspace of $\svo$ given as follows:
$$s_{k}(V_{0}):=\{\alpha\in\svo:\Dim(P_{k}(\alpha))=k\}.$$
Then $P_{k}:s_{k}(V_{0})\to G_{k}(V_{0})$ is continuous.
\end{cor}
\begin{proof} Firstly if $k=0$ then $P_{0}(\alpha)=0$ and the claim follows trivially. Now if $k=d_{0}$ then $P_{d_{0}}:s_{d_{0}}(V_{0})\to G_{d_{0}}(V_{0})$ is the continuous map $s_{d_{0}}(V_{0})\to\text{pt.}$. Let $0<k<d_{0}$, then $e_{d_{0}-k}$ and $e_{d_{0}-k-1}$ are well-defined continuous maps and we have $s_{k}(V_{0})=\{\alpha\in\svo:e_{d_{0}-k-1}(\alpha)<e_{d_{0}-k}(\alpha)\}$. Define a map $s_{k}(V_{0})\to \Real$ given as follows:
$$\alpha\mapsto 1/2(e_{d_{0}-k-1}(\alpha)+e_{d_{0}-k}(\alpha)).$$
This is clearly a continuous map and $1/2(e_{d_{0}-k-1}(\alpha)+e_{d_{0}-k}(\alpha))$ is not an eigenvalue of $\alpha$, thus $\alpha-1/2(e_{d_{0}-k-1}(\alpha)+e_{d_{0}-k}(\alpha))$ is in $s(V_{0})^{\times}$. The claim then follows by observing that $P_{k}(\alpha)$ is $f(\alpha-1/2(e_{d_{0}-k-1}(\alpha)+e_{d_{0}-k}(\alpha)))$ for the function $f$ defined in \ref{poseigenvaluescontinuouslybuildplanes}.
\end{proof}

This map will prove to be useful in constructing some of the later continuity arguments. Another useful ingredient is the
following observation:

\begin{note}\label{perpishomeo} $G_{k}(V_{0})$ is homeomorphic to $G_{d_{0}-k}(V_{0})$ via $W\mapsto W^{\bot}$. This follows from the homeomorphism $G_{k}'(V_{0})\cong G_{d_{0}-k}'(V_{0})$ sending $\pi$ to $1_{V_{0}}-\pi$.
\end{note}

We define the following sets:
$$Z_{k}:=\{(W,\gamma):W\in G_{k}(V_{0}),\gamma\in\Hom(W,V_{1})\},$$
$$\tilde{Z}_{k}:=\{(W,\gamma,\psi):W\in G_{k}(V_{0}),\gamma\in\Hom(W,V_{1}),\psi\in s(W^{\bot})\}.$$

We wish to equip the above sets with a suitable topology. This is done as follows:

\begin{lem}\label{topologizingthefibrebundles} We have a bijection between $Z_{k}$ and the space
$Z_{k}':=\{(\pi,\beta):\pi\in G_{k}'(V_{0}),\beta\in\aich,
\beta\circ (1-\pi)=0\}$ given by the below maps:
$$(W,\gamma)\mapsto (1_{W}\oplus 0_{W^{\bot}},\gamma\circ(1_{W}\oplus 0_{W^{\bot}}))$$
$$(\IM(\pi),\beta|_{\IM(\pi)})\mapsfrom (\pi,\beta).$$
This lets us topologize $Z_{k}$ as a subspace of $G_{k}'(V_{0})\times \aich$. Moreover we have the following surjection onto $Z_{k}$:
$$\mathcal{L}(\Complex^{k},V_{0})\times\Hom(\Complex^{k},V_{1})\to Z_{k}$$
$$(\zeta,\gamma_{0})\mapsto (\IM(\zeta),\gamma_{0}\circ \zeta^{\dag}).$$
Thus $Z_{k}$ can also be topologized as a quotient of $\mathcal{L}(\Complex^{k},V_{0})\times\Hom(\Complex^{k},V_{1})$. The composition $\mathcal{L}(\Complex^{k},V_{0})\times\Hom(\Complex^{k},V_{1})\to G_{k}'(V_{0})\times \aich$ is continuous and proper and hence by Lemma \ref{whensubspaceisquotient} the two topologies are the same.
\end{lem}
\begin{proof} The proof is simple barring demonstrating that the below composition is proper:
$$f:\mathcal{L}(\Complex^{k},V_{0})\times\Hom(\Complex^{k},V_{1})\to G_{k}'(V_{0})\times \aich$$
$$(\zeta,\gamma_{0})\mapsto (\zeta\zeta^{\dag},\gamma_{0}\circ\zeta^{\dag}\circ(\zeta\zeta^{\dag})).$$
First note that $\zeta^{\dag}\circ(\zeta\zeta^{\dag})=\zeta^{\dag}$. Let $C$ be a compact subset of $G_{k}'(V_{0})\times \aich$. We need to check that the inverse image of $C$ is closed in each factor and bounded in the $\Hom$-factor. By Tychonoff's Theorem and \ref{extendedheineborel} this will be enough to show that $f^{-1}\{C\}$ is compact. Firstly consider the projection from $\mathcal{L}(\Complex^{k},V_{0})\times\Hom(\Complex^{k},V_{1})$ down to $\Hom(\Complex^{k},V_{1})$, this is a closed map via \ref{closedprojifcompact} as $\mathcal{L}(\Complex^{k},V_{0})$ is compact. Now let $(\zeta,\gamma_{0})\in f^{-1}\{C\}$. We claim that there is a bound on the norm of $\gamma_{0}$. As $C$ is compact we are given that $\|\gamma_{0}\circ\zeta^{\dag}\|\leqslant R$ for some real number $R$, the bound on $\gamma_{0}$ then follows by a standard observation that $\|\gamma_{0}\|=\|\gamma_{0}\circ\zeta^{\dag}\|$. Thus $f^{-1}\{C\}$ is compact in the second factor. Another application of \ref{closedprojifcompact} gives us that the projection of $f^{-1}\{C\}$ down to $\mathcal{L}(\Complex^{k},V_{0})$ is closed and hence compact as $\mathcal{L}(\Complex^{k},V_{0})$ is compact. This is enough to show that $f^{-1}\{C\}$ is compact and hence $f$ is proper. The rest of the claim follows.
\end{proof}

\begin{cor}\label{factoringinsTperp} There is a bijection between $\tilde{Z}_{k}$ and the following space, $\pi$ and $\beta$ satisfying the same conditions as in Lemma \ref{topologizingthefibrebundles} above:
$$\{(\pi,\beta,\xi):\xi\in s(V_{0}), \xi\circ\pi=0, \IM(\xi)\subseteq\IM(1_{V_{0}}-\pi)\}.$$
This is given by:
$$(W,\gamma,\psi)\mapsto (1_{W}\oplus 0_{W^{\bot}},\gamma\circ(1_{W}\oplus 0_{W^{\bot}}),\psi\circ(1_{W^{\bot}}\oplus 0_{W}))$$
$$(\IM(\pi),\beta|_{\IM(\pi)},\xi|_{\IM(1_{V_{0}}-\pi)})\mapsfrom (\pi,\beta,\xi).$$
Hence we can topologize $\tilde{Z}_{k}$ as a subspace of $G_{k}'(V_{0})\times \aich\times \svo$. Moreover we have the below surjection onto $\tilde{Z}_{k}$:
$$\mathcal{L}(\Complex^{k}\oplus \Complex^{d_{0}-k},V_{0})\times\Hom(\Complex^{k},V_{1})\times s(\Complex^{d_{0}-k})\to \tilde{Z}_{k}$$
$$((\zeta,\eta),\gamma_{0},\psi_{0})\mapsto (\IM(\zeta),\gamma_{0}\circ \zeta^{\dag},\eta\circ\psi_{0}\circ\eta^{\dag}).$$
Thus $\tilde{Z}_{k}$ can also be topologized as a quotient. The composition of the above maps $\mathcal{L}(\Complex^{k}\oplus \Complex^{d_{0}-k},V_{0})\times\Hom(\Complex^{k},V_{1})\times s(\Complex^{d_{0}-k})\to G_{k}'(V_{0})\times \aich\times\svo$ is continuous and proper and hence by Lemma \ref{whensubspaceisquotient} the two topologies are the same.
\end{cor}
\begin{proof} Note that $\|\psi_{0}\|=\|\eta\circ\psi_{0}\circ\eta^{\dag}\|$. The result can then easily be proved by following the method and style of the previous proof and recalling that compact subsets of spaces of selfadjoints are closed and bounded subsets as noted in \ref{extendedheineborel}.
\end{proof}

We can also topologize subspaces of $\tilde{Z}_{k}$ in a similar fashion. We now prove that the space $Z_{k}$ is a vector bundle over $G_{k}(V_{0})$ by demonstrating how $Z_{k}$ under this topology satisfies local triviality. We note here that this result easily generalizes using the same techniques as below to give similar results for other cases; for example $\tilde{Z}_{k}$ is a vector bundle. 

First fix $W\in G_{k}(V_{0})$ and set $\pi_{W}$ as its image in $G_{k}'(V_{0})$. We now define the following set:
$$U':=\{\pi\in G_{k}'(V_{0}):\pi_{W}\circ \pi|_{W}:W\to W\text{ is an isomorphism}\}.$$

We set $U$ to be the image of $U'$ in $G_{k}(V_{0})$. We first claim that $U'$ is open in $G_{k}'(V_{0})$. Choose $1>\epsilon>0$ and let $\pi$ be such that $\|\pi-\pi_{W}\|<\epsilon$. The map $\pi-\pi_{W}$ is $1$ on some subset of $V_{0}$ and $0$ everywhere else, thus the value of the norm $\|\pi-\pi_{W}\|$ is either $0$ or $1$. As $\epsilon<1$ then
$\|\pi-\pi_{W}\|=0$, hence $\pi_{W}\circ \pi|_{W}$ is an isomorphism and thus $\pi\in U'$; this demonstrates that $U'$ and $U$ are open. The set $U$ will be our neighbourhood around $W$ satisfying the properties of the local triviality condition.

Now for $\alpha\in \Hom(W,W^{\bot})$ set $\hat{\alpha}:W\to W\oplus W^{\bot}=V_{0}$ to be given by $\hat{\alpha}(w)=w+ \alpha(w)$. We also give this map an adjoint by $\hat{\alpha}^{\dag}:=1+\alpha^{\dag}$. This allows us to define a
map $g'$ out of $\Hom(W,W^{\bot})$:
$$g'(\alpha):=\left(V_{0}\overset{\hat{\alpha}^{\dag}}{\to}W\overset{(1+\alpha^{\dag}\alpha)^{-1}}{\to}W\overset{\hat{\alpha}}{\to} V_{0}\right).$$

This is well-defined as $(1+\alpha^{\dag}\alpha)=\hat{\alpha}^{\dag}\hat{\alpha}$ is strictly positive and thus has an inverse. Its easy to see that $g'$ actually maps into $G'(V_{0})$; the map $g'(\alpha)$ is patently
selfadjoint and idempotent. Moreover, as $W$ is a dimension $k$ subspace and $\hat{\alpha}^{\dag}$ is surjective it is easy to see that $g'(\alpha)$ has trace $k$ and thus lies in $G_{k}'(V_{0})$.

There is hence also a map $g:\Hom(W,W^{\bot})\to G_{k}(V_{0})$, from the surjectivity of $\hat{\alpha}^{\dag}$ it follows that $\IM(g'(\alpha))=\IM(\hat{\alpha})$. Thus we have that $g(\alpha)=\IM(\hat{\alpha})$. We now claim that $g'$ actually lands in $U'$ and moreover gives a homeomorphism between $U'$ and $\Hom(W,W^{\bot})$. That $g'$ lands in $U'$ follows from observing the behaviour of $g'(\alpha)|_{W}$, it can then further be observed that $\pi_{W}\circ g'(\alpha)|_{W}$ is an isomorphism as required. Continuity of $g'$ is also easy to observe, thus we only
need to demonstrate the existence of a continuous inverse to $g'$.

To do this, we first note that $V_{0}\cong W\oplus W^{\bot}$ and hence $g'(\alpha):W\oplus W^{\bot}\to W\oplus W^{\bot}$ can be decomposed into a $2\times 2$ matrix. This is a standard technique used in functional analysis, for example it is prevalent throughout \cite{Wegge-Olsen}. In our case the matrix is below, it takes values along the top row in $\Hom(W,W)$ and $\Hom(W,W^{\bot})$ and along the bottom row in $\Hom(W^{\bot},W)$ and $\Hom(W^{\bot},W^{\bot})$:
$$g'(\alpha)=\left(\begin{array}{cc}\alpha(1+\alpha^{\dag}\alpha)^{-1}\alpha^{\dag}&\alpha(1+\alpha^{\dag}\alpha)^{-1}\\
(1+\alpha^{\dag}\alpha)^{-1}\alpha^{\dag}&(1+\alpha^{\dag}\alpha)^{-1}\end{array}\right).$$

Using the same decomposition as above let $g'':U'\to \Hom(W,W^{\bot})$ be built as follows:
$$g''\left[\left(\begin{array}{cc}\pi_{11}&\pi_{12}\\\pi_{21}&\pi_{22}\end{array}\right)\right]:=\pi_{12}\circ\pi_{22}^{-1}.$$

This is patently continuous and moreover one can easily observe that it provides the required inverse to $g'$. Thus $g'$ is a homeomorphism, a fact which is key towards proving the below result:

\begin{prop}\label{xisavectorbundle} $Z_{k}$ is a vector bundle over $G_{k}(V_{0})$.
\end{prop}
\begin{proof} Let $W$ be a point in $G_{k}(V_{0})$, take $U$ as defined above and let $q:Z_{k}\to G_{k}(V_{0})$ be the bundle projection. We wish to show the following triangle commutes for $\Phi$ a homeomorphism and $p_{1}$ the projection to the first factor:
$$\xymatrix{U\times \Hom(W,V_{1})\ar[dr]_{p_{1}}\ar[r]^{\phantom{xxxx}\Phi}&q^{-1}(U)\ar[d]\\
& U}$$

This will prove that $Z_{k}$ satisfies local triviality and the claim will follow. However, commutativity of the above diagram follows immediately if the map $\tilde{g}$ fitting into the below commutative diagram is a homeomorphism:
$$\xymatrix{\Hom(W,W^{\bot})\times \Hom(W,V_{1})\ar[d]_{p_{1}}\ar[r]^{\phantom{xxxxxxxxx}\tilde{g}}&q^{-1}(U)\ar[d]\\
\Hom(W,W^{\bot})\ar[r]_{g}^{\cong}& U}$$

Here $\tilde{g}(\alpha,\beta):=(g(\alpha),g(\alpha)\overset{\sigma(\hat{\alpha})^{-1}}{\to} W\overset{\beta}{\to} V_{1})$ recalling $\sigma$ from \ref{themapsigma} - we introduce $\sigma$ here to produce an isometric trivialization. This is clearly continuous, with continuous inverse given by $(V,\theta)\mapsto (g^{-1}(V),\theta\circ\sigma(\widehat{g^{-1}(V)}))$ - this is easily checked to be well-defined and moreover it can also be checked that it gives the required two-sided inverse. Thus $Z_{k}$ is a fibre bundle of vector spaces and hence a vector bundle over $G_{k}(V_{0})$.
\end{proof}

There is also a similar result proving that $\tilde{Z}_{k}$ is a vector bundle over $G_{k}(V_{0})$. Moreover, reintroducing the tautological bundle notation from Chapter \ref{ch:ch1} and using $E$ as notation for the total space of a vector bundle it is clear that we can use the following notation for $Z_{k}$ and $\tilde{Z}_{k}$:
$$Z_{k}=E(\Hom(T,V_{1}))$$
$$\tilde{Z}_{k}=E(\Hom(T,V_{1})\oplus s(T)).$$

As we are working with vector bundles there is a well-defined concept of a Thom space. Moreover since $G_{k}(V_{0})$ is compact the Thom space is just the one-point compactification of the base space. Hence we have the following Thom spaces:
$$G_{k}(V_{0})^{\Hom(T,V_{1})}=(Z_{k})_{\infty}$$
$$G_{k}(V_{0})^{\Hom(T,V_{1})\oplus s(T^{\bot})}=(\tilde{Z}_{k})_{\infty}.$$

We conclude this section by reintroducing equivariance. A similar result also holds for $Z_{k}$:

\begin{lem} The following is a well-defined $G$-action on $\tilde{Z}_{k}$, recalling the action on $G_{k}(V_{0})$ and the conjugation action on spaces of maps:
$$g.(W,\gamma,\psi):=(g.W,g.\gamma,g.\psi).$$
This action makes $\tilde{Z}_{k}$ into both a $G$-space and a $G$-vector bundle over $G_{k}(V_{0})$.
\end{lem}
\begin{proof} The only issue is checking that the action is well-defined, the rest of the claims are standard. To check that the action is well-defined we first need to check that $g.\gamma$ and $g.\psi$ have domain $g.W$. This involves checking that $g^{-1}(g.W)=W$ but this is clear. We also note that $g.\psi$ clearly has codomain $g.W$ by the nature of the action. This is enough to prove the result.
\end{proof}

Hence we have $G_{k}(V_{0})^{\Hom(T,V_{1})\oplus s(T^{\bot})}=(\tilde{Z}_{k})_{\infty}$, a well-defined based $G$-space. Moreover, the topological analysis above will allow us to construct various technical continuity arguments in Chapter \ref{ch:ch5}. We conclude this section by stating a lemma that will allow us to stabilize this Thom space and in the next section build the spectra that will become our cofibres. The proof is fairly simple and follows from the finite $CW$-structure possessed by $G_{k}(V_{0})$. 

\begin{lem}\label{thomspaceisfiniteCW} $G_{k}(V_{0})^{\Hom(T,V_{1})\oplus s(T^{\bot})}$ has the homotopy type of a finite based $G$-$CW$-complex.
\end{lem}

\section{Stabilizing the Cofibres}\label{DestabilizingtheTower}

When we turn to calculating the cofibre of the map $\pi_{k}$ we instead will wish to calculate
the homotopy cofibre of the map $\tilde{\pi}_{k}$. This cofibre will turn out to be $G_{k}(V_{0})^{\Real\oplus\Hom(T,V_{1})\oplus s(T^{\bot})}$. Thus the stable cofibre of $\pi_{k}$ will be $S^{-\svo}\wedge G_{k}(V_{0})^{\Real\oplus\Hom(T,V_{1})\oplus s(T^{\bot})}$. However, this spectrum can be re-written in a nicer form.

All the standard categories of general spectra contain virtual vector bundles, formal differences $V-W$ for vector bundles $V$ and $W$. This is not immediately noticeable in $\mathcal{F}_{G}$ as virtual bundles don't explicitly take the form of a desuspension of a space but it can easily shown to be true. We use the shorthand $\Hom(T,V_{1}-V_{0})$ for the virtual bundle $\Hom(T,V_{1})-\Hom(T,V_{0})$; further we note here that this bundle is the honest bundle $\Hom(T,V_{1}\ominus V_{0})$ in the case where $V_{0}\leqslant V_{1}$. This virtual bundle has a $G$-action given in a similar fashion to the action on $\tilde{Z}_{k}$ in the previous section. Then via the below lemma we can rewrite our stable cofibre in the following form:
$$S^{-\svo}\wedge G_{k}(V_{0})^{\Real\oplus\Hom(T,V_{1})\oplus s(T^{\bot})}\cong G_{k}(V_{0})^{\Real\oplus\Hom(T,V_{1}-V_{0})\oplus s(T)}.$$

\begin{lem}\label{destabilizinglemma} As $G$-bundles over $G_{k}(V_{0})$ we have:
$$\Hom(T,V_{1} - V_{0})\oplus
s(T)\oplus \svo\cong \Hom(T,V_{1})\oplus s(T^{\bot}).$$
\end{lem}
\begin{proof} Firstly we have that if $W$ is a subspace of $V_{0}$ then $\svo$ can be decomposed into parts in $s(W)$, $s(W^{\bot})$ and $\Hom(W,W^{\bot})$; this gives us the bundle identification $\svo\cong s(T)\oplus s(T^{\bot})\oplus\Hom(T,T^{\bot})$. This is easiest to see when $\alpha\in \svo$ is written down as a $2\times
2$ matrix using the techniques used in the last section or in \cite{Wegge-Olsen}, for example. Via this technique the
selfadjoint endomorphism $\alpha$ takes the form below for $\beta\in s(W)$, $\gamma\in s(W^{\bot})$ and $\delta\in\Hom(W,W^{\bot})$; it is clear that this implies the above bundle isomorphism and further that this bundle isomorphism is equivariant:
$$\alpha=\left(\begin{array}{cc}\beta& \delta^{\dag}\\
\delta & \gamma\end{array}\right).$$
Next we note that if $A\in \End(V)$ then there is a classical decomposition of $A$ into the sum of two selfadjoint endomorphisms. This is given by $A\mapsto (1/2(A+A^{\bot}),i/2 (A-A^{\bot}))$ and has inverse $(\alpha,\beta)\mapsto \alpha\oplus i\beta$. Thus we get the bundle identification $\Hom(T,T)\cong 2s(T)$, further this identity is again  clearly equivariant. This allows us to make the following bundle identifications:
\begin{align}\Hom(T,V_{0})&\cong\Hom(T,T)\oplus\Hom(T,T^{\bot})\nonumber\\
&\cong 2s(T)\oplus\Hom(T,T^{\bot}).\nonumber
\end{align}
The result follows from this and the earlier statement that $\svo\cong s(T)\oplus s(T^{\bot})\oplus\Hom(T,T^{\bot})$.
\end{proof}

Thus, we wish to construct cofibre sequences that take the below form:
$$\xymatrix{\tilde{X}_{k}\ar[d]_{\tilde{\pi}_{d_{0}}}&G_{k}(V_{0})^{\Hom(T,V_{1})\oplus s(T^{\bot})}\ar[l]\ar[l]_{\tilde{\phi}_{k}\phantom{xxxxx}}\\
\tilde{X}_{k-1}\ar[ur]|\bigcirc_{\phantom{x}\tilde{\delta}_{k}}&}
$$

In particular, the top of the unstable tower should be:
$$\xymatrix{\tilde{X}_{d_{0}}\ar[d]_{\tilde{\pi}_{d_{0}}}&S^{\aich}\ar[l]_{\tilde{\phi}_{d_{0}}\phantom{xxx}}\\
\tilde{X}_{d_{0}-1}\ar[ur]|\bigcirc_{\phantom{x}\tilde{\delta}_{d_{0}}}&}
$$

We now explicitly state candidates for $\tilde{\phi}_{k}$ and $\tilde{\delta}_{k}$.

\section{Candidate Maps}\label{CandidateMaps}

We now state what will eventually be shown to be $\tilde{\phi}_{k}$ and $\tilde{\delta}_{k}$. To do this we will have to first deal with the top triangle, before moving on to consider the rest of the tower. While we remark that this seems somewhat unsatisfactory at first it will become clear in the chapters following that the cofibre sequences will fall out from this in a somewhat natural way.

Throughout this section we recall $\rho$, $\sigma$ and $\lambda_{k}$ from \ref{themaprho}, \ref{themapsigma} and \ref{lambdak}. We make the following definition:

\begin{defn}\label{topphi} Set the map $\tilde{\phi}_{d_{0}}:S^{\Hom(V_{0},V_{1})}\to S^{\svo}\wedge\Ell_{\infty}$ to be the collapse $\kappa^{!}$ from Proposition \ref{thetaEalpha}.
\end{defn}

This is patently a $G$-map. To build $\tilde{\delta}_{d_{0}}$ we need a similar result to Proposition \ref{thetaEalpha}. This turns out to be the below proposition:

\begin{prop}\label{sigmainjchomeo} There is a homeomorphism:
$$\tau:\tilde{X}'_{d_{0}-1}=\{(\alpha,\theta):\alpha\in\svo, \theta\in\mathcal{L}(P_{d_{0}-1}(\alpha),V_{1})\}\overset{\cong}{\to}\Real\times \inj(V_{0},V_{1})^{c}$$
$$(\alpha,\theta)\mapsto(e_{0}(\alpha),-\theta\circ(\alpha-e_{0}(\alpha))).$$
Thus we have $\tilde{X}_{d_{0}-1}\cong \Sigma \inj(V_{0},V_{1})^{c}_{\infty}$.
\end{prop}

\begin{proof} We build this map in two stages. Firstly we note that in this case the map $\lambda_{d_{0}-1}(\alpha)$ is given by $\alpha\mapsto\alpha-e_{0}(\alpha)$. From Proposition \ref{topologisingthingy} we have a homeomorphism between $\tilde{X}_{d_{0}-1}'$ and $\{(\alpha,\beta):\alpha\in\svo, \beta:V_{0}\to V_{1}, \rho(\beta)=\alpha-e_{0}(\alpha)\}$. We then define a map out of this homeomorphic space and into $\Real\times \inj(V_{0},V_{1})^{c}$ given by $(\alpha,\beta)\mapsto (e_{0}(\alpha),\beta)$. It is clear that $\tau$ is this composition. Thus all that is required is proof that the second map in the factorization is a homeomorphism.

Firstly we check that this map is well-defined, i.e. that $\beta$ in this case is non-injective. This is clear, however, as if $v$ is an eigenvector of $\alpha$ with eigenvalue $e_{0}(\alpha)$ then $\beta$ will send $v$ to $0$. We also note that this map is patently continuous via previous remarks. We next build a continuous inverse of the map, this will
therefore be enough to show that the map is a homeomorphism and hence that $\tau$ is a homeomorphism.

Define the map out of $\Real\times \inj(V_{0},V_{1})^{c}$ given by $(t,\delta)\mapsto (\rho(\delta)+t,\delta)$. We claim that this is a well-defined continuous inverse. On the first point, note that $\rho(\delta)+t$ is patently selfadjoint. We now just need to check that if $\alpha=\rho(\delta)+t$ then $\rho(\delta)$ is $\alpha-e_{0}(\alpha)$. Firstly note that $e_{0}(\alpha)$ is going to be $t$ - the lowest eigenvalue of $\rho(\delta)$ will be zero as $\delta$ is non-injective. Thus $\alpha-e_{0}(\alpha)$ will be $\rho(\delta)$ as required and the map is well-defined, landing in
the right domain. We also note here that the map is continuous via Lemmas \ref{fnalcalclemma} and \ref{continuouseigenvalues}.

Finally, we check that the two given maps are self-inverses. This, however, follows immediately from the lowest eigenvalue of $\rho(\delta)+t$ being $t$. Thus the map is a homeomorphism and thus $\tau$ is, proving the claim.
\end{proof}

\begin{defn}\label{topdelta} The map $\tilde{\delta}_{d_{0}}:\tilde{X}_{d_{0}-1}\to S^{\Real\oplus\aich}$ is given by composition of the homeomorphism $\tilde{X}_{d_{0}-1}\cong \Sigma\inj(V_{0},V_{1})^{c}_{\infty}$ with the twisted inclusion $-\Sigma i_{d_{0}}:\Sigma\inj(V_{0},V_{1})^{c}_{\infty}\to S^{\Real\oplus\aich}$, recalling $-\Sigma$ from Definition \ref{thecofibrestwist}.
\end{defn}

This map is again patently a $G$-map; both parts of the composition are easily observed to be $G$-maps. We now state what we will later demonstrate to be $\tilde{\phi}_{k}$ and $\tilde{\delta}_{k}$ for $k< d_{0}$. 

\begin{defn}\label{tildephik} $\tilde{\phi}_{k}$ is the extension of the below proper map:
$$\tilde{Z}_{k}\to \tilde{X}_{k}'$$
$$(W,\gamma,\psi)\mapsto (\psi|_{W^{\bot}}\oplus (\rho(\gamma)+e_{top}(\psi))|_{W},-\sigma(\gamma)).$$
\end{defn}

\begin{defn}\label{makingthecofibseqeasier} Set $Y_{k}'$ to be the following subspace of $\tilde{X}_{k}'$:
$$Y_{k}':=\{(\alpha,\theta)\in\tilde{X}_{k}':\Dim(P_{k}(\alpha))<k\}.$$

Abusing notation somewhat, this can be considered as both a subspace of $\tilde{X}_{k}'$ and of $\tilde{X}_{k-1}'$. We put $Y_{k}:=(Y_{k}')_{\infty}$.
\end{defn}

\begin{defn}\label{tildedeltak} $\tilde{\delta}_{k}$ is the composition of the collapse map $\tilde{X}_{k-1}\to \tilde{X}_{k-1}/Y_{k}$ with the map $\tilde{X}_{k-1}/Y_{k}\to G_{k}(V_{0})^{\Real\oplus \Hom(T,V_{1})\oplus s(T^{\bot})}$ built out of the below proper map:
$$\tilde{X}_{k-1}'\backslash Y'_{k}\to \Real\times \tilde{Z}_{k}$$
$$(\alpha,\theta)\mapsto \left(e_{d_{0}-k}(\alpha),P_{k}(\alpha),-\theta\circ \lambda_{k-1}(\alpha)|_{P_{k}(\alpha)},-\log((e_{d_{0}-k}(\alpha)-\alpha)|_{P_{k}(\alpha)^{\bot}})\right).$$
Finally, we insert a suspension twist $-\Sigma$ as detailed in Definition \ref{thecofibrestwist}.
\end{defn}

This gives us definitions for all of the maps in the tower. We note, however, that again one important issue has yet to be dealt with - barring the top triangle we have yet to demonstrate whether these maps are continuous, proper, equivariant or even well-defined. We will do this for $\tilde{\delta}_{k}$ in the next chapter, while we hold off the work on $\tilde{\phi}_{k}$ until Chapter \ref{ch:ch6} wherein we equate the cofibre sequences back to the map above.
\chapter{The Cofibres}
\label{ch:ch5}

\section{The Top Triangle}\label{TheTopTriangle}

In $\S$\ref{CandidateMaps} we defined candidate maps for the top of the tower. Recalling the definitions from the previous chapter, we now wish to show that the following triangle is a cofibre sequence:
$$\xymatrix{\tilde{X}_{d_{0}}\ar[d]_{\tilde{\pi}_{d_{0}}}&S^{\aich}\ar[l]_{\tilde{\phi}_{d_{0}}\phantom{xxx}}\\
\tilde{X}_{d_{0}-1}\ar[ur]|\bigcirc_{\phantom{x}\tilde{\delta}_{d_{0}}}&}
$$

In defining $\tilde{\phi}_{d_{0}}$ and $\tilde{\delta}_{d_{0}}$ we first demonstrated homeomorphisms $\tilde{X}_{d_{0}}\cong\inj(V_{0},V_{1})_{\infty}$ and $\tilde{X}_{d_{0}-1}\cong\Sigma\inj(V_{0},V_{1})^{c}_{\infty}$. Further, it is easy to observe that these homeomorphisms are equivariant when $\inj(V_{0},V_{1})$ and $\inj(V_{0},V_{1})^{c}$ are equipped with the conjugation group action. We thus have a unique map $\chi$ that completes the following strictly commutative square, recalling $\kappa$ from \ref{thetaEalpha} and $\tau$ from \ref{sigmainjchomeo}:
$$\xymatrix@C=2cm{\tilde{X}_{d_{0}}\ar[r]_{\cong}^{\kappa}\ar[d]_{\tilde{\pi}_{d_{0}}}&\inj(V_{0},V_{1})_{\infty}\ar[d]^{\chi}\\
\tilde{X}_{d_{0}-1}\ar[r]^{\cong}_{\tau}&\Sigma\inj(V_{0},V_{1})^{c}_{\infty}}
$$

By chasing the various definitions one can check that $\chi$ is explicitly given by the below formulation, recalling $\rho$ and $\sigma$ from \ref{themaprho} and \ref{themapsigma}:
$$\chi(\gamma)=\left(e_{0}(\log(\rho(\gamma))),\sigma(\gamma)\circ(\log(\rho(\gamma))-e_{0}(\log(\rho(\gamma))))\right).$$

Again, it is simple to see that $\chi$ is a $G$-map. We wish to show that the cofibre of $\tilde{\pi}_{d_{0}}$, or equivalently the cofibre of $\chi$, is $S^{\aich\oplus
\Real}$. Corollary \ref{homcofibseq}, however, gives us the following cofibre sequence. The maps $e$, $-\Sigma i$ and $-\Sigma p$ are defined in the corollary and again the maps are $G$-maps:
$$\inj(V_{0},V_{1})_{\infty}\overset{e}{\to}\Sigma\inj(V_{0},V_{1})^{c}_{\infty}\overset{-\Sigma i}{\to} S^{\aich\oplus\Real}\overset{-\Sigma p}{\to}\Sigma \inj(V_{0},V_{1})_{\infty}.$$

Thus by the definition of isomorphisms of cofibre sequences \ref{isomorphistoacofibreseq} we just need to show that $e$ and $\chi$ are homotopic - this will then allow us to replace $e$ in the above sequence with $\chi$, proving that the homotopy cofibre of $\chi$ is $S^{\aich\oplus \Real}$ as claimed. To show this, we first note that by construction the map $e$ is actually of the form $\mathfrak{B}_{f}$, a map built using the extended functional calculus of Proposition
\ref{fnalvarB}. We recall $u'$ and $h'_{0}$ as defined in \ref{D2NDRpair}, as well as the construction \ref{MakingHatWork}. The codomain of $u'$ naturally runs over $[0,1]$ so in the below map we have a suspension taking coordinates in $(0,1)$, thus we also recall our choice of homeomorphism \ref{UnitIntervalHomeomorphism} and abusing notation somewhat assume that the suspension coordinate in the below map actually runs over $\Real$. The following map $f:D_{+}(d_{0})\to S^{1}\wedge D_{+}(d_{0})$ is used to build $e$:
$$f:t\mapsto (u'(t_{0},t_{d-1}))\wedge\widehat{h_{0}'}(t).$$

The map $e$ has domain $\inj(V_{0},V_{1})_{\infty}$ and codomain $\Sigma\inj(V_{0},V_{1})_{\infty}^{c}$ so by Lemma
\ref{factoringthroughinfnalcalc} the domain and codomain of the facial map $f$ are $D_{+}(d_{0})/D_{0}(d_{0})$ and $\Sigma D_{0}(d_{0})$ respectively. Similarly, however, $\chi$ can be written as $\mathfrak{B}_{g}$ for some facial map $g:D_{+}(d_{0})/D_{0}(d_{0})\to \Sigma D_{0}(d_{0})$.

\begin{prop}\label{chiisfnalcalc} $\chi$ is $\mathfrak{B}_{g}$ for a facial map $g:D_{+}(d_{0})/D_{0}(d_{0})\to\Sigma D_{0}(d_{0})$ given as follows:
$$g:t\mapsto (\log(t_{0}))\wedge(\log(t_{1})-\log(t_{0}),...,\log(t_{d_{0}-1})-\log(t_{0})).$$
\end{prop}
\begin{proof} We first note that by Proposition \ref{fnalvarA} we have a map $\mathfrak{A}_{g}:s_{+}(V_{0})_{\infty}\to S^{1}\wedge s_{+}(V_{0})_{\infty}$ built from $g$. We claim that $\chi=\mathfrak{B}_{g}$, which instantly follows by observing that the diagram below commutes and by noting the uniqueness of $\mathfrak{B}_{g}$ from Proposition \ref{fnalvarB}:
$$
\xymatrix{\Ell_{\infty}\wedge s_{+}(V_{0})_{\infty}\ar[rr]^{\mu}\ar[ddd]_{1\wedge \mathfrak{A}_{g}}&&\ar@{->>}[dl]S^{\aich}\ar[ddd]^{\mathfrak{B}_{g}}\\
&\inj(V_{0},V_{1})_{\infty}\ar[d]_{\chi}&\\
&\Sigma\inj(V_{0},V_{1})^{c}_{\infty}\ar@{ >->}[dr]&\\
\Ell_{\infty}\wedge s_{+}(V_{0})_{\infty}\wedge
S^{1}\ar[rr]_{\mu\wedge 1}&&S^{\aich}\wedge S^{1}}
$$
\end{proof}

Thus $e$ and $\chi$ come from facial maps $f$ and $g$ which as detailed in $\S$\ref{TheHomotopyTypeofCertainMapsinFunctionalCalculus} are classified up to homotopy by degree. Hence all we need
to show is the following proposition, while observing that homotopies built out of the extended functional calculus of $\S$\ref{Variations} will be equivariant in this case:

\begin{prop}\label{thedegreesarethesame} Referring to the concept of degree detailed in Remark \ref{whatdegreesmean} the maps $f$ and $g$ are both of degree 1. Thus via Proposition \ref{facialhomotopiesviadegrees} modified as discussed in Remark \ref{whatdegreesmean} $f$ is homotopic to $g$ through facial maps and thus $e$ is homotopic to $\chi$.
\end{prop}
\begin{proof} Via \ref{wereactuallyworkingwithd} we have an inclusion $S^{1}\rightarrowtail D_{+}(d_{0})/D_{0}(d_{0})$ given by $t\mapsto (e^{t},\ldots,e^{t})$. We also have an inclusion $S^{1}\rightarrowtail \Sigma D_{0}(d_{0})$ given by $t\mapsto t\wedge \underline{0}$. Considering $g$ first, we have a map $g':S^{1}\to S^{1}$ given by $t\mapsto t$. This map then makes the below diagram strictly commute: 
$$\xymatrix{S^{1}\ar@{ >->}[r]\ar[d]_{g'}&\frac{D_{+}(d_{0})}{D_{0}(d_{0})}\ar[d]^{g}\\
S^{1}\ar@{ >->}[r]&\Sigma D_{0}(d_{0})}$$

Comments in Remark \ref{whatdegreesmean} detail why the degree of $g$ is given by the degree of $g'$. The map $g'$ is the identity, thus the degree of $g$ is $1$ as claimed.

We now consider the map $f$. Define the map $f':S^{1}\to S^{1}$ as follows:
$$t\mapsto\left\{\begin{array}{ll}
\log\left(\frac{8e^{t}}{1-6e^{t}}\right)& \quad t< -\log(6)\\
\infty& \quad\text{otherwise.}
\end{array}
\right.
$$

This map then makes the below diagram strictly commute, this can be observed by chasing through the definition of $u'$ back through the definitions of $u''$ from \ref{unitdiscNDR} and $\phi$ from \ref{theconformalmapphi} before applying the homeomorphism \ref{UnitIntervalHomeomorphism}. The collapse occurs because we take minimums when considering $u''$.
$$\xymatrix{S^{1}\ar@{ >->}[r]\ar[d]_{f'}&\frac{D_{0}(d_{0})}{D_{0}(d_{0})}\ar[d]^{f}\\
S^{1}\ar@{ >->}[r]&\Sigma D_{0}(d_{0})}$$

Thus to calculate the degree of $f$ we calculate the degree of $f'$. We claim this map is degree $1$. We first observe that we have the below map:
$$f'':\Real\to\Real$$
$$t\mapsto \log\left(\frac{e^{t}}{8+6e^{t}}\right).$$

It is easy to see that $f'$ is the collapse corresponding to this embedding, $f'=(f'')^{!}$. Further $f''$ is strictly increasing. Define a homotopy $h_{s}:\Real\to\Real$ given by $h_{s}(t)=(1-t)f''(t)+st$. Then $h_{s}$ is again a strictly increasing embedding and thus we have collapse maps $h_{s}^{!}:S^{1}\to S^{1}$. These provide a homotopy between $f'$ and the identity. Hence $f'$ has degree $1$. The rest of the proposition then follows from Proposition \ref{facialhomotopiesviadegrees}, Lemma \ref{homotopiesthroughfnalcalc} and Remark \ref{whatdegreesmean}.
\end{proof}

Thus we have the below cofibre sequence:
$$\tilde{X}_{d_{0}}\overset{\tilde{\pi}_{d_{0}}}{\to}\tilde{X}_{d_{0}-1}\to S^{\aich\oplus\Real}\to\Sigma\tilde{X}_{d_{0}}.$$

By the discussion in $\S$\ref{onthecategories} we get a cofibre sequence in $\mathcal{F}_{G}$ by applying the suspension spectrum functor $\Sigma^{\infty}$ throughout. Finally, we smash through the whole sequence by $S^{-\svo}$ to get the following cofibre sequence in $\mathcal{F}_{G}$:
$$X_{d_{0}}\overset{\pi_{d_{0}}}{\to} X_{d_{0}-1}\to S^{\Hom(V_{0},V_{1}-V_{0})\oplus \svo\oplus \Real}\to\Sigma X_{d_{0}}.$$

This proves part $2$ of Theorem \ref{themaintheorem} for the top triangle. Finally, we equate the result with our explicit maps via the proposition below, which is clear:

\begin{prop}\label{thetoptriangleiswhatwethink} We have an isomorphism of cofibre sequences, with the left and right squares commuting strictly:
$$\xymatrix{S^{\aich}\ar[d]_{1}^{=}\ar[r]^{\tilde{\phi}_{d_{0}}}&\tilde{X}_{d_{0}}\ar[d]_{\kappa}^{\cong}\ar[r]^{\tilde{\pi}_{d_{0}}}&\tilde{X}_{d_{0}-1}\ar[d]_{\tau}^{\cong}\ar[r]^{\tilde{\delta}_{d_{0}}}&S^{\aich\oplus \Real}\ar[d]_{1}^{=}\\
S^{\aich}\ar[r]_{p}&\inj(V_{0},V_{1})_{\infty}\ar[r]_{e}&\Sigma\inj(V_{0},V_{1})^{c}_{\infty}\ar[r]_{-\Sigma i}&S^{\aich\oplus \Real}}$$
\end{prop}

\section{Generalizing the Top of the Tower}\label{GeneralizingtheTopoftheTower}

We now discuss how this work all generalizes up to the level of bundles via Lemma \ref{cofibseqfibrebundles}. Let $W$ and $V$ be bundles over a base space $X$. Then we have $\inj(W_{x},V_{x})_{\infty}$, $\Sigma \inj(W_{x},V_{x})_{\infty}^{c}$ and $S^{\Hom(W_{x},V_{x})\oplus \Real}$ as $X$-parameterized based families of spaces. We have cofibre sequences on each fibre by Corollary \ref{homcofibseq} and we have sections $\sigma_{i}$ that
send $x$ to the basepoint of each space. Thus by Lemma \ref{cofibseqfibrebundles} we have the below cofibre sequence. Note that we are being deliberately vague here about topology and equivariance, we wish in this section to simply detail
the method of generalization - how we concretely use this will be covered in the next section:
$$\frac{\inj(W,V)_{\infty}}{\sigma_{1}(X)}\to\frac{\Sigma \inj(W,V)_{\infty}^{c}}{\sigma_{2}(X)}\to \frac{\Sigma\Hom(W,V)_{\infty}}{\sigma_{3}(X)}\to \frac{\Sigma \inj(W,V)_{\infty}}{\sigma_{1}(X)}.$$

Assume each $W_{x}$ is of dimension $d$ and each $V_{x}$ is such that $\dim(V_{x})\geqslant d$. We have a well-defined facial map $g:D_{+}(d)/D_{0}(d)\to \Sigma D_{0}(d)$ given as follows, a simple dimension modification of the map $g$ defined in Proposition \ref{chiisfnalcalc}:
$$g(t):=(\log(t_{0}))\wedge(\log(t_{1})-\log(t_{0}),\ldots,\log(t_{d-1})-\log(t_{0})).$$

This map is used to build maps $\chi_{x}:\inj(W_{x},V_{x})_{\infty}\to \Sigma \inj(W_{x},V_{x})_{\infty}^{c}$ which take similar explicit formulations to that of the map $\chi$ introduced in the previous section. Each $\chi_{x}$ will lie in a
cofibre sequence exactly as $\chi$ did - $\chi_{x}$ will be homotopic to the map $e_{x}$ as defined in Corollary
\ref{homcofibseq} when applied to the pair $(W_{x},V_{x})$. Thus we have another family of based cofibre sequences over $X$ and hence we have the below cofibre sequence:
$$\frac{\inj(W,V)_{\infty}}{\sigma_{1}(X)}\overset{\chi_{X}}{\to}\frac{\Sigma \inj(W,V)_{\infty}^{c}}{\sigma_{2}(X)}\to \frac{\Sigma\Hom(W,V)_{\infty}}{\sigma_{3}(X)}\to \frac{\Sigma \inj(W,V)_{\infty}}{\sigma_{1}(X)}.$$

Finally, we note here that there is one further generalization we can take, that of `smashing' throughout by a bundle. Let $U$ be an arbitrary fibre bundle of based spaces over $X$, then one can smash by $U_{x}$ in each cofibre sequence via Lemma \ref{smashingalongacofibseq}. We still have sections etc and the process detailed above follows through
exactly as previously detailed, giving us the following cofibre sequence:
\small $$\frac{\inj(W,V)_{\infty}\wedge U}{\sigma_{1}(X)}\overset{\chi_{X}}{\to}\frac{\Sigma \inj(W,V)_{\infty}^{c}\wedge U}{\sigma_{2}(X)}\to\frac{\Sigma\Hom(W,V)_{\infty}\wedge U}{\sigma_{3}(X)}\to \frac{\Sigma \inj(W,V)_{\infty}\wedge U}{\sigma_{1}(X)}.$$
\normalsize

This thus gives a systematic method of approach - we identify a base space that will (with a smash) yield the right
cofibre, prove that our spaces are in fact the bundles $\inj(W,V)_{\infty}\wedge U/\sigma_{1}(X)$ and $\Sigma
\inj(W,V)_{\infty}^{c}\wedge U/\sigma_{2}(X)$ and prove that the map is of the form $\chi_{X}$, i.e. that the map on each fibre is $\chi_{x}$. We now do this explicitly.

\section{The Other Triangles}\label{TheOtherTriangles}

We now consider the map $\tilde{\pi}_{k}:\tilde{X}_{k}\to \tilde{X}_{k-1}$ for $k<d_{0}$. For now we forget that we have stated maps $\tilde{\delta}_{k}$ and $\tilde{\phi}_{k}$ and concentrate simply on finding the cofibre of $\tilde{\pi}_{k}$. We will exhibit $\tilde{\delta}_{k}$ naturally in the proof and lift to $\tilde{\phi}_{k}$ in the next chapter when we detail how to bypass certain technical necessities we encounter. 

We want our cofibres to be $G_{k}(V_{0})^{\Real\oplus\Hom(T,V_{1})\oplus s(T^{\bot})}$. This suggests taking $G_{k}(V_{0})$ as the base space in the systematic approach put forward at the end of the last section. Further, take the bundle denoted by $W$ in the last section to be the tautological bundle $T$, and $V$ to be the representation $V_{1}$. We also have the bundle $s(T^{\bot})$, each fibre can be given a one-point compactification which will allow us to smash with $\Sigma\Hom(T,V_{1})_{\infty}$ as noted. It is then easy to see that we topologize one of the spaces in the sequence in the previous section as follows:
$$G_{k}(V_{0})^{\Real\oplus\Hom(T,V_{1})\oplus s(T^{\bot})}=\frac{\Sigma\Hom(T,V_{1})_{\infty}\wedge s(T^{\bot})}{\sigma_{3}(G_{k}(V_{0}))}.$$

Thus we can follow the systematic approach of $\S$\ref{GeneralizingtheTopoftheTower} and build a cofibre sequence; our aim is to show that $\tilde{X}_{k}$, $\tilde{X}_{k-1}$ and $\tilde{\pi}_{k}$ fit into this idea in some way. We can simplify the problem, however. We recall the space $Y_{k}'$ defined in \ref{makingthecofibseqeasier} and note here that it is both a subspace of $\tilde{X}_{k}'$ and of $\tilde{X}_{k-1}'$ and it is closed under the action of $G$. Thus we can consider the induced map $\tilde{X}_{k}/Y_{k}\to \tilde{X}_{k-1}/Y_{k}$ - this will naturally have the same cofibre as $\tilde{\pi}_{k}$ by the theory of $\S$\ref{TotalCofibres}.

Hence we wish to look at $\tilde{X}_{k}/ Y_{k}$, $\tilde{X}_{k-1}/ Y_{k}$, and the associated unbased spaces $\tilde{X}_{k}'\backslash Y_{k}'$ and $\tilde{X}_{k-1}'\backslash Y_{k}'$. Identifying $\tilde{X}_{k}/Y_{k}$ and
$\tilde{X}_{k-1}/Y_{k}$ back into the ideas of $\S$\ref{GeneralizingtheTopoftheTower} follows from the
below two homeomorphisms:

\begin{prop}\label{thetaEalphaextended} Note that:
$$\tilde{X}_{k}'\backslash Y_{k}'=\{(\alpha,\theta):\alpha\in\svo,\Dim(P_{k}(\alpha))=k,\theta\in\mathcal{L}(P_{k}(\alpha),V_{1})\}.$$
Define the space $\mathcal{I}'_{k}$ by:
$$\mathcal{I}_{k}':=\{(W,\gamma,\psi):W\in
G_{k}(V_{0}),\gamma\in\inj(W,V_{1}),\psi\in s(W^{\bot})\}.$$

Then we topologize $\mathcal{I}_{k}'$ as a subspace of $\tilde{Z}_{k}$ using the method of
$\S$\ref{TheTopologyOfBundlesOverGrassmannians}. We set $\mathcal{I}_{k}$ to be the one-point compactification $(\mathcal{I}_{k}')_{\infty}$ and note that this space has a $G$-action inherited from the action on $\tilde{Z}_{k}$. Now note that as sets we have the following identification, we topologize to make this into a homeomorphism:
$$\mathcal{I}_{k}\cong\frac{\inj(T,V_{1})_{\infty}\wedge s(T^{\bot})}{\sigma_{1}(G_{k}(V_{0}))}.$$

Recalling $\rho$ and $\sigma$ from \ref{themaprho} and \ref{themapsigma}, the maps $\mathfrak{q}_{k}:\tilde{X}_{k}'\backslash Y_{k}'\to \mathcal{I}'_{k}$ and $\mathfrak{r}_{k}:\mathcal{I}'_{k}\to \tilde{X}_{k}'\backslash Y_{k}'$ stated below are well-defined continuous $G$-maps that are inverses of each other.
$$\mathfrak{q}_{k}:(\alpha,\theta)\mapsto\left(P_{k}(\alpha),-\theta\circ \Exp(\alpha|_{P_{k}(\alpha)}),-\log((e_{d_{0}-k}(\alpha)-\alpha)|_{P_{k}(\alpha)^{\bot}})\right)$$
$$\left((\log(e_{0}(\rho(\gamma)))-\Exp(-\psi))|_{W^{\bot}}\oplus \log(\rho(\gamma))|_{W},-\sigma(\gamma) \right)\mapsfrom(W,\gamma,\psi):\mathfrak{r}_{k}.$$

Thus we have that $\tilde{X}_{k}'\backslash Y_{k}'\cong \mathcal{I}'_{k}$ and hence that $\tilde{X}_{k}/ Y_{k}\cong \mathcal{I}_{k}$.
\end{prop}

\begin{proof} To prove this we need to check that $\mathfrak{q}_{k}$ and $\mathfrak{r}_{k}$ are both well-defined, that they are inverses of each other and that they are both continuous; that they are $G$-maps is standard to check. Throughout we implicitly assume use of Lemma \ref{fnalcalclemma} in order to write down the needed expressions,
we note where any issues crop up when using this.

First consider the map $\mathfrak{q}_{k}$, we wish to check that the map lands in the specified domain. The only issue for this is checking that $-\log((e_{d_{0}-k}(\alpha)-\alpha)|_{P_{k}(\alpha)^{\bot}})$ is selfadjoint and lands in $P_{k}(\alpha)^{\bot}$ - the first factor is patently an element of the Grassmannian and the map $-\theta\circ
\Exp(\alpha|_{P_{k}(\alpha)})$ is injective by similar techniques to those used in \ref{thetaEalpha}. Firstly, we note
that $(P_{k}(\alpha))^{\bot}$ is the sum of the first $d_{0}-k$ eigenspaces of $\alpha$. Thus $(e_{d_{0}-k}(\alpha)-\alpha)|_{P_{k}(\alpha)^{\bot}}$ has positive eigenvalues, the bottom $d_{0}-k$ eigenvalues of $\alpha$ are the only ones which come into play and all are less than or equal to $e_{d_{0}-k}$.
Moreover, as $P_{k}(\alpha)$ has been assumed to have dimension precisely $k$ we have that $e_{d_{0}-k-1}$ is strictly less than $e_{d_{0}-k}$. Thus it follows that $(e_{d_{0}-k}(\alpha)-\alpha)|_{P_{k}(\alpha)^{\bot}}$ is strictly
positive. It is also clearly selfadjoint by the distributivity of $^{\dag}$ over addition. Finally it is clear that the domain of $(e_{d_{0}-k}(\alpha)-\alpha)|_{P_{k}(\alpha)^{\bot}}$ is $P_{k}(\alpha)^{\bot}$. Thus we have an element of
$s_{++}(P_{k}(\alpha)^{\bot})$ which by Corollary \ref{shomeotosplus} maps to a general element of $s(P_{k}(\alpha)^{\bot})$ under $\log$. Further, there is no issue with the minus sign as this provides a self-homeomorphism of $s(P_{k}(\alpha)^{\bot})$. This proves that $\mathfrak{q}_{k}$ is well-defined.

To check that $\mathfrak{r}_{k}$ is well-defined we first need to show that if we set $\alpha'$ to be as below then $P_{k}(\alpha')$ has dimension $k$:
$$\alpha'=(\log(e_{0}(\rho(\gamma)))-\Exp(-\psi))|_{W^{\bot}}\oplus(\log(\rho(\gamma)))|_{W}.$$

We note that $\alpha'$ in this case is trivially seen to be selfadjoint using Lemma \ref{fnalcalclemma} so it has ordered eigenvalues and the concept of $P_{k}$ makes sense. To calculate the dimension of $P_{k}(\alpha')$ we show that $P_{k}(\alpha')=W$. First note that $(\log(e_{0}(\rho(\gamma)))-\Exp(-\psi))$ has $d_{0}-k$ eigenvalues
that are all strictly less than $\log(e_{0}(\rho(\gamma)))$, the minimal eigenvalue of $\alpha'|_{W}$. Hence $P_{k}(\alpha')=W$ and thus $P_{k}(\alpha')$ has the right dimension.

We also check $-\sigma(\gamma)$ is well-defined, but from all the earlier remarks this is a trivial exercise - we have the identification $W=P_{k}(\alpha')$ from above and we also note that that as $\gamma$ is injective $\ker(\gamma)$ is zero. Thus $-\sigma(\gamma)$ is a reasonable isometry out of $W$. This checks that $\mathfrak{r}_{k}$ is well-defined.

We now prove that $\mathfrak{q}_{k}\circ \mathfrak{r}_{k}$ is the identity on $\mathcal{I}'_{k}$. Set
$\theta':=-\sigma (\gamma)$ and set $\alpha'$ as above. We have already demonstrated that $P_{k}(\alpha')$ is $W$, next note that by definition $\alpha'|_{P_{k}(\alpha')}=\log(\rho(\gamma))$ and thus that the following holds: $$-\theta'\circ \Exp(\alpha'|_{P_{k}(\alpha')})=\sigma(\gamma)\circ \Exp(\log(\rho(\gamma)))=\sigma(\gamma)\circ\rho(\gamma)=\gamma.$$

Finally we note that we have set $\alpha'|_{P_{k}(\alpha')^{\bot}}$ to be $\log(e_{0}(\rho(\gamma)))-\Exp(-\psi)$. We also have $\log(e_{0}(\rho(\gamma)))=e_{d_{0}-k}(\alpha')$. Thus the below expression holds:
$$-\log((e_{d_{0}-k}(\alpha')-\alpha')|_{P_{k}(\alpha)^{\bot}})=-\log(\Exp(-\psi))=\psi.$$

This proves that $\mathfrak{q}_{k}\circ \mathfrak{r}_{k}$ is the identity. We now check that $\mathfrak{r}_{k}\circ \mathfrak{q}_{k}$ is the identity on $\tilde{X}_{k}'\backslash Y_{k}'$. To do this, set $\gamma'$ to be $-\theta\circ
\Exp(\alpha|_{P_{k}(\alpha)})$, set $\psi'$ to be $-\log((e_{d_{0}-k}(\alpha)-\alpha)|_{P_{k}(\alpha)^{\bot}})$ and
set $W'$ to be $P_{k}(\alpha)$. First consider the expression $\log(\rho(\gamma'))|_{W'}$ - recall that
$\rho(\gamma')=(\gamma'^{\dag}\gamma')^{1/2}$ gives $\rho(\gamma')=\Exp(\alpha|_{P_{k}(\alpha)})$ by noting that $\theta$
is an isometry. Thus $\log(\rho(\gamma'))|_{W'}=\alpha|_{P_{k}(\alpha)}$. Similarly we note that $\Exp(-\psi')=(e_{d_{0}-k}(\alpha)-\alpha)|_{P_{k}(\alpha)^{\bot}}$ and that $\log(e_{0}(\rho(\gamma')))$ is $e_{d_{0}-k}(\alpha)$. This makes the below identity hold:
$$(\log(e_{0}(\rho(\gamma)))-\Exp(-\psi))|_{W^{\bot}}=\alpha|_{P_{k}(\alpha)^{\bot}}.$$
To conclude we observe that $-\sigma (\gamma')$ is by definition $-\gamma'\circ (\gamma'^{\dag}\gamma')^{-1/2}$, this is well-defined as $\gamma'$ is injective. The $(\gamma'^{\dag}\gamma')^{-1/2}$ part is easily seen to be $\Exp(\alpha|_{P_{k}(\alpha)})^{-1}$ when you recall the definition of $\gamma'$ from above, thus we conclude that
$\sigma(\gamma')=\theta$. Combined with the above remarks, this shows that $\mathfrak{r}_{k}\circ \mathfrak{q}_{k}$ is the identity.

To check that $\mathfrak{r}_{k}$ and $\mathfrak{q}_{k}$ are continuous we must be concrete in our topology on $\mathcal{I}'_{k}$. Define a subspace topology on $\mathcal{I}'_{k}$ inherited from Corollary \ref{factoringinsTperp} as $\mathcal{I}'_{k}$ is a subspace of $\tilde{Z}_{k}$, this makes $\mathcal{I}'_{k}$ into a subspace of $G_{k}(V_{0})\times \aich\times\svo$. We now claim that $\mathfrak{q}_{k}$ is continuous. First equip $\tilde{X}_{k}'\backslash Y_{k}'$ with the quotient topology. Under this it is sufficient to show that the map
below is continuous, as continuity will then be inherited from taking the quotient and the subspace:
$$\{(\alpha,\theta)\in\svo\times\Ell:\Dim(P_{k}(\alpha))=k\} \to G_{k}(V_{0})\times \aich\times\svo$$
$$(\alpha,\theta)\mapsto(P_{k}(\alpha),-\theta\circ \Exp(\alpha),-\log(e_{d_{0}-k}(\alpha)-\alpha)).$$

The eigenfunctions have been shown to be continuous in Lemma \ref{continuouseigenvalues} and $\Exp$ and $\log$ are continuous in this context via Lemma \ref{fnalcalclemma}. Finally the map to the first factor $(\alpha,\theta)\mapsto P_{k}(\alpha)$ inherits continuity from the continuous map $P_{k}:s_{k}(V_{0})\to G_{k}(V_{0})$ as given in Corollary \ref{Pkiscont}.

Regarding the continuity of $\mathfrak{r}_{k}$, consider the subspace topology on $\tilde{X}_{k}'\backslash
Y_{k}'$. This allows us to instead consider the continuity of the following map:
$$G_{k}(V_{0})\times \aich\times\svo \to\svo\times\aich$$
$$(W,\gamma,\psi)\mapsto ((\log(e_{0}(\rho(\gamma)))-\Exp(-\psi))|_{W^\bot}\oplus \log(\rho(\gamma))|_{W},\gamma).$$

It is easy to see that under the identifications made this gives the map $\mathfrak{r}_{k}$ on the required subspaces. It is also simple to see from here that the map into the second factor comes from the identity. Thus continuity will follow from the continuity of the below map:
$$G_{k}(V_{0})\times \aich\times\svo\to \svo$$
$$(W,\gamma,\psi)\mapsto ((\log(e_{0}(\rho(\gamma)))-\Exp(-\psi))|_{W^\bot}\oplus \log(\rho(\gamma))|_{W}).$$

To show continuity here, consider instead the identification of $G_{k}(V_{0})$ with $G_{k}'(V_{0})$. If $\pi\in G_{k}'(V_{0})$ then by Note \ref{perpishomeo} we can instead think of the map below:
$$(\pi,\gamma,\psi)\mapsto ((\log(e_{0}(\rho(\gamma)))-\Exp(-\psi))\circ(1_{V_{0}}-\pi)\oplus \log(\rho(\gamma))\circ\pi).$$

This is continuous by the standard functional calculus and continuous eigenvalue lemmas. Thus continuity of $\mathfrak{r}_{k}$ follows. This then proves all that is required to prove the proposition.
\end{proof}

\begin{prop}\label{sigmainjchomeoextended} Note that:
$$\tilde{X}_{k-1}'\backslash
Y_{k}'=\{(\alpha,\theta):\alpha\in\svo,\Dim(P_{k}(\alpha))=k,\theta\in\mathcal{L}(P_{k-1}(\alpha),V_{1})\}.$$
Define the space $\mathcal{J}_{k}'$ by:
$$\mathcal{J}_{k}':=\{(W,\delta,\psi):W\in
G_{k}(V_{0}),\delta\in\inj(W,V_{1})^{c},\psi\in s(W^{\bot})\}.$$

Then we topologize $\mathcal{J}_{k}'$ as a subspace of $\tilde{Z}_{k}$ using the method of $\S$\ref{TheTopologyOfBundlesOverGrassmannians}. We set $\mathcal{J}_{k}$ to be the one-point compactification $(\mathcal{J}_{k}')_{\infty}$ and note that this space has a $G$-action inherited from the action on $\tilde{Z}_{k}$. Now note that as
sets we have the following identification of $\mathcal{J}_{k}$, we topologize to make this into a homeomorphism:
$$\Sigma \mathcal{J}_{k}\cong\frac{\Sigma\inj(T,V_{1})_{\infty}^{c}\wedge s(T^{\bot})}{\sigma_{2}(G_{k}(V_{0}))}.$$

Recalling $\rho$, $\sigma$ and $\lambda_{k-1}$ from \ref{themaprho} and \ref{themapsigma} and \ref{lambdak}, the maps $\mathfrak{f}_{k}:\tilde{X}_{k-1}'\backslash Y_{k}'\to \Real\times \mathcal{J}_{k}'$ and $\mathfrak{g}_{k}:\Real\times \mathcal{J}_{k}'\to \tilde{X}_{k-1}'\backslash Y_{k}'$ stated below are well-defined continuous $G$-maps that are
inverses of each other.
\small $$\mathfrak{f}_{k}:(\alpha,\theta)\mapsto\left(e_{d_{0}-k}(\alpha),P_{k}(\alpha),-\theta\circ \lambda_{k-1}(\alpha)|_{P_{k}(\alpha)},-\log((e_{d_{0}-k}(\alpha)-\alpha)|_{P_{k}(\alpha)^{\bot}})\right)$$
$$\left((t-\Exp(-\psi))|_{W^{\bot}}\oplus (\rho(\delta)+t)|_{W}, -\sigma(\delta)\right)\mapsfrom(t,W,\delta,\psi):\mathfrak{g}_{k}.$$
\normalsize

Thus we have that $\tilde{X}_{k-1}'\backslash Y_{k}'\cong \Real\times \mathcal{J}_{k}'$ and hence that $\tilde{X}_{k-1}/ Y_{k}\cong\Sigma \mathcal{J}_{k}$.
\end{prop}
\begin{proof} Again, we note that throughout we implicitly invoke Lemma \ref{fnalcalclemma} and related results in order to write down certain identifications. To build the two maps we proceed in two stages. Firstly we note that we have a homeomorphism onto the space below given by $(\alpha, \theta)\mapsto (\alpha, -\theta\circ\lambda_{k-1}(\alpha))$:
$$\tilde{X}_{k-1}'\backslash Y_{k}'\cong\{(\alpha,\beta):\alpha\in\svo, \beta:V_{0}\to V_{1}, \rho(\beta)=\lambda_{k-1}(\alpha), \Dim(P_{k}(\alpha))=k\}.$$

We take the map out of this space given by:
$$\mathfrak{f}_{k}':(\alpha, \beta)\mapsto (e_{d_{0}-k}(\alpha), P_{k}(\alpha), \beta|_{P_{k}(\alpha)},-\log((e_{d_{0}-k}(\alpha)-\alpha)|_{P_{k}(\alpha)^{\bot}})).$$

It is easy to see that this composition is $\mathfrak{f}_{k}$. We claim that it is a reasonable map. We also note that we can consider the map below, this maps into the space above identified as homeomorphic to $\tilde{X}_{k-1}'\backslash Y_{k}'$. Again, it can easily be seen that composing this map with the above homeomorphism retrieves $\mathfrak{g}_{k}$:
$$\mathfrak{g}_{k}':(t,W, \delta, \psi)\mapsto ((t-\Exp(-\psi))|_{W^{\bot}}\oplus (\rho(\delta)+t)|_{W},\delta|_{W}\oplus 0|_{W^{\bot}}).$$

We claim that $\mathfrak{g}_{k}$ too is reasonable, that both are inverses of each other and that both are continuous. Again, checking that both maps are $G$-maps is standard.

We firstly consider the above factorization of $\mathfrak{f}_{k}$, which we refer to as $\mathfrak{f}_{k}'$. We need to check that $\mathfrak{f}_{k}'$ lands in the right domain, this is clear for the projection onto the first two factors and by the previous proof it also follows for the projection to the last factor. Thus the only issue is to show that $\beta|_{P_{k}(\alpha)}$ is a non-injective map out of $P_{k}(\alpha)$. To do this, note that $\Ker(\beta)=\Ker(\rho(\beta))=\Ker(\lambda_{k-1}(\alpha))$. The image of $\lambda_{k-1}(\alpha)$ is $P_{k-1}(\alpha)$ which we note is strictly smaller than $P_{k}(\alpha)$ as we have conditioned that $\Dim(P_{k}(\alpha))=k$. Thus the kernel of $\beta$ contains something in $P_{k}(\alpha)$ - the bit that is not in $P_{k-1}(\alpha)$ and thus the restriction is non-injective.

We now check that $\mathfrak{g}_{k}'$, the map stated above as being equivalent to $\mathfrak{g}_{k}$ via the homeomorphism, is well-defined. To do this, set $\alpha':=(t-\Exp(-\psi))|_{W^{\bot}}\oplus (\rho(\delta)+t)|_{W}$ and
$\beta':=\delta|_{W}\oplus 0|_{W^{\bot}}$. We need to prove that $P_{k}(\alpha')$ is of dimension $k$ and that
$\rho(\beta')=\lambda_{k-1}(\alpha')$ - the rest is trivial. For the first point we prove that $P_{k}(\alpha')$ is actually $W$. Note that $t-\Exp(-\psi)$ has $d_{0}-k$ eigenvalues all strictly less than $t$ as $\Exp(-\psi)$ is strictly positive. Moreover, note that $(\rho(\delta)+t)|_{W}$ has $k$ eigenvalues all greater than or
equal to $t$ as $\rho(\delta)$ is weakly positive. This allows us to conclude that $P_{k}(\alpha')$ is $W$ and thus that its dimension is $k$.

We now wish to prove that $\rho(\beta')=\lambda_{k-1}(\alpha')$. To do this, recall that $\lambda_{k-1}(\alpha')$ is zero on the bottom $d_{0}-k+1$ eigenspaces of $\alpha'$ and is $\alpha'-e_{d_{0}-k}(\alpha')$ otherwise. From the remarks above, the eigenvalue $e_{d_{0}-k}(\alpha')$ is the lowest eigenvalue of $(\rho(\delta)+t)$ which is $t$ as $\delta$ is non-injective. Thus it is easy to see that $\lambda_{k-1}(\alpha')=\rho(\delta)\oplus 0|_{W^{\bot}}$ which is $\rho(\beta')$. Thus $\mathfrak{g}_{k}'$ is well-defined.

We now prove that $\mathfrak{f}_{k}'\circ \mathfrak{g}_{k}'$ is the identity. Let $\alpha'$ and $\beta'$ be given by $(t-\Exp(-\psi))|_{W^{\bot}}\oplus (\rho(\delta)+t)|_{W}$ and $\delta|_{W}\oplus 0|_{W^{\bot}}$ respectively. We first note that $e_{d_{0}-k}(\alpha')$ is going to patently be $t$ for the reasons outlined above. Similarly, from the
above we can also conclude that $P_{k}(\alpha')$ will map to $W$ and thus that $\beta'|_{P_{k}(\alpha')}$ will map to
$\delta$. Finally we check that $-\log((e_{d_{0}-k}(\alpha')-\alpha')|_{P_{k}(\alpha')^{\bot}})$ is just $\psi$ but this follows as detailed below from what has already been shown:
$$-\log((e_{d_{0}-k}(\alpha')-\alpha')|_{P_{k}(\alpha')^{\bot}})\cong -\log(t-t+\Exp(-\psi))\cong\psi.$$

This proves that $\mathfrak{f}_{k}'\circ \mathfrak{g}_{k}'$ is the identity. To check that $\mathfrak{g}_{k}'\circ \mathfrak{f}_{k}'$ is the identity we set $t'$, $W'$, $\delta'$ and $\psi'$ to be $e_{d_{0}-k}(\alpha)$, $P_{k}(\alpha)$,
$\beta|_{P_{k}(\alpha)}$ and $-\log((e_{d_{0}-k}(\alpha)-\alpha)|_{P_{k}(\alpha)^{\bot}})$ respectively. We first claim that the map $\beta|_{P_{k}(\alpha)}\oplus 0|_{P_{k}(\alpha)^{\bot}}$ is just the map $\beta$, but this follows from the remarks above on the kernel of $\beta$ - it is $P_{k-1}(\alpha)^{\bot}$ which contains $P_{k}(\alpha)^{\bot}$. We now consider $t'-\Exp(-\psi')$ and wish to show this is $\alpha|_{P_{k}(\alpha)^{\bot}}$, but this follows immediately from the identifications given above. Finally, we wish to show that $\rho(\delta')+t'$ is $\alpha|_{P_{k}(\alpha)}$
but this follows from the below identity:
$$\rho(\delta')+t'=\rho(\beta|_{P_{k}(\alpha)})+e_{d_{0}-k}(\alpha)=(\lambda_{k}(\alpha)+e_{d_{0}-k}(\alpha))|_{P_{k}(\alpha)}=\alpha|_{P_{k}(\alpha)}.$$

Thus $\mathfrak{f}_{k}'\circ \mathfrak{g}_{k}'$ is the identity.

We now check the continuity of $\mathfrak{f}_{k}$ and $\mathfrak{g}_{k}$. Firstly, we note that $\Real\times \mathcal{J}_{k}'$ is topologized by Corollary \ref{factoringinsTperp} similarly to the topology on $\mathcal{I}_{k}'$ - $\Real\times \mathcal{J}_{k}'$ is a subspace of $\Real\times G_{k}(V_{0})\times\aich\times\svo$. To check the continuity of $\mathfrak{f}_{k}$, equip $\tilde{X}_{k-1}'\backslash Y_{k}'$ with the quotient topology. Again, it is sufficient to prove that the below map is continuous - the restriction to subspace and quotient then produces continuity on the level we desire:
\small $$\{(\alpha,\theta)\in\svo\times\Ell:\Dim(P_{k}(\alpha))=k\}\to \Real\times G_{k}(V_{0})\times\aich\times\svo$$
\normalsize
$$(\alpha,\theta)\mapsto(e_{d_{0}-k}(\alpha),P_{k}(\alpha),-\theta\circ(\alpha-e_{d_{0}-k}(\alpha)),-\log(e_{d_{0}-k}(\alpha)-\alpha)).$$

As in the previous proof, continuity follows from Lemmas \ref{fnalcalclemma} and \ref{continuouseigenvalues} and Corollary \ref{Pkiscont}.

Regarding the continuity of $\mathfrak{g}_{k}$, equipping $\tilde{X}_{k-1}'\backslash Y_{k}'$ with the subspace topology allows us to observe that continuity follows from the continuity of the below map:
$$\Real\times G_{k}'(V_{0})\times\aich\times\svo\mapsto \svo\times\aich$$
$$(t,\pi,\delta,\psi)\mapsto ((t-\Exp(-\psi))\circ(1_{V_{0}}-\pi)\oplus(\rho(\delta)+t)\circ\pi,\delta).$$

Similar to the previous proof this is patently going to restrict to the map $\mathfrak{g}_{k}$ by Note \ref{perpishomeo} and it is patently continuous by Lemmas \ref{fnalcalclemma} and \ref{continuouseigenvalues}. Thus continuity on $\mathfrak{g}_{k}$ follows and this is enough to prove all that is claimed in the proposition.
\end{proof}

Now note that there is a unique $G$-map $\chi'$ making the below diagram strictly commute:
$$\xymatrix@C=2cm{\tilde{X}_{k}\ar[d]_{\tilde{\pi}_{k}}\ar[r]^{\text{collapse}}&\frac{\tilde{X}_{k}}{Y_{k}}\ar[r]_{\cong}^{\mathfrak{q}_{k}}\ar[d]_{\text{proj.}}&\mathcal{I}_{k}\ar[d]^{\chi'}\\
\tilde{X}_{k-1}\ar[r]_{\text{collapse}}&\frac{\tilde{X}_{k-1}}{Y_{k}}\ar[r]^{\cong}_{\mathfrak{f}_{k}}&\Sigma\mathcal{J}_{k}}
$$

As discussed during the conclusion of the previous section, we have the following cofibre sequence, moreover with the equipped actions the maps are $G$-maps:
$$\mathcal{I}_{k}\to\Sigma \mathcal{J}_{k}\to G_{k}(V_{0})^{\Real\oplus\Hom(T,V_{1})\oplus
s(T^{\bot})}\to \Sigma \mathcal{I}_{k}.$$

We can also build a commutative diagram as follows:
$$\xymatrix{\tilde{X}_{k}/Y_{k}\ar[d]_{\cong}\ar[r]&\tilde{X}_{k-1}/Y_{k}\ar[d]_{\cong}\ar@{.>}[dr]&&\\
\mathcal{I}_{k}\ar[r]_{\chi'}&\Sigma
\mathcal{J}_{k}\ar[r]&G_{k}(V_{0})^{\Real\oplus\Hom(T,V_{1})\oplus
s(T^{\bot})}\ar[r]&\Sigma \mathcal{I}_{k} }
$$

We now deduce Theorem \ref{themaintheorem} by proving the following proposition:

\begin{prop}\label{chiprimeischi} The map $\chi'$ is built from maps $\chi_{W}:\inj(W,V_{1})_{\infty}\to
\Sigma\inj(W,V_{1})^{c}_{\infty}$ for each $W\in G_{k}(V_{0})$. Hence via the construction detailed in $\S$\ref{GeneralizingtheTopoftheTower} the sequence below is a cofibre sequence:
$$\mathcal{I}_{k}\overset{\chi'}{\to}\Sigma\mathcal{J}_{k}\to G_{k}(V_{0})^{\Real\oplus\Hom(T,V_{1})\oplus
s(T^{\bot})}\to \Sigma \mathcal{I}_{k}.$$

From the diagram above and Definition \ref{isomorphistoacofibreseq} we conclude that the cofibre of the map $\tilde{X}_{k}/Y_{k}\to \tilde{X}_{k-1}/Y_{k}$ is $G_{k}(V_{0})^{\Real\oplus\Hom(T,V_{1})\oplus s(T^{\bot})}$ and thus
by Lemma \ref{quotientcofibresequences} the cofibre of $\tilde{\pi}_{k}$ is $G_{k}(V_{0})^{\Real\oplus\Hom(T,V_{1})\oplus
s(T^{\bot})}$.
\end{prop}
\begin{proof}Fixing $W\in G_{k}(V_{0})$, we have the maps $\chi'_{W}:\inj(W,V_{1})_{\infty}\wedge s(W^{\bot})_{\infty}\to
\Sigma\inj(W,V_{1})^{c}_{\infty}\wedge s(W^{\bot})_{\infty}$ given by:
$$(\gamma,\psi)\mapsto \left(e_{0}(\log(\rho(\gamma))),\sigma(\gamma)\circ(\log(\rho(\gamma))-e_{0}(\log(\rho(\gamma)))),\psi\right).$$

These maps are clearly built from maps $\chi_{W}:\inj(W,V_{1})_{\infty}\to \Sigma\inj(W,V_{1})^{c}_{\infty}$ as constructed in the previous section. The rest of the result then follows from the previously cited results and discussion.
\end{proof}

This thus proves part $2$ of Theorem \ref{themaintheorem} - apply $\Sigma^{\infty}$ and smash by $S^{-\svo}$ to get the following cofibre triangles in $\mathcal{F}_{G}$:
$$\xymatrix{X_{k}\ar[d]_{\pi_{k}}&G_{k}(V_{0})^{\Hom(T,V_{1} - V_{0})\oplus s(T)}\ar[l]\\
X_{k-1}\ar[ur]|\bigcirc}
$$

These are precisely the triangles in the tower of \ref{themaintheorem} and that they are cofibre sequences follows from
the discussion in $\S$\ref{onthecategories}. Theorem \ref{themaintheorem} then follows. Moreover, we can equate one of our explicit map statements to the above proof in the following way:

\begin{prop}\label{deltakisright} The map $\tilde{\delta}_{k}:\tilde{X}_{k-1}\to G_{k}(V_{0})^{\Real\oplus\Hom(T,V_{1})\oplus s(T^{\bot})}$ as stated in \ref{tildedeltak} is the composition below, $-\Sigma i_{k}$ being the twisted suspension of the standard fibrewise inclusion map $i_{k}:\mathcal{J}_{k}\to G_{k}(V_{0})^{\Hom(T,V_{1})\oplus s(T^{\bot})}$:
$$\tilde{\delta}_{k}:\tilde{X}_{k-1}\overset{\text{coll.}}{\to}\frac{\tilde{X}_{k-1}}{Y_{k}}\overset{\mathfrak{f}_{k}}{\cong}\mathcal{J}_{k}\overset{-\Sigma i_{k}}{\to}G_{k}(V_{0})^{\Real\oplus\Hom(T,V_{1})\oplus
s(T^{\bot})}.$$
Moreover the composition $\tilde{X}_{k-1}/Y_{k}\to G_{k}(V_{0})^{\Real\oplus\Hom(T,V_{1})\oplus s(T^{\bot})}$ coincides with the map in the cofibre sequence constructed above, and hence $\delta_{k}:=\Sigma^{-\svo}\tilde{\delta}_{k}$ is the map in the stable cofibre sequence.
\end{prop}

This is clear to observe. In particular, it demonstrates that $\tilde{\delta}_{k}$ is a well-defined continuous $G$-map. We now proceed to look at the further topological properties of the result, starting with equating the results above to our statement of a candidate map $\tilde{\phi}_{k}$.
\chapter{Further Topological Properties}
\label{ch:ch6}

\section{Lifting the Tower}\label{BypassingtheQuotientbyaLift}

We have constructed the below diagram, the right triangle being the cofibre sequence constructed in $\S$\ref{TheOtherTriangles}. We recall from that section the definitions of $\mathfrak{r}_{k}$, $\mathfrak{f}_{k}$ and $-\Sigma i_{k}$ and their relationship with the map $\tilde{\delta}_{k}$. We denote by $p_{k}$ the fibrewise collapse map $G_{k}(V_{0})^{\Hom(T,V_{1})\oplus s(T^{\bot})}\to \mathcal{I}_{k}$ and also choose to label the top collapse map $\tilde{X}_{k}\to \tilde{X}_{k}/Y_{k}$ by $c_{k}$:
$$\xymatrix{Y_{k}\ar@{ >->}[r]\ar[d]_{1}&\tilde{X}_{k}\ar[d]_{\tilde{\pi}_{k}}\ar@{->>}[r]_{c_{k}}&\frac{\tilde{X}_{k}}{Y_{k}}\ar[d]&G_{k}(V_{0})^{\Hom(T,V_{1})\oplus s(T^{\bot})}\ar@/_2.5pc/@{.>}[ll]_{\tilde{\phi}_{k}?}\ar[l]_{\mathfrak{r}_{k}\circ p_{k}\phantom{xxxxxxxx}}\\
Y_{k}\ar@{ >->}[r]&\tilde{X}_{k-1}\ar@/_3.7pc/[rru]|\bigcirc_{\phantom{xx}\tilde{\delta_{k}}}\ar@{->>}[r]&\frac{\tilde{X}_{k-1}}{Y_{k}}\ar[ur]|\bigcirc _{\phantom{x}-\Sigma i_{k}\circ \mathfrak{f}_{k}}
}$$
By the theory of $\S$\ref{TotalCofibres} the map $\mathfrak{r}_{k}\circ p_{k}$ lifts to some map $j$ and $j$, $\tilde{\pi}_{k}$ and $\tilde{\delta}_{k}$ with $\tilde{X}_{k}$, $\tilde{X}_{k-1}$ and $G_{k}(V_{0})^{\Hom(T,V_{1})\oplus s(T^{\bot})}$ form a cofibre sequence. We wish to prove that $j$ is the map $\tilde{\phi}_{k}$ from \ref{tildephik}. We first check that $\tilde{\phi}_{k}$ is a reasonable choice of map.

\begin{prop}\label{phikworks} Recall $\tilde{Z}_{k}$ from Corollary \ref{factoringinsTperp}. Let $\tilde{\phi}_{k}'$ be the map below:
$$\tilde{\phi}_{k}':\tilde{Z}_{k}\to \tilde{X}_{k}'$$
$$(W,\gamma,\psi)\mapsto (\psi|_{W^{\bot}}\oplus (\rho(\gamma)+e_{top}(\psi))|_{W},-\sigma(\gamma)).$$
Then $\tilde{\phi}_{k}'$ is a continuous proper $G$-map and hence the map $\tilde{\phi}_{k}:=(\tilde{\phi}_{k}')_{\infty}$ exists and is well-defined.
\end{prop}
\begin{proof} Firstly equip $\tilde{Z}_{k}$ with the quotient topology and equip $\tilde{X}_{k}'$ with the subspace topology. Then demonstrating continuity in $\tilde{\phi}_{k}'$ is equivalent to demonstrating continuity in the map below:
$$\mathcal{L}(\Complex^{k}\oplus \Complex^{d_{0}-k},V_{0})\times\Hom(\Complex^{k},V_{1})\times s(\Complex^{d_{0}-k})\to \svo\times \aich$$
$$((\zeta,\eta),\gamma_{0},\psi_{0})\mapsto((\eta\circ\psi_{0}\circ\eta^{\dag}\oplus (\rho(\gamma_{0}\circ\zeta^{\dag}) +e_{top}(\eta\circ\psi_{0}\circ\eta^{\dag}))),\gamma_{0}\circ\zeta^{\dag}).$$

For standard functional calculus reasons this is continuous. Now let $proj.:\tilde{X}_{k}'\to \svo$ be the projection to the first factor and let $C$ be a compact subset of $\svo$. We wish to demonstrate that $(proj.\circ\tilde{\phi}_{k}')^{-1}\{C\}$ is compact - this will demonstrate that $\tilde{\phi}_{k}'$ is proper by Lemma \ref{propercompositionisproper}. The compact subsets of $\tilde{Z}_{k}$ are known as the compact subsets of $G_{k}(V_{0})\times\aich\times\svo$ are known - by Tychonoff's Theorem and \ref{extendedheineborel} they are the closed subsets bounded in the two vector space coordinates. The space $\tilde{Z}_{k}$ is a subspace of this, hence we observe that under the vector bundle characterization of $\tilde{Z}_{k}$ compact subsets are closed and fibrewise bounded.

Now let $(W,\gamma,\psi)\in (proj.\circ\tilde{\phi}_{k}')^{-1}\{C\}$. It will be enough to show that there is a bound on the norms of $\gamma$ and $\psi$ - the inverse image of $C$ is clearly closed. By the compactness of $C$ we know that there is a positive real number $R$ such that $\gamma$ and $\psi$ satisfy the following conditions:
\begin{itemize}
\item $\|\psi\|\leqslant R$.
\item $\|\rho(\gamma)+e_{top}(\psi)\|\leqslant R$.
\end{itemize}
We immediately get the bound on $\|\psi\|$. For the other bound, let $e$ be an eigenvalue of $\rho(\gamma)$, it is sufficient to show that there is a positive upper bound on $e$. From the second condition we have that $e\leqslant R-e_{top}(\psi)$. Now from the first condition $-R\leqslant e_{top}(\psi)\leqslant R$ and hence $-R\leqslant -e_{top}(\psi)\leqslant R$. Thus $e\leqslant R-e_{top}(\psi)\leqslant 2R$, giving an upper bound on $e$ as required. This implies properness of $\tilde{\phi}_{k}'$. Moreover, observing that $\tilde{\phi}_{k}'$ is a $G$-map is standard. The claim then follows.
\end{proof}

Thus $\tilde{\phi}_{k}$ is at least well-defined. To show that it is our lift we attempt to further extend the functional calculus to maps which take the below form:
$$G_{k}(V_{0})^{\Hom(T,V_{1})\oplus s(T^{\bot})}\to \tilde{X}_{k}.$$

This will then let us, as earlier, classify maps that arise from the functional calculus up to homotopy. We first fix $k$ to be strictly less than $d_{0}$ - the $d_{0}=k$ case is trivial. The below statement will lead to a definition of facial in the case we want, the proof is standard:

\begin{lem}\label{thirdfacialhomeomorphism} There is a homeomorphism $f:D(d_{0}-k)\wedge D_{+}(k)\to D(d_{0})$ given as follows:
$$(s,t)\mapsto (s,s_{top}+t).$$
\end{lem}

\begin{defn}\label{extendingfacialdefinition} We say a map $g:D(d_{0}-k)\wedge D_{+}(k)\to D(d_{0})$ is facial if the composite $D(d_{0})\overset{f^{-1}}{\to}D(d_{0}-k)\wedge D_{+}(k)\overset{g}{\to} D(d_{0})$ is facial.
\end{defn}

In particular, the map $f$ is facial under this definition. We now set up machinery to extend the functional calculus ideas from $\S$\ref{Variations}. In a similar fashion to the constructions \ref{fnalvarA} and \ref{fnalvarB} we will construct maps $p$, $q$ and $\mathfrak{C}_{g}$ such that the below diagram commutes:
$$\xymatrix{\mathcal{L}(\Complex^{d_{0}},V_{0})_{\infty}\wedge\mathcal{L}(\Complex^{k},V_{1})_{\infty}\wedge D(d_{0}-k)\wedge D_{+}(k)\ar[d]_{1\wedge1\wedge g}\ar[r]^{\phantom{xxxxxxxxxxxx}p}&G_{k}(V_{0})^{\Hom(T,V_{1})\oplus s(T^{\bot})}\ar[d]^{\mathfrak{C}_{g}}\\
\mathcal{L}(\Complex^{d_{0}},V_{0})_{\infty}\wedge\mathcal{L}(\Complex^{k},V_{1})_{\infty}\wedge D(d_{0})\ar[r]_{\phantom{xxxxxxxxxxxx}q}&\tilde{X}_{k}}$$

We first consider the maps $p$ and $q$:

\begin{defn}\label{themapp} Let $i:\Complex^{k}\to \Complex ^{d_{0}}$ be the choice of inclusion sending $\Complex^{k}$ to the last $k$ copies of $\Complex$ in $\Complex^{d_{0}}$. We recall $\tilde{Z}_{k}$ from \ref{factoringinsTperp} and $\Delta$ from \ref{themapnu} and define the map $p'$ as follows:
$$p':\mathcal{L}(\Complex^{d_{0}},V_{0})\times\mathcal{L}(\Complex^{k},V_{1})\times  D'(d_{0}-k)\times D'_{+}(k)\to \tilde{Z}_{k}$$
$$(\lambda,\mu,s,t)\mapsto (\lambda(i(\Complex^{k})),-\mu\circ \Delta(t)\lambda^{-1}|_{\lambda(i(\Complex^{k}))},\lambda\Delta(s)\lambda^{-1}|_{\lambda(i(\Complex^{k}))^{\bot}}).$$
\end{defn}

\begin{lem}\label{pworks} The map $p'$ is a well-defined proper surjection and thus the map $p:=(p')_{\infty}$ is well-defined.
\end{lem}
\begin{proof} The only real issue with this claim is properness, the map is pretty easily observed to be well-defined and continuous and moreover surjectivity can be shown via a standard linear algebra argument similar to the surjectivity claims in \ref{fnalvarA}. To prove that $p'$ is proper, consider the map below:
$$\eta'':\tilde{Z}_{k}\to D(d_{0}-k)\times D'_{+}(k)$$
$$(W,\gamma,\psi)\mapsto (e_{0}(\psi),...,e_{d_{0}-k-1}(\psi)), (e_{0}(\rho(\gamma)),...,e_{k-1}(\rho(\gamma))).$$
Then we claim that $\eta''\circ p'$ is proper. One can observe that this composition is just the projection to $D'(d_{0}-k)\times D_{+}'(k)$. Let $U$ be a compact subset of $D'(d_{0}-k)\times D_{+}'(k)$, then the inverse image of $C$ under $\eta''\circ p'$ is $\mathcal{L}(\Complex^{d_{0}},V_{0})\times\mathcal{L}(\Complex^{k},V_{1})\times U$. This is compact, thus the composition is proper and hence by Lemma \ref{propercompositionisproper} $p'$ is proper.
\end{proof}

\begin{defn}\label{themapq} Define the map $q'$ as follows:
$$q':\mathcal{L}(\Complex^{d_{0}},V_{0})\times\mathcal{L}(\Complex^{k},V_{1})\times D'(d_{0})\to \tilde{X}'_{k}$$
$$(\lambda,\mu,t')\mapsto (\lambda\Delta(t')\lambda^{-1},-\mu\circ \lambda^{-1}|_{P_{k}(\lambda\Delta(t')\lambda^{-1})}).$$
\end{defn}

\begin{lem}\label{qworks} The map $q'$ is a well-defined proper surjection and thus the map $q:=(q')_{\infty}$ is well-defined.
\end{lem}
\begin{proof} Again, the only real issue is demonstrating that $q'$ is proper. To prove this we first take $\eta':\tilde{X}_{k}'\to D'(d_{0})$ to be the first factor eigenvalue map - we then claim that the composition $\eta'\circ q'$ is proper. This, however, follows from the observation that $\eta'\circ q'$ is just the projection onto the factor $D'(d_{0})$, thus the inverse image of a compact set $U$ is $\mathcal{L}(\Complex^{d_{0}},V_{0})\times\mathcal{L}(\Complex^{k},V_{1})\times U$ which is compact. Thus
$\eta'\circ q'$ is proper and hence so is $q'$ by Lemma \ref{propercompositionisproper}.
\end{proof}

We are now ready to construct another functional calculus variation:

\begin{prop}\label{fnalvarC} There is a unique continuous map $\mathfrak{C}_{g}$ making the below diagram commute:
$$\xymatrix{\mathcal{L}(\Complex^{d_{0}},V_{0})_{\infty}\wedge\mathcal{L}(\Complex^{k},V_{1})_{\infty}\wedge D(d_{0}-k)\wedge D_{+}(k)\ar[d]_{1\wedge1\wedge g}\ar[r]^{\phantom{xxxxxxxxxxxx}p}&G_{k}(V_{0})^{\Hom(T,V_{1})\oplus s(T^{\bot})}\ar[d]^{\mathfrak{C}_{g}}\\
\mathcal{L}(\Complex^{d_{0}},V_{0})_{\infty}\wedge\mathcal{L}(\Complex^{k},V_{1})_{\infty}\wedge D(d_{0})\ar[r]_{\phantom{xxxxxxxxxxxx}q}&\tilde{X}_{k}}$$

Moreover:
\begin{itemize}
\item The below map is continuous:
$$\mathfrak{C}:\FMap(D(d_{0}-k)\wedge D_{+}(k))\to \Map(G_{k}(V_{0})^{\Hom(T,V_{1})\oplus s(T^{\bot})}, \tilde{X}_{k})$$
$$g\mapsto \mathfrak{C}_{g}.$$
\item There is a characterization of $\mathfrak{C}_{g}(W,\gamma,\psi)$ as follows. Let $v_{0},...,v_{d_{0}-k-1}$ be an orthonormal basis for $W^{\bot}$ of eigenvectors of $\psi$ with eigenvalues $e_{0},...,e_{d_{0}-k-1}$ and let $v_{d_{0}-k},...,v_{d_{0}-1}$ be an orthonormal basis of $W$ of eigenvectors of $\gamma^{\dag}\gamma$ with eigenvalues
$e_{d_{0}-k}^{2},...,e_{d_{0}-1}^{2}$. Moreover, let $m_{i}$ be the vectors of $V_{1}$ given by $\gamma(v_{i})=e_{i}m_{i}$. Then if $g(e)=s$, then $\mathfrak{C}_{g}(W,\gamma,\psi)$ maps to $(\alpha,\theta)$ where $\alpha$ is the selfadjoint transformation of $V_{0}$ with eigenvectors $v_{i}$ and eigenvalues $s_{i}$ and $\theta$ is the isometry on the top eigenspaces of $\alpha$ given by $\theta(v_{i})=-m_{i}$. Further, $\theta$ can additionally be characterized as $\theta= -\sigma(\gamma)$ where here $\sigma$ is as in \ref{themapsigma}.
\end{itemize}
\end{prop}
\begin{proof} We firstly want to check that a map $\mathfrak{C}_{g}$ of this form is reasonable. The potential issue is that neither $p$ nor $q$ are injective - thus we need to check whether if $p(\lambda,\mu,s,t)=p(\lambda',\mu',s',t')=(W,\gamma,\psi)$ then $q(\lambda,\mu,g(s,t))=q(\lambda',\mu',g(s',t'))$. Firstly set $\zeta:=\lambda^{-1}\lambda':\Complex^{d_{0}}\to \Complex^{d_{0}}$. Then as $\lambda(i(\Complex^{k}))=\lambda'(i(\Complex^{k}))$ we have $\zeta(i(\Complex^{k}))=i(\Complex^{k})$ and thus $\zeta$ splits into $\zeta_{0}:\Complex^{k}\to \Complex^{k}$ and $\zeta_{1}:\Complex^{d_{0}-k}\to \Complex^{d_{0}-k}$, we suppress the notation of $i$ from this point though we use the function implicitly. Now as $-\mu\circ
\Delta(t)\lambda^{-1}|_{W}=-\mu'\circ\Delta(t')\lambda'^{-1}|_{W}=\gamma$ we have that
$\gamma^{\dag}\gamma=\lambda\Delta(t^{2})\lambda^{-1}|_{W}=\lambda'\Delta((t')^{2})\lambda'^{-1}|_{W}$ and hence
$\rho(\gamma)=\lambda\Delta(t)\lambda^{-1}|_{W}=\lambda'\Delta(t')\lambda'^{-1}|_{W}$. This demonstrates that
$\zeta_{0}^{-1}\Delta(t)\zeta_{0}=\Delta(t')$, hence $t=t'$ and $\zeta_{0}$ is in the centralizer of $\Delta(t)$.

Similarly $\psi=\lambda\Delta(s)\lambda^{-1}|_{W^{\bot}}=\lambda'\Delta(s')\lambda'^{-1}|_{W^{\bot}}$ implies that $\zeta_{1}^{-1}\Delta(s)\zeta_{1}=\Delta (s')$ and hence that $s=s'$ and $\zeta_{1}$ is in the centralizer of $\Delta(s)$. Now, let $\Delta(s\oplus t)$ denote the diagonal matrix with the entries $s$ and then $t$. The above demonstrates that $\zeta$ is in the centralizer of $\Delta(s\oplus t)$. Now, in order to prove that $q(\lambda,\mu,g(s,t))=q(\lambda',\mu',g(s',t'))$ on the first factor we need to show that $\zeta$ is in the centralizer of $\Delta(g(s,t))$. We note, however, that the centralizer only depends on repeating values down the matrix, and as $g$ is facial we note that this implies that the centralizer of $\Delta(s\oplus t)$ is a subgroup of the centralizer of $\Delta(g(s,t))$. Hence $\zeta$ is in the centralizer of $\Delta(g(s,t))$ and equality on the first factor follows.

For equality on the second factor it is clear that this will follow if $\mu^{-1}\mu'=\zeta_{0}$. This, however, follows if $\Delta(t)$ is invertible from $\mu^{-1}\mu'\Delta(t)=\Delta(t)\zeta_{0}=\zeta_{0}\Delta(t)$, which follows from looking at the two identifications of $\gamma$. If $\Delta(t)$ is not invertible, however, then we're still okay here.
We want to show that $-\mu\circ \lambda^{-1}|_{P_{k}(\lambda\Delta(g(s,t))\lambda^{-1})}=-\mu'\circ \lambda'^{-1}|_{P_{k}(\lambda'\Delta(g(s',t'))\lambda'^{-1})}$ so can concern ourselves with just showing that $\mu^{-1}\mu'$ agrees with $\zeta_{0}$ on the restriction to the top $n$ copies of $\Complex$ in $\Complex^{k}$ for some $n$ - we determine $n$ by taking the maximal value such that $\Delta(t)|_{\Complex^{n}}$ is invertible. Thus equality in the second factor follows for the same reasons as above. Hence having a map $\mathfrak{C}_{g}$ making the square commute is reasonable; moreover, it is clear to see that the choice for $\mathfrak{C}_{g}$ as detailed above is one that would work to make the diagram commute.

Uniqueness of $\mathfrak{C}_{g}$ follows from the surjectivity of $p$. We now check the continuity of $\mathfrak{C}_{g}$ by showing that $\mathfrak{C}_{g}\circ p$ is continuous - this then implies continuity in $\mathfrak{C}_{g}$ as $p$ is a
quotient. We can show that $\mathfrak{C}_{g}\circ p$ is continuous, however, by noting that it is equal to the continuous map $(1\wedge 1\wedge g)\circ q$.

Finally we wish to show that the below map is continuous, i.e. a morphism in $CGWH$:
$$\mathfrak{C}:\FMap (D(d_{0}-k)\wedge D_{+}(k))\to \Map(G_{k}(V_{0})^{\Hom(T,V_{1})\oplus s(T^{\bot})}, \tilde{X}_{k}).$$

We have a standard adjunction:
$$\xymatrix@R=0.0001cm{\Hom(\FMap (D(d_{0}-k)\wedge D_{+}(k)),\Map(G_{k}(V_{0})^{\Hom(T,V_{1})\oplus s(T^{\bot})},
\tilde{X}_{k}))\\
\cong \\
\Hom(\FMap (D(d_{0}-k)\wedge D_{+}(k))\wedge
G_{k}(V_{0})^{\Hom(T,V_{1})\oplus s(T^{\bot})},\tilde{X}_{k}) }$$
Thus if we show that $\mathfrak{C}^{\#}$ from $\FMap (D(d_{0}-k)\wedge D_{+}(k))\wedge G_{k}(V_{0})^{\Hom(T,V_{1})\oplus s(T^{\bot})}$ to $\tilde{X}_{k}$ given by $\mathfrak{C}^{\#}(g,W,\gamma,\psi):=\mathfrak{C}_{g}(W,\gamma,\psi)$ is continuous, then continuity of $\mathfrak{C}$ follows.

Take the shorthand $\mathcal{L}:=\mathcal{L}(\Complex^{d_{0}},V_{0})_{\infty}\wedge\mathcal{L}(\Complex^{k},V_{1})_{\infty}$. Let $\text{eval.}$ be the map below:
$$\FMap (D(d_{0}-k)\wedge D_{+}(k))\wedge\mathcal{L}\wedge D(d_{0}-k)\wedge D_{+}(k)\overset{\text{eval.}}{\to}\mathcal{L}\wedge D(d_{0})$$
$$(g,\lambda,\mu,s,t)\mapsto(\lambda,\mu,g(s,t)).$$

We have the following commutative diagram:
$$\xymatrix{\FMap (D(d_{0}-k)\wedge D_{+}(k))\wedge\mathcal{L}\wedge D(d_{0}-k)\wedge D_{+}(k)\ar[r]^{\phantom{xxxxxxxxxxxxxxxxxxx}\text{eval.}}\ar[d]_{1\wedge p}&\mathcal{L}\wedge D(d_{0})\ar[d]^{q}\\
\FMap (D(d_{0}-k)\wedge
D_{+}(k))\wedge G_{k}(V_{0})^{\Hom(T,V_{1})\oplus s(T^{\bot})}\ar[r]_{\phantom{xxxxxxxxxxxxxxxxxxx}\mathfrak{C}^{\#}}&\tilde{X}_{k}}$$

Observe from the diagram that $q\circ\text{eval.}=\mathfrak{C}^{\#}\circ (1\wedge p)$ is continuous and hence so is $\mathfrak{C}^{\#}$ as $(1\wedge p)$ is a quotient map. Continuity of $\mathfrak{C}$ follows.
\end{proof}

We can immediately state the below lemma, the proof is simple to observe:

\begin{lem}\label{tildephikisfnalcalc} $\tilde{\phi}_{k}=\mathfrak{C}_{f}$ for the facial map $f$ defined in \ref{thirdfacialhomeomorphism}.
\end{lem}

Further, we can extend our functional calculus to the below result. The proof is clear from the proof of \ref{fnalvarC}:

\begin{cor}\label{extendingCtoD} Let $g':D(d_{0}-k)\wedge D_{+}(k)/D_{0}(k)\to D(d_{0})/F_{d_{0}-k-1}(D(d_{0}))$ be facial. Recalling $\mathcal{I}_{k}'$ from \ref{thetaEalphaextended} we can restrict $p'$ and $q'$ to the below maps:
$$p':\mathcal{L}(\Complex^{d_{0}},V_{0})\times\mathcal{L}(\Complex^{k},V_{1})\times D'(d_{0}-k)\times D'_{+}(k)\backslash D_{0}'(k)\to \mathcal{I}_{k}'$$
$$q':\mathcal{L}(\Complex^{d_{0}},V_{0})\times\mathcal{L}(\Complex^{k},V_{1})\times (D'(d_{0})\backslash F_{d_{0}-k-1}(D'(d_{0})))\to \tilde{X}'_{k}\backslash Y'_{k}.$$

There is a unique continuous map $\mathfrak{D}_{g'}$ making the below diagram commute:
$$\xymatrix{\mathcal{L}(\Complex^{d_{0}},V_{0})_{\infty}\wedge\mathcal{L}(\Complex^{k},V_{1})_{\infty}\wedge D(d_{0}-k)\wedge\frac{D_{+}(k)}{D_{0}(k)} \ar[d]_{1\wedge1\wedge g'}\ar[r]^{\phantom{xxxxxxxxxxxxxxxxxxxx}p}&\mathcal{I}_{k}\ar[d]^{\mathfrak{D}_{g'}}\\
\mathcal{L}(\Complex^{d_{0}},V_{0})_{\infty}\wedge\mathcal{L}(\Complex^{k},V_{1})_{\infty}\wedge
\frac{D(d_{0})}{F_{d_{0}-k-1}(D(d_{0}))}\ar[r]_{\phantom{xxxxxxxxxxxxxxxxxxxx}q}&\frac{\tilde{X}_{k}}{Y_{k}}}$$
\end{cor}

We also have the following two lemmas, the proofs are yet again easy to observe:

\begin{lem}\label{collapseinfnalcalc} Let $h:D(d_{0})\to D(d_{0})/F_{d_{0}-k-1}(D(d_{0}))$ be the standard collapse map. Then the below diagram commutes:
$$\xymatrix{\mathcal{L}(\Complex^{d_{0}},V_{0})_{\infty}\wedge\mathcal{L}(\Complex^{k},V_{1})_{\infty}\wedge D(d_{0})\ar[r]^{\phantom{xxxxxxxxxxxx}q}\ar[d]_{1\wedge h}&\tilde{X}_{k}\ar[d]^{c_{k}}\\
\mathcal{L}(\Complex^{d_{0}},V_{0})_{\infty}\wedge\mathcal{L}(\Complex^{k},V_{1})_{\infty}\wedge \frac{D(d_{0})}{F_{d_{0}-k-1}(D(d_{0}))}\ar[r]_{\phantom{xxxxxxxxxxxxxxxxx}q}&\frac{\tilde{X}_{k}}{Y_{k}}
}$$
\end{lem}

\begin{lem}\label{theothercollapseinfnalcalc} Let $h':D(d_{0}-k)\wedge D_{+}(k)\to D(d_{0}-k) \wedge (D_{+}(k)/D_{0}(k))$ be the standard collapse map. Then the below diagram commutes:
$$\xymatrix{\mathcal{L}(\Complex^{d_{0}},V_{0})_{\infty}\wedge\mathcal{L}(\Complex^{k},V_{1})_{\infty}\wedge D(d_{0}-k)\wedge D_{+}(k)
\ar[d]_{1\wedge 1\wedge h'}\ar[r]^{\phantom{xxxxxxxxxxxx}p}&G_{k}(V_{0})^{\Hom(T,V_{1})\oplus s(T^{\bot})}\ar[d]^{p_{k}}\\
\mathcal{L}(\Complex^{d_{0}},V_{0})_{\infty}\wedge\mathcal{L}(\Complex^{k},V_{1})_{\infty}\wedge
D(d_{0}-k)\wedge
\frac{D_{+}(k)}{D_{0}(k)}\ar[r]_{\phantom{xxxxxxxxxxxx}p}&\mathcal{I}_{k}}$$
\end{lem}

We need one more piece of technology before we can phrase our work in functional calculus, the below lemma. Again, the
proof is clear:

\begin{lem}\label{comparingCandD} Let $g$ and $g'$ be such that $h\circ g=g'\circ h'$. Then $\mathfrak{D}_{g'}\circ p_{k}=c_{k}\circ \mathfrak{C}_{g}$.
\end{lem}

This allows us to switch to working in the functional calculus if the map $\mathfrak{r}_{k}$ comes from $\mathfrak{D}_{g'}$ for some $g'$.

\begin{lem}\label{RcircKisfnalcalc} Let $g':D(d_{0}-k)\wedge D_{+}(k)/D_{0}(k)\to D(d_{0})/F_{d_{0}-k-1}(D(d_{0}))$ be given by the homeomorphism below:
$$(s,t)\mapsto (\log(t_{0})-\Exp(-s),\log(t)).$$
Then $\mathfrak{D}_{g'}= \mathfrak{r}_{k}$.
\end{lem}
\begin{proof} This follows immediately from the uniqueness of $\mathfrak{D}_{g'}$ and the below commutative diagram:
$$\xymatrix{\mathcal{L}(\Complex^{d_{0}},V_{0})_{\infty}\wedge\mathcal{L}(\Complex^{k},V_{1})_{\infty}\wedge D(d_{0}-k)\wedge \frac{ D_{+}(k)}{D_{0}(k)}\ar[d]_{1\wedge1\wedge g'}\ar[r]^{\phantom{xxxxxxxxxxxxxxxxxxx}p}&\mathcal{I}_{k}\ar[d]^{\mathfrak{r}_{k}}\\
\mathcal{L}(\Complex^{d_{0}},V_{0})_{\infty}\wedge\mathcal{L}(\Complex^{k},V_{1})_{\infty}\wedge
\frac{D(d_{0})}{F_{d_{0}-k-1}(D(d_{0}))}\ar[r]_{\phantom{xxxxxxxxxxxxxxxxxx}q}&\frac{\tilde{X}_{k}}{Y_{k}}}$$
\end{proof}

This thus reduces the problem. We have the below diagram, recalling the maps from the discussion above:
$$\xymatrix{&D(d_{0})\ar[d]^{h}\\
D(d_{0}-k)\wedge D_{+}(k)\ar[r]_{g'\circ h'}\ar[ur]^{f}&
\frac{D(d_{0})}{F_{d_{0}-k-1}(D(d_{0}))}}
$$
If we can show this diagram commutes up to facial homotopy then by the fact that facial homotopies pass through the functional calculus and by the machinery set up above we will have the map $\mathfrak{D}_{g'}\circ p_{k}\cong \mathfrak{r}_{k}\circ p_{k}$ lifting to the map $\mathfrak{C}_{f}\cong\tilde{\phi}_{k}$; moreover it is clear that the restrictions and homotopies arising from this method will be equivariant. We note that there is a map $\bar{f}:D(d_{0}-k)\wedge D_{+}(k)/D_{0}(k)\to D(d_{0})/F_{d_{0}-k-1}(D(d_{0}))$ given by the same formulation as $f$ which makes the below diagram strictly commute:
$$\xymatrix{D(d_{0}-k)\wedge D_{+}(k)\ar[d]_{h'}\ar[r]^{f}&D(d_{0})\ar[d]^{h}\\
D(d_{0}-k)\wedge
\frac{D_{+}(k)}{D_{0}(k)}\ar[r]_{\bar{f}}&\frac{D(d_{0})}{F_{d_{0}-k-1}(D(d_{0}))}}$$
Thus it follows that if $g'$ and $\bar{f}$ are homotopic through facial maps, then $\tilde{\phi}_{k}$ is good for our candidate lift.

We wish to study the facial homotopy type of the below space of maps:
$$\FMap (D(d_{0}-k)\wedge D_{+}(k)/D_{0}(k),D(d_{0})/F_{d_{0}-k-1}(D(d_{0}))).$$

We firstly recall that $f$ from \ref{thirdfacialhomeomorphism} is a facial homeomorphism which moreover factors to $\bar{f}$ above, which is also a facial homeomorphism. Hence we can use $f$ to rephrase the problem yet again, to looking at the facial homotopy type of self-maps of $D(d_{0}-k)\wedge D_{+}(k)/D_{0}(k)$. Next, recall the first face-preserving homeomorphism from Lemma \ref{wereactuallyworkingwithd}. Applying this allows us to actually consider the facial homotopy type of self-maps of $D(d_{0}-k)\wedge D(k)$. This, however, we can classify using the theory of the facial homotopy type of self-maps of $D(d)$ as detailed through $\S$\ref{TheHomotopyTypeofCertainMapsinFunctionalCalculus}. The proof of the below claim follows from this previous work and by noting that facial maps $a:D(d_{0}-k)\wedge D(k)\to D(d_{0}-k)\wedge D(k)$ will by necessity of the facial structure come from facial maps $a_{0}:D(d_{0}-k)\to D(d_{0}-k)$ and $a_{1}:D(k)\to D(k)$:

\begin{prop}\label{facialhomotopytypeofC} Let $a,b:D(d_{0}-k)\wedge D(k)\to D(d_{0}-k)\wedge D(k)$ be facial. Then there is a copy of $S^{2}$ embedded in $D(d_{0}-k)\wedge D(k)$ arising from the copies of $S^{1}$ that lie in $D(d_{0}-k)$ and $D(k)$ as the intersections of all faces. Assume that the self-maps of $S^{2}$ coming from $a$ and $b$ have the same degree. Then $a\simeq b$ through facial maps.
\end{prop}

We thus have a homotopical classification in this case. Further, we can pull this classification back through the homeomorphisms to $\FMap (D(d_{0}-k)\wedge D_{+}(k)/D_{0}(k),D(d_{0})/F_{d_{0}-k-1}(D(d_{0})))$ and thus we
have a criterion for $g'$ and $\bar{f}$ to be homotopic - equality in the degree of each induced map $S^{2}\to S^{2}$ gleaned by considering what happens to the intersection of faces, i.e. points of the form $((s,...,s), (t,...,t))$.

\begin{prop}\label{matchingdegreesofC} The maps $g'$ and $\bar{f}$ are both of degree $1$.
\end{prop}
\begin{proof} First consider $g'$, we have the map $g'':S^{2}\to S^{2}$ given below:
$$g''(s,t):=(t-e^{-s},-s).$$
We note that $g''$ makes the below diagram strictly commute:
$$\xymatrix{S^{2}\ar[d]_{g''}\ar@{ >->}[r]&\frac{D_{+}(k)}{D_{0}(k)}\wedge D(d_{0}-k)\ar[d]^{g'}\\
S^{2}\ar@{ >->}[r]&\frac{D(d_{0})}{F_{d_{0}-k-1}(D(d_{0}))}}$$
Now, $g''$ is injective almost everywhere, so by a local degree argument, for example as in $2.2$ of \cite{Hatcher}, it can only be degree $1$ or $-1$. To check which it is, we look at the underlying map $\Real\times\Real\to \Real\times\Real$ given by $(s,t)\mapsto (t-e^{-s},-s)$. This map's derivative matrix is given as follows:
$$\left(\begin{array}{cc}e^{-s}& 1\\
-1 & 0\end{array}\right).$$

This has determinant $1$ and thus the underlying map is orientation preserving. This demonstrates that $g''$ has degree $1$ and thus that $g'$ is of degree $1$.

For $\bar{f}$, let $\bar{f}':S^{2}\to S^{2}$ be the identity map. The map $\bar{f}'$ makes the below diagram strictly commute:
$$\xymatrix{S^{2}\ar[d]_{\bar{f}'}\ar@{ >->}[r]&\frac{D_{+}(k)}{D_{0}(k)}\wedge D(d_{0}-k)\ar[d]^{\bar{f}}\\
S^{2}\ar@{ >->}[r]&\frac{D(d_{0})}{F_{d_{0}-k-1}(D(d_{0}))}}$$

As $\bar{f}'$ is the identity, it follows that it patently has degree $1$. Thus $\bar{f}$ has degree $1$ and the claim follows.
\end{proof}

This result makes $g'$ and $\bar{f}$ homotopic through facial maps via Proposition \ref{facialhomotopytypeofC}. Thus this proves the following proposition:

\begin{prop}\label{explicitlift} We have an explicit lift of $\mathfrak{r}_{k}\circ p_{k}$ given by $\mathfrak{C}_{f}$ for the map $f:D(d_{0}-k)\wedge D_{+}(k)\to D(d_{0})$ defined in \ref{thirdfacialhomeomorphism} which sends $(s,t)$ to
$(s,s_{top}+t)$. Moreover, this lift is actually the map $\tilde{\phi}_{k}$ defined in \ref{tildephik}. Hence the map $\phi_{k}$ completes the stable cofibre sequence as claimed.
\end{prop}

\section{Explicit Null-Homotopies}\label{ExplicitNullHomotopies}

A combination of the results in Chapter \ref{ch:ch5} and in the previous section allows us to declare that the below diagram is a cofibre sequence for the stated maps $\tilde{\pi}_{k}$, $\tilde{\phi}_{k}$ and $\tilde{\delta}_{k}$:
$$\xymatrix{\tilde{X}_{k}\ar[d]_{\tilde{\pi}_{d_{0}}}&G_{k}(V_{0})^{\Hom(T,V_{1})\oplus s(T^{\bot})}\ar[l]\ar[l]_{\tilde{\phi}_{k}\phantom{xxxxx}}\\
\tilde{X}_{k-1}\ar[ur]|\bigcirc_{\phantom{x}\tilde{\delta}_{k}}&}
$$

We note, however, that the proof of this is somewhat roundabout - we prove that something else is a cofibre sequence which then implies the result. Thus we take an aside to explicitly demonstrate why certain key properties held by cofibre sequences apply in this case. In particular, cofibre sequences are easily observed to have null composites along the sequence. We choose to spend this section writing down the null-homotopies; we note here that all homotopies and maps we write down in this section will be $G$-maps as standard.

Firstly we consider the top triangle:
$$\xymatrix{\tilde{X}_{d_{0}}\ar[d]_{\tilde{\pi}_{d_{0}}}&S^{\aich}\ar[l]_{\tilde{\phi}_{d_{0}}\phantom{xxx}}\\
\tilde{X}_{d_{0}-1}\ar[ur]|\bigcirc_{\phantom{x}\tilde{\delta}_{d_{0}}}&}
$$
The composite $\Sigma\tilde{\phi}_{d_{0}}\circ\tilde{\delta}_{d_{0}}$ is easily seen to be strictly null; the map composes the collapse to injectives with the inclusion of non-injective maps and this is patently null. We now look at the other two composites, recalling the definitions of $\rho$, $\sigma$, and $\lambda_{k}$ from \ref{themaprho}, \ref{themapsigma} and \ref{lambdak}.

\begin{prop}\label{thefirsttopcomposite} $\tilde{\pi}_{d_{0}}\circ\tilde{\phi}_{d_{0}}$ is null-homotopic.
\end{prop}
\begin{proof} This composition is the below map:
$$\gamma\mapsto
\left\{\begin{array}{ll}(\log(\rho(\gamma)),-\sigma(\gamma)|_{P_{d_{0}-1}(\log(\rho(\gamma)))}
)&\text{if $\gamma$ is
injective}\\
\infty&\text{otherwise.}\end{array}\right.$$
Define the map $h'$ and space $U$ as follows:
$$U:=\{(t,\gamma)\in [0,\infty)\times\aich:t+e_{0}(\gamma)>0\}$$
$$h':U\to \tilde{X}_{d_{0}-1}'$$
$$(t,\gamma)\mapsto(\log(\rho(\gamma)+t),-\sigma(\gamma)|_{P_{d_{0}-1}(\log(\rho(\gamma)))}).$$
We claim this is well-defined, continuous and proper. If this is true then the one-point compactification $h:U_{\infty}\to \tilde{X}_{d_{0}-1}$ will be our required null-homotopy. We only have to check that
$P_{d_{0}-1}(\log(\rho(\gamma)))=P_{d_{0}-1}(\log(\rho(\gamma)+t))$ to make the map well-defined but this is standard as adding $t$ changes the eigenvalues but not the eigenvectors and hence not the eigenspaces. For properness, let $C$ be a compact subset of $\tilde{X}_{d_{0}-1}$ and let $(t,\gamma)\in h'^{-1}\{C\}$. We need a positive bound on $t$, an upper bound on the eigenvalues of $\rho(\gamma)$ and a strictly positive lower bound on $t+e_{0}(\rho(\gamma))$ for this to be proper - everything else is standard to check. By assumption we have $\|\log(\rho(\gamma)+t)\|\leqslant R$ for some positive real number $R$. Let $e$ be an eigenvalue of $\rho(\gamma)$. Then the inequality gives us that $|\log(e+t)|\leqslant R$ and hence that $e^{-R}\leqslant e+t\leqslant e^{R}$. As $e$ and $t$ are positive this gives us that $e\leqslant e^{R}$ and $t\leqslant e^{R}$ and hence we have upper bounds as required. For the lower bound we have
$0<e^{-R}\leqslant e+t$ which is enough. Finally continuity is easy to see when $X_{d_{0}-1}$ is embedded into $\svo\times \aich$.
\end{proof}

\begin{prop}\label{thesecondtopcomposite} $\tilde{\delta}_{d_{0}}\circ\tilde{\pi}_{d_{0}}$ is null-homotopic.
\end{prop}
\begin{proof} Consider the map $h'$ and space $U$ below:
$$U:=\{(t,\alpha,\theta)\in[0,\infty)\times \svo\times\Ell:e_{1}(\alpha)-e_{0}(\alpha)+t>0\}$$
$$h':U\to\Real\times \aich$$
$$(t,\alpha,\theta)\mapsto (e_{0}(\alpha),-\theta\circ(\alpha-e_{0}(\alpha)+t)).$$
That this is well-defined and continuous is trivial. If it is proper, then $h:U_{\infty}\to S^{\Real\oplus\aich}$ is easily seen to generate our null-homotopy when composed with a suspension twist. To show that this is proper let $C$ be a
compact subset of $\Real\times \aich$ and let $(t,\alpha,\theta)\in h'^{-1}\{C\}$. We need bounds on $t$ and the eigenvalues of $\alpha$, this will be enough to imply that the map is proper. By assumption we have
$\|-\theta\circ(\alpha-e_{0}(\alpha)+t)\|\leqslant R$ and $|e_{0}(\alpha)|\leqslant R$ for some positive real number $R$. Let $e$ be an eigenvalue of $\alpha$, then we have that $|e-e_{0}(\alpha)+t|\leqslant R$ by assumption. As $t$ and
$e-e_{0}(\alpha)$ are positive this immediately gives us that $t\leqslant R$ and $e-e_{0}(\alpha)\leqslant R$, giving one of the bounds. For the other, $|e_{0}(\alpha)|\leqslant R$ gives that $e_{0}(\alpha)\leqslant R$, hence $e\leqslant R+e_{0}(\alpha)\leqslant 2R$, giving the final bound. This is enough to show the result.
\end{proof}

This gives null-homotopies in the top of the tower. We now proceed to calculate the null-homotopies for $1\leqslant k<d_{0}$.

\begin{prop}\label{thefirstcomposite} $\tilde{\pi}_{k}\circ\tilde{\phi}_{k}$ is null-homotopic.
\end{prop}
\begin{proof} This composition is the below map:
$$(W,\gamma,\psi)\to \left(\delta:=((\rho(\gamma)+ e_{top}(\psi))|_{W}\oplus\psi|_{W^{\bot}}),-\sigma(\gamma)|_{P_{k-1}(\delta)}\right).$$
Consider the map $h'$ defined as follows, recalling $\tilde{Z}_{k}$ from \ref{factoringinsTperp}:
$$h':[0,\infty)\times \tilde{Z}_{k}\to \tilde{X}_{k-1}$$
$$(t,W,\gamma,\psi)\mapsto \left(\delta_{t}:=((\rho(\gamma)+ e_{top}(\psi)+t)|_{W}\oplus\psi|_{W^{\bot}}),-\sigma(\gamma)|_{P_{k-1}(\delta)}\right).$$
We claim $h'$ is well-defined, continuous and proper. It will hence generate the right null-homotopy on one-point compactifications. We first check that the map is well-defined, i.e. that $P_{k-1}(\delta)=P_{k-1}(\delta_{t})$. This, however, is easy to observe as both spaces will just be $P_{k-1}(\rho(\gamma))$ contained in $W$. We now check continuity. Equipping $\tilde{Z}_{k}$ with the quotient topology and $\tilde{X}_{k-1}$ with the subspace
topology reduces the argument down to demonstrating continuity in the below map:
$$[0,\infty)\times\mathcal{L}(\Complex^{k}\oplus \Complex^{d_{0}-k},V_{0})\times\Hom(\Complex^{k},V_{1})\times s(\Complex^{d_{0}-k})\to \svo\times \aich$$
$$(t,(\zeta,\eta),\gamma_{0},\psi_{0})\mapsto\left(\begin{array}{c}(\rho(\gamma_{0}\circ\zeta^{\dag})
+e_{top}(\eta\circ\psi_{0}\circ\eta^{\dag})+t)\oplus \eta\circ\psi_{0}\circ\eta^{\dag}\\
\sigma(\gamma_{0}\circ\zeta^{\dag})\circ\lambda_{k-1}(\rho(\gamma_{0}\circ\zeta^{\dag}))\end{array}\right).$$
For functional calculus reasons this is continuous. We finally demonstrate that the map is proper. Let $C$ be a compact subset of $\tilde{X}_{k-1}$ and let $(t,W,\gamma,\psi)\in h'^{-1}\{C\}$. Then properness follows if we can demonstrate bounds on $|t|$, $\|\gamma\|$ and $\|\psi\|$. By assumption we have that there exists a positive real number $R$ such that $\|\psi\|\leqslant R$ and $\|\rho(\gamma)+e_{top}(\psi)+t\|\leqslant R$. Let $e$ be an eigenvalue of $\rho(\gamma)$, then the second condition gives that $e+e_{top}(\psi)+t\leqslant R$, i.e. $e+t\leqslant R-e_{top}(\psi)$.
Now $\|\psi\|\leqslant R$ gives that $-R\leqslant e_{top}(\psi)\leqslant R$ and hence $-R\leqslant -e_{top}(\psi)\leqslant R$. Thus as $t$ and $e$ are positive we deduce that $t\leqslant 2R$ and $e\leqslant 2R$. This combined with $\|\psi\|\leqslant R$ demonstrates properness in the map, and the claim follows.
\end{proof}

\begin{prop}\label{thesecondcomposite} $\tilde{\delta}_{k}\circ\tilde{\pi}_{k}$ is null-homotopic.
\end{prop}
\begin{proof} Consider the map $h'$ below:
$$h':[0,\infty)\times \tilde{X}_{k}\backslash Y_{k}\to \Real \times\tilde{Z}_{k}$$
$$(t,\alpha,\theta)\mapsto (e_{d_{0}-k}(\alpha),P_{k}(\alpha),-\theta\circ(\lambda_{k-1}(\alpha)+t)|_{P_{k}(\alpha)},-\log(e_{d_{0}-k}(\alpha)-\alpha)|_{P_{k}(\alpha^{\bot})}).$$
We first check that the map is well-defined. The only potential issue is $-\theta\circ(\lambda_{k-1}(\alpha)+t)$, but this is reasonable as $\theta$ is coming from $\tilde{X}_{k}$. To check continuity, equip $\tilde{Z}_{k}$ with the subspace topology and $\tilde{X}_{k}$ with the quotient topology. Then continuity follows from the standard continuity of the below map, recalling the notation of $s_{k}(V_{0})$ from Corollary \ref{Pkiscont}:
$$[0,\infty)\times s_{k}(V_{0})\times\Ell\to \Real\times G_{k}(V_{0})\times\aich\times\svo $$
$$(t,\alpha,\theta)\mapsto(e_{d_{0}-k}(\alpha),P_{k}(\alpha),-\theta\circ(\alpha-e_{d_{0}-k}(\alpha)+t),-\log(e_{d_{0}-k}(\alpha)-\alpha)).$$
We finally check that the map is proper. Let $C$ be a compact subset of $\Real\times\tilde{Z}_{k}$ and let $(t,\alpha,\theta)\in h'^{-1}\{C\}$. For this map to be proper, we need a bound on $t$ and $\|\alpha\|$ and we need to keep $e_{d_{0}-k}(\alpha)$ and $e_{d_{0}-k-1}(\alpha)$ apart. By assumption there is a positive real number $R$ which bounds $|e_{d_{0}-k}(\alpha)|$, $\|-\theta\circ(\lambda_{k-1}(\alpha)+t)|_{P_{k}(\alpha)}\|$ and $\|-\log(e_{d_{0}-k}(\alpha)-\alpha)|_{P_{k}(\alpha^{\bot})}\|$. The second bound gives us that $e_{d_{0}-1}(\alpha)+t-e_{d_{0}-k}(\alpha)\leqslant R$. This will give a bound on $t$ for similar reasons to previous properness arguments, as well as a bound of $R$ on $e_{d_{0}-1}(\alpha)-e_{d_{0}-k}(\alpha)$. Moreover, the final assumption gives us that $-R\leqslant \log(e_{d_{0}-k}(\alpha)-e_{d_{0}-k-1}(\alpha))\leqslant R$ and hence $e^{-R}\leqslant e_{d_{0}-k}(\alpha)-e_{d_{0}-k-1}(\alpha)$; the right eigenvalues are suitably kept apart. We can repeat a similar process with $e_{0}(\alpha)$ to get $e_{d_{0}-k}(\alpha)-e_{0}(\alpha)\leqslant
e^{R}$, i.e. that $e_{0}(\alpha)\geqslant e_{d_{0}-k}(\alpha)-e^{R}$. Combining this with $e_{d_{0}-1}(\alpha)-e_{d_{0}-k}(\alpha)\leqslant R$ gives $\max(|e_{0}(\alpha)|,|e_{d_{0}-1}(\alpha)|)\leqslant R+e^{R}$. This gives a bound on $\|\alpha\|$. Hence $h'$ is proper. It is easy to see that from here the map $h$ on one-point compactifications can be suitably modified and twisted across the suspension to make the required null-homotopy.
\end{proof}

\begin{prop}\label{thethirdcomposite} $\Sigma\tilde{\phi}_{k}\circ\tilde{\delta}_{k}$ is null-homotopic.
\end{prop}
\begin{proof} First set $j:=d_{0}-k$. We note that the composite resembles the map below, once a suspension twist has been factored in:
$$\tilde{X}_{k-1}\to\frac{\tilde{X}_{k-1}}{Y_{k}}\to \Sigma\tilde{X}_{k}$$
\small $$(\alpha,\theta)\mapsto (e_{j}(\alpha),-\log(e_{j}(\alpha)-\alpha)|_{P_{k}(\alpha)^{\bot}}\oplus(\alpha-e_{j}(\alpha)-\log(e_{j}(\alpha)-e_{j-1}(\alpha)))|_{P_{k}(\alpha)},\theta).$$
\normalsize
Firstly define the selfadjoint endomorphism $\beta_{t}$ as follows: $$\beta_{t}:=-\log(e_{j}(\alpha)-\alpha+t)|_{P_{k}(\alpha)^{\bot}}\oplus(\alpha-e_{j}(\alpha)-\log(e_{j}(\alpha)-e_{j-1}(\alpha)+t))|_{P_{k}(\alpha)}.$$
Now set $U$ and $h'$ to be the following space and map:
$$U:=\{(t,\alpha,\theta)\in[0,\infty)\times\tilde{X}_{k-1}':e_{j}(\alpha)-e_{j-1}(\alpha)+t>0\}$$
$$h':U\to \Real\times \tilde{X}'_{k}$$
$$(t,\alpha,\theta)\mapsto (e_{j}(\alpha),\beta_{t},\theta).$$
We check this map is well-defined, continuous and proper. The map $h$ on one-point compactifications can then be modified to build the required null-homotopy. Firstly, for well-definedness we need to check that $\theta$ is well-defined, i.e. that $P_{k}(\beta_{t})$ is $P_{k-1}(\alpha)$. This is possible to see, however, by first observing that it is true for $\beta_{0}$ via the construction of the composite. Adding $t$ in such a way as we do should not affect eigenvectors, only eigenvalues; however, we have to be careful as intrinsic to this is the `collapsing out $P_{k}(\alpha)$' condition that holds at $t=0$. We're fine, however, by shifting the collapse by $t$ as we do via setting that $e_{j}(\alpha)-e_{j-1}(\alpha)+t>0$. This is enough to demonstrate that $P_{k}(\beta_{t})=P_{k-1}(\alpha)$ as required, and hence the map above is reasonable.

We note here that the map is easily seen to be continuous. Finally, we need to check that the map is proper. Let $C$ be a compact subset of $\Real\times \tilde{X}_{k}'$ and let $(t,\alpha,\theta)\in h'^{-1}\{C\}$. For properness in this case we require bounds on $t$ and $\|\alpha\|$ and a strictly positive lower bound on $e_{j}(\alpha)-e_{j-1}(\alpha)+t$. By assumption we have a positive real number $R$ such that $|e_{j}(\alpha)|\leqslant R$ and $\|\beta_{t}\|\leqslant R$. Firstly note that the second condition gives that if $e$ is a top eigenvalue of $\alpha$ then $e-e_{j}(\alpha)-\log(e_{j}(\alpha)-e_{j-1}(\alpha)+t)\leqslant R$. We know that $e-e_{j}(\alpha)$ is positive, hence $-\log(e_{j}(\alpha)-e_{j-1}(\alpha)+t)\leqslant R$. Multiplying by $-1$ and taking exponentials gives $e_{j}(\alpha)-e_{j-1}(\alpha)+t\geqslant e^{-R}$ giving a strictly positive lower bound on $e_{j}(\alpha)-e_{j-1}(\alpha)+t$ as required.

We also have from this condition that if $e'$ is a low eigenvalue then $-R\leqslant -\log(e_{j}(\alpha)-e'+t)\leqslant R$ and hence $e^{-R}\leqslant e_{j}-e'+t\leqslant e^{R}$. We first use this to note that $t\leqslant e^{R}$ and hence we have an upper bound for $t$. We now seek out a lower bound for $e'$. We have $e_{j}(\alpha)-e'\leqslant e^{R}$ and hence $e'\geqslant e^{R}-e_{j}(\alpha)$. We have a bound $-R\leqslant e_{j}(\alpha)\leqslant R$ by assumption and hence $e'\geqslant e^{R}-R$ for a lower bound.

Finally, we return to looking at $e$ a high eigenvalue. We wish to find an upper bound for $e$, this will then be enough to imply a bound on $\|\alpha\|$. Firstly we note from the above paragraph we have $-R\leqslant -\log(e_{j}(\alpha)-e_{j-1}(\alpha)+t)$. We also note that $-e_{j}(\alpha)\leqslant R$ from the first assumption. Now recall that we can also write down the inequality $e\leqslant R-e_{j}(\alpha)+\log(e_{j}(\alpha)-e_{j-1}(\alpha)+t)$. Combining these give $e\leqslant 3R$ and hence we have a bound on $\|\alpha\|$. This gives us that $h'$ is proper and hence we can build the required null-homotopy.
\end{proof}

\section{The Bottom of the Tower}\label{TheBottomoftheTower}

We have built the below stable $G$-map as part of our construction of the tower, noting here that $s(T)\cong \Real$ for $T$ the tautological line bundle over complex projective space:
$$\delta_{1}:S^{0}\to PV_{0}^{\Real\oplus \Hom(T,V_{1}-V_{0})\oplus s(T)}\cong\Sigma^{2}PV_{0}^{\Hom(T,V_{1}-V_{0})}.$$
We claim further that this map possesses extra geometric structure - it can be built from a Pontryagin-Thom collapse map for the right choice of embedding.

Let $j:PV_{0}\to W$ be an equivariant embedding of $PV_{0}$ into a representation $W$. Details on the Pontryagin-Thom construction can be found amidst much of the literature, we note here that a brief description of it lies in $B.2$ of \cite{Ravenel} - this book also includes a list of other references that cover the theory. Working initially in $G$-spaces we have the following Pontryagin-Thom collapse map; this map is a $G$-map:
$$j^{!}:S^{W}\to PV_{0}^{\nu_{j}}.$$

Here $\nu_{j}$ is the normal bundle $\tau_{W}-\tau_{PV_{0}}$. As $W$ is a vector space we can rewrite $PV_{0}^{\nu_{j}}$ as $\Sigma^{W}PV_{0}^{-\tau_{PV_{0}}}$. Pass to $G$-spectra via $\Sigma^{\infty}$ and smash by $S^{-W}$ to get the below stable $G$-map:
$$\Sigma^{-W}j^{!}:S^{0}\to PV_{0}^{-\tau_{PV_{0}}}.$$

We now take the lemma stated below as fact, the result is discussed in $\S14$ of \cite{CharacteristicClassesMilnorStasheff}, for example.

\begin{lem}\label{thetangentbundleofcomplexprojectivespace} The bundle $\tau_{PV_{0}}$ is the bundle $\Hom(T,T^{\bot})$.
\end{lem}

Using this and the bundle identities used in Lemma \ref{destabilizinglemma} allows us to see that we have a stable $G$-map as follows:
$$\Sigma^{-W}j^{!}:S^{0}\to PV_{0}^{2s(T)-\Hom(T,V_{0})}\cong \Sigma^{2}PV_{0}^{-\Hom(T,V_{0})}.$$

Finally, we compose with the zero-section map $PV_{0}\to PV_{0}^{\Hom(T,V_{1})}$ and add in a twist $t\mapsto (1-t)$ to the first suspension coordinate to get the map $\delta^{j}$ below:
$$\delta^{j}:S^{0}\to \Sigma^{2}PV_{0}^{\Hom(T,V_{1}-V_{0})}.$$

\begin{prop}\label{thebottommapisPT} The map $\delta^{j}$ corresponds to the map $\delta_{1}$ for the choice of embedding $j:PV_{0}\to \svo$ given by $j:L\mapsto 0_{L}\oplus -1_{L^{\bot}}$.
\end{prop}
\begin{proof} First note that this embedding is easily seen to be equivariant. We prove this unstably, first explicitly describing the map $\tilde{\delta}_{1}$. Let $p_{0}:S^{\svo}\to S^{\svo}/\sim$ be the quotient map of the equivalence relation given by $\alpha\sim\alpha'$ if and only if $e_{d_{0}-1}(\alpha)=e_{d_{0}-2}(\alpha)$ and $e_{d_{0}-1}(\alpha')=e_{d_{0}-2}(\alpha')$ - this is a $G$-map and is the equivalent of taking the quotient of $\tilde{X}_{0}$ by $Y_{1}$ as described in Definition \ref{makingthecofibseqeasier} and $\S$\ref{TheOtherTriangles}. We now note that in the case $k=1$ the map $\lambda_{k-1}(\alpha)$ recalled from \ref{lambdak} is the zero map. This allows us to split up $\tilde{\delta}_{1}$ into the following composition. First define $m$ to be the map below:
$$m:S^{\svo}/\sim\to \Sigma PV_{0}^{s(T^{\bot})}$$
$$\alpha\mapsto (e_{d_{0}-1}(\alpha),\Ker(\alpha-e_{d_{0}-1}(\alpha)),-\log(e_{d_{0}-1}(\alpha)-\alpha)|_{\Ker(\alpha-e_{d_{0}-1}(\alpha))^{\bot}}).$$
We then compose with the zero-section:
$$\Sigma PV_{0}^{s(T^{\bot})}\to \Sigma PV_{0}^{\Hom(T,V_{1})\oplus s(T^{\bot})}.$$
Finally, we twist the suspension coordinate with a $t\mapsto (1-t)$ twist to maintain the cofibre sequence structure. Pre-composing with $p_{0}$ gives us a map below, this is $\tilde{\delta}_{1}$:
$$S^{\svo}\to \Sigma PV_{0}^{\Hom(T,V_{1})\oplus s(T^{\bot})}.$$
Thus by noting the description of the build of $j^{!}$ it is clear that the claim follows from the equivalence of the two maps below:
$$j^{!}:S^{\svo}\to PV_{0}^{\svo-\Hom(T,T^{\bot})}\cong\Sigma PV_{0}^{s(T^{\bot})}$$
$$m\circ p_{0}:S^{\svo}\to \Sigma PV_{0}^{s(T^{\bot})}.$$
To show this we first put $\sovo$ to be the space of selfadjoint endomorphisms with zero top eigenvalue, we have a $G$-homeomorphism $S^{\svo}\cong \Sigma(\sovo_{\infty})$ given by the map $\alpha\mapsto (e_{d_{0}-1}(\alpha),\alpha-e_{d_{0}-1}(\alpha))$. We next note that $j$ actually provides an equivariant embedding into $\sovo$. It's clear that from here $j^{!}$ can be desuspended to $\Sigma^{-1}j^{!}:\sovo_{\infty}\to PV_{0}^{s(T^{\bot})}$ by including the top eigenvalue as the suspension coordinate. 

We now note that by choosing an embedding $j$ which outputs a selfadjoint that sends precisely one line to zero it becomes clear that $j^{!}$ will give $\Ker(\alpha-e_{d_{0}-1}(\alpha))$ as the base-value of the Thom space. Moreover is is also clear from this that $j^{!}$ will collapse out precisely the selfadjoints that have top two matching eigenvalues. Thus $j^{!}$ will incorporate the collapse $p_{0}$ as claimed. Finally, it suffices to show that the
bundle factors match up. This is again clear, however, by looking at $\Sigma^{-1} j^{!}$. This has to send a selfadjoint $\beta$ with top eigenvalue $0$ to a general selfadjoint element $\beta_{j}$ on $\Ker(\beta)^{\bot}$; moreover this map must be injective and maintain the local eigenvalue order structure, i.e. if $e_{k}(\beta)\leqslant e_{k}(\gamma)$ then $e_{k}(\beta_{j})\leqslant e_{k}(\gamma_{j})$. A standard degree argument will show that this is fibrewise $G$-homotopic
to $-\log(\beta|_{\Ker(\beta)^{\bot}})$ as required. This is enough to show that $\delta^{j}\simeq \delta_{1}$ as claimed.
\end{proof}
\chapter{The Subrepresentation Case}
\label{ch:ch7}

\section{The Conjecture}\label{TheConjecture}

We now conjecture what we believe to happen to Theorem \ref{themaintheorem} should $V_{0}\leqslant V_{1}$. For a more
precise framework, let $V_{1}=V_{0}\oplus V_{2}$ for some Hermitian representation $V_{2}$, here $\oplus$ is orthogonal direct sum with respect to the Hermitian structure. Further, let $I:V_{0}\to V_{1}$ to be the standard inclusion. Miller previously defined the following filtration of $\Ell$:
$$F_{k}(\Ell):=\{\alpha\in\Ell:\text{rank}(\alpha-I)\leqslant k\}.$$

It has been demonstrated that the based quotients of this filtration form a stable splitting of $\Ell_{\infty}$; we refer the reader to Appendix \ref{ch:appendixa}, \cite{Miller}, \cite{Kitchloo} or \cite{Crabb} for a proof. We believe our work links in with the work of Miller via the below conjecture:

\begin{conjecture}\label{equivofmillersplittings} Let $i:F_{k}(\Ell)_{\infty}\to \Ell_{\infty}$ be the standard inclusion. Let $\pi:\Ell_{\infty}\to X_{k}$ be the stable projection onto the spectrum $X_{k}$ from Theorem \ref{themaintheorem}. Then denote the composite $F_{k}(\Ell)_{\infty}\to X_{k}$ by $res_{k}$. The map $res_{k}$ is an equivalence.
\end{conjecture}

We so far have been unable to prove this, however, we have been able to glean some evidence towards it and have proved certain special cases. Firstly we have the below result:

\begin{prop}\label{homeomorphicondensesubsets} $res_{k}$ provides a homeomorphism when restricted to certain open dense subsets.
\end{prop}
\begin{proof} We prove this unstably, considering the below map, which abusing notation somewhat we also call $res_{k}$:
$$S^{\svo}\wedge F_{k}(\Ell)_{\infty}\to \tilde{X}_{k}$$
$$(\alpha,\varphi)\mapsto (\alpha,\varphi|_{P_{k}(\alpha)}).$$
Next note that for each fixed $\alpha$ we can decompose $V_{0}$ into $P_{k}(\alpha)\oplus P_{k}(\alpha)^{\bot}$ and $V_{1}$ into $P_{k}(\alpha)\oplus P_{k}(\alpha)^{\bot}\oplus V_{2}$. Hence we can rewrite $(\alpha,\varphi)\in \svo\times F_{k}(\Ell)$ as $(\alpha,(\varphi_{ij}))$ where $(\varphi_{ij})$ is the below $3\times 2$-matrix:
$$\left(\begin{array}{cc}\varphi_{11}:P_{k}(\alpha)\to P_{k}(\alpha)&\varphi_{12}:P_{k}(\alpha)^{\bot}\to P_{k}(\alpha)\\
\varphi_{21}:P_{k}(\alpha)\to P_{k}(\alpha)^{\bot}&\varphi_{22}:P_{k}(\alpha)^{\bot}\to P_{k}(\alpha)^{\bot}\\
\varphi_{31}:P_{k}(\alpha)\to V_{2}&\varphi_{32}:P_{k}(\alpha)^{\bot}\to V_{2}\end{array}\right).$$
Similarly we can rewrite $(\alpha,\theta)\in \tilde{X}_{k}'$ as $(\alpha,(\theta_{i}))$ where $(\theta_{i})$ is the below vector:
$$\left(\begin{array}{c}\theta_{1}:P_{k}(\alpha)\to P_{k}(\alpha)\\
\theta_{2}:P_{k}(\alpha)\to P_{k}(\alpha)^{\bot}\\
\theta_{3}:P_{k}(\alpha)\to V_{2}\end{array}\right).$$
Under this identification it is clear that $res_{k}$ is the map sending the matrix $(\varphi_{ij})$ to its first column. Let $A$ and $B$ be the following spaces:
$$A:=\{(\alpha,\varphi)\in\svo\times F_{k}(\Ell):\Dim(P_{k}(\alpha))=k,\varphi_{11}-I\text{ is invertible}\}$$
$$B:=\{(\alpha,\theta)\in\tilde{X}_{k}':\Dim(P_{k}(\alpha))=k,\theta_{1}-I\text{ is invertible}\}.$$
$A$ and $B$ are open and dense in $\svo\times F_{k}(\Ell)$ and $\tilde{X}_{k}'$. Moreover $res_{k}$ clearly restricts to a map $A\to B$. We show that this is a homeomorphism, taking basepoints will then complete the proof. Let $(\alpha,\theta)\in B$, we show that it comes from an element of $A$. Since $\theta_{1}-I$ is invertible so is $\theta^{\dag}_{1}-I$. Let $\xi:=(\theta_{1}^{\dag}-I)^{-1}\theta_{2}^{\dag}$. We can now define $\xi_{1}:P_{k}(\alpha)^{\bot}\to P_{k}(\alpha)$ by $\xi_{1}:=(\theta_{1}-I)\xi$. Similarly, set $\xi_{2}:P_{k}(\alpha)^{\bot}\to P_{k}(\alpha)^{\bot}$ by $\xi_{2}:=\theta_{2}\xi+I$ and $\xi_{3}:P_{k}(\alpha)^{\bot}\to
V_{2}$ by $\xi_{3}:=\theta_{2}\xi$. Then we claim the below pair is in $A$:
$$\left(\alpha,\left(\begin{array}{cc}\theta_{1}&\xi_{1}\\
\theta_{2}&\xi_{2}\\
\theta_{3}&\xi_{3}\end{array}\right)\right).$$

We want to check firstly that the matrix is an isometry. Following through with the definitions and calculating the algebra easily demonstrates the below identity, which is enough to prove the claim:
$$\left(\begin{array}{ccc}\theta_{1}^{\dag}&\theta_{2}^{\dag}&\theta_{3}^{\dag}\\
\xi_{1}^{\dag}&\xi_{2}^{\dag}&\xi_{3}^{\dag}
\end{array}\right).\left(\begin{array}{cc}\theta_{1}&\xi_{1}\\
\theta_{2}&\xi_{2}\\
\theta_{3}&\xi_{3}\end{array}\right)=\left(\begin{array}{cc}I&0\\0&I\end{array}\right).$$

We now need to check that the matrix is in $F_{k}(\Ell)$, i.e. that the rank of the below matrix is at most $k$:
$$\left(\begin{array}{cc}\theta_{1}-I&\xi_{1}\\
\theta_{2}&\xi_{2}-I\\
\theta_{3}&\xi_{3}\end{array}\right)=\left(\begin{array}{cc}\theta_{1}-I&(\theta_{1}-I)\xi\\
\theta_{2}&\theta_{2}\xi\\
\theta_{3}&\theta_{3}\xi\end{array}\right).$$

Let $(x,y)\in V_{0}=P_{k}(\alpha)\oplus P_{k}(\alpha)^{\bot}$. Then applying the matrix above to $(x,y)$ is easily seen to be the same as applying $\theta-I$ to $(x,\xi(y))$. Hence it is clear that the matrix above has image $(\theta-I)(W)$ which is of dimension $k$. This shows that our built pre-image is in $A$ as claimed. Hence $res_{k}:A\to B$ surjective.

We now show the map is injective. This will then prove the claim. Let $(\alpha,\varphi)$ map to $(\alpha,\theta)$. Then $\varphi$ must be of the form below:
$$\left(\begin{array}{cc}\theta_{1}&\varphi_{12}\\
\theta_{2}&\varphi_{22}\\
\theta_{3}&\varphi_{32}\end{array}\right).$$
Moreover, as $\theta_{1}-I$ is invertible then $\theta_{1}-I$ has rank $k$. As $\varphi-I$ has rank at most $k$ we observe that it factors through $\theta-I$ and hence uniquely takes the form below for some $\xi'$:
$$\left(\begin{array}{cc}\theta_{1}-I&(\theta_{1}-I)\xi'\\
\theta_{2}&\theta_{2}\xi'\\
\theta_{3}&\theta_{3}\xi'\end{array}\right).$$

Chasing the definitions, calculating $\varphi^{\dag}\phi$ and recalling that this is $I$ leads to the conclusion that $\xi'$ must actually be $(\theta_{1}^{\dag}-I)^{-1}\theta_{2}^{\dag}$. This shows that $res_{k}:A\to B$ is injective and the proposition follows.
\end{proof}

This is not quite enough to prove the conjecture, but it does reduce it somewhat. Consider the map $res_{k}'$, the restriction of $res_{k}$ away from where it is a homeomorphism. Then if we can show that $res_{k}'$ is an equivalence we can retrieve by a standard cofibre sequences argument that $res_{k}$ is a homotopy equivalence. We can prove this in a special case:

\begin{prop}\label{directdimn22proof} Let $V_{0}=V_{1}$ be dimension $2$. Then $res_{1}'$ and hence $res_{1}$ are homotopy equivalences.
\end{prop}
\begin{proof} We again work unstably, let $V_{0}=V_{1}=V$. Consider $res_{1}':((s(V)\times F_{1}(\mathcal{L}(V,V)))\backslash A)_{\infty}\to (\tilde{X}_{1}'\backslash B)_{\infty}$. It is easy to see that the below identification holds:
$$(s(V)\times F_{1}(\mathcal{L}(V,V)))\backslash A\cong\{(\alpha,\varphi)\in s(V)\times F_{1}(\mathcal{L}(V,V)):\varphi_{11}=I\}.$$
Denote this space by $A^{c}$. Moreover, consider $(\alpha,\theta)\in\tilde{X}_{1}'\backslash B$. If $P_{1}(\alpha)$ is empty then $\theta$ is zero. Also, if $P_{1}(\alpha)$ is nonempty then $\theta$ is forced to be $I_{P_{1}(\alpha)}$. Thus we can identify $\tilde{X}_{1}'\backslash B$ with $s(V)$. Hence we are looking at the map below:
$$res_{1}':A^{c}_{\infty}\to S^{s(V)}$$
$$(\alpha,\varphi)\mapsto \alpha.$$

From here we can trivially remove a suspension from the $s(V)$-coordinate. We refer back to the proof of Proposition \ref{thebottommapisPT} - we take again the definition of $s_{0}(V)$ to be the space of selfadjoints of $V$ that have zero top eigenvalue. Thus we instead consider the following map:
$$res_{1}'':\bar{A}_{\infty}:=\{(\alpha,\varphi)\in s_{0}(V)\times F_{1}(\mathcal{L}(V,V)):\varphi_{11}=I\}_{\infty}\to s_{0}(V)_{\infty}.$$

Let $su(1)$ denote the unit complex numbers. We now define a map $g:[0,\infty]\times su(1) \times PV\to\bar{A}_{\infty}$ given by $(t,\lambda,L)\mapsto (0_{L}\oplus -tI_{L^{\bot}},I_{L}\oplus \lambda I_{L^{\bot}})$ and $(\infty,\lambda,L)\mapsto\infty$. This is surjective, and away from the endpoints of $[0,\infty]$ is injective. At $t=0$ it is again injective away from $\lambda=1$, but $(0,1,L)$ has the image $(0,I)$ independent of $L$. We thus have the below pushout:
$$\xymatrix{(\{0\}\times\{1\}\times PV)\coprod(\{\infty\}\times su(1)\times PV)\ar@{->>}[d]\ar[r]&[0,\infty]\times su(1)\times PV\ar@{->>}[d]\\
\{0\}\coprod\{\infty\}\ar@{ >->}[r]&\bar{A}_{\infty}}$$ 

There is also a standard map $[0,\infty]\times PV\to S^{s_{0}(V)}$ given by $(t,L)\mapsto 0_{L}\oplus-tI_{L^{\bot}}$. We can thus also form another pushout:
$$\xymatrix{\{0,\infty\}\times PV\ar@{->>}[d]\ar[r]&[0,\infty]\times PV\ar@{->>}[d]\\
\{0\}\coprod\{\infty\}\ar@{ >->}[r]& S^{s_{0}(V)}}$$

We can collapse out along the top of the first pushout to get the second pushout, which demonstrates the equivalence of $res_{1}''$ via comparing the two squares. This implies that $res_{1}'$ and hence $res_{1}$ are equivalences.
\end{proof}

Finally, we demonstrate one other equivalence. We note that this covers the result above, technically, but we include both due to the slightly tangential nature of the below proof.

\begin{prop}\label{topequivalence} $res_{d_{0}-1}$ is a stable equivalence.
\end{prop}
\begin{proof} We use a shorthand of $G_{d_{0}}$ for $G_{d_{0}}(V_{0})^{\Hom(T,V_{1}-V_{0})\oplus s(T)}$ for notational convenience. Recall three maps, firstly we recall our stable map $\phi_{d_{0}}:G_{d_{0}}\to \Ell_{\infty}$ from Theorem
\ref{themaintheorem}. Next we have from the proof of the Miller splitting the below collapse and homeomorphism. Details of these maps are found in the usual literature, we label the maps to be consistent with Appendix \ref{ch:appendixa}:
$$h_{d_{0}}:\Ell_{\infty}\to \frac{\Ell}{F_{d_{0}-1}(\Ell)}_{\infty}$$
$$\tau_{d_{0}}:\frac{\Ell}{F_{d_{0}-1}(\Ell)}_{\infty}\cong G_{d_{0}}.$$
We hence have the composition $\tau_{d_{0}}\circ h_{d_{0}}\circ \phi_{d_{0}}:G_{d_{0}}\to G_{d_{0}}$. We claim this is the identity. This is easy to see, however, as $\phi_{d_{0}}$ matches exactly the description of Miller's top-level splitting map and hence the composite is the identity - this is intrinsic to the proof of Miller's theorem. We can thus put together the below diagram of cofibre sequences, $\text{pt.}$ representing the contractible cofibre of the identity map:
$$\xymatrix{&&\text{pt.}\ar[dr]|\bigcirc&&\\
&G_{d_{0}}\ar[ur]\ar[dl]|\bigcirc&&\ar[ll]_{\text{id.}}G_{d_{0}}\ar[dl]^{\phi_{d_{0}}}&\\
F_{d_{0}-1}(\Ell)_{\infty}\ar@/_1pc/[rrrr]_{res_{d_{0}-1}}\ar[rr]&&\Ell_{\infty}\ar[ul]^{\tau_{d_{0}}\circ
h_{d_{0}}}\ar[rr]&&X_{d_{0}-1}\ar[ul]|\bigcirc}$$

We can build this diagram in the equivariant stable homotopy category. There, we can use the the octahedral axiom to add morphisms around the edge of the diagram and build the below cofibre sequence:
$$F_{d_{0}-1}(\Ell)_{\infty}\overset{res_{d_{0}-1}}{\longrightarrow}X_{d_{0}-1}\longrightarrow \text{pt.}.$$

Hence the map $res_{d_{0}-1}$ has contractible cofibre. It follows from Proposition \ref{zerocofibre} that it is an equivalence.
\end{proof}

\section{Conjecture Follow-up}\label{ConjectureFollowup}

We now follow up with what would happen should Conjecture \ref{equivofmillersplittings} hold. We firstly prove the below lemma:

\begin{lem}\label{splittingmaps} Recall from the Miller splitting the below collapse and homeomorphism. As we have done previously we use the labels of Appendix \ref{ch:appendixa} for these maps:
$$h_{k}:F_{k}(\Ell)_{\infty}\to \frac{F_{k}(\Ell)}{F_{k-1}(\Ell)}_{\infty}$$
$$\tau_{k}:\frac{F_{k}(\Ell)}{F_{k-1}(\Ell)}_{\infty}\cong G_{k}(V_{0})^{\Hom(T,V_{1}-V_{0})\oplus s(T)}.$$
Further, recall the stable map $\phi_{k}:G_{k}(V_{0})^{\Hom(T,V_{1}-V_{0})\oplus s(T)}\to X_{k}$. Finally, assume $res_{k}$ is a homotopy equivalence and let $res_{k}^{-1}$ denote the homotopy inverse. Then we have the below equivalence:
$$\tau_{k}\circ h_{k}\circ res_{k}^{-1}\circ \phi_{k}\simeq \text{id}:G_{k}(V_{0})^{\Hom(T,V_{1}-V_{0})\oplus s(T)}\to G_{k}(V_{0})^{\Hom(T,V_{1}-V_{0})\oplus s(T)}.$$
\end{lem}
\begin{proof} We again make notation simplifications, we write $G_{k}$ as a shorthand for $G_{k}(V_{0})^{\Hom(T,V_{1}-V_{0})\oplus s(T)}$ and $f$ as shorthand for the map $\tau_{k}\circ h_{k}\circ res_{k}^{-1}\circ \phi_{k}$. We now have the below diagram of cofibre sequences:
$$\xymatrix@C=0.8cm{&&C_{f}\ar[dr]|\bigcirc&&\\
&G_{k}\ar[ur]\ar[dl]|\bigcirc&&\ar[ll]_{f}G_{k}\ar[dl]^{\phi_{k}}&\\
F_{k-1}(\Ell)_{\infty}\ar[rr]&&F_{k}(\Ell)_{\infty}\simeq
X_{k}\ar[ul]^{\tau_{k}\circ
h_{k}}\ar[rr]&&X_{k-1}\ar[ul]|\bigcirc}$$ 

Via the octahedral axiom we can put arrows around the diagram and form the cofibre sequence below:
$$F_{k-1}(\Ell)_{\infty}\to X_{k-1}\to C_{f}.$$
It is a trivial exercise to observe that in this case the map $F_{k-1}(\Ell)_{\infty}\to X_{k-1}$ is $res_{k-1}$. This, however, is an equivalence and so has contractible cofibre. Thus $C_{f}$ is contractible and by Proposition \ref{zerocofibre} this gives that $f\simeq \text{id}$.
\end{proof}

In particular this result would demonstrate that should $res_{k}$ be an equivalence then $\phi_{k}$ has a one-way inverse. We now recall the theory from $\S$\ref{SplittingsviaCofibreSequences}. Applying this allows us to conjecture Claim \ref{thetowersplits} below; if Conjecture \ref{equivofmillersplittings} holds then the claim would follow:

\begin{claim}\label{thetowersplits} Let $V_{0}\leqslant V_{1}$ and recall the tower constructed in Theorem \ref{themaintheorem}. Then we have the following stable equivalences for all $0< k\leqslant d_{0}$:
$$X_{k}\simeq X_{k-1}\vee G_{k}(V_{0})^{\Hom(T,V_{1}-V_{0})\oplus s(T)}.$$

These imply the equivariant splittings given below, this includes an alternative method to building the equivariant Miller splitting:
$$X_{k}\simeq\bigvee_{r=0}^{k} G_{r}(V_{0})^{\Hom(T,V_{1}-V_{0})\oplus s(T)}$$
$$\Ell_{\infty}\simeq\bigvee_{r=0}^{d_{0}} G_{r}(V_{0})^{\Hom(T,V_{1}-V_{0})\oplus s(T)}.$$
\end{claim}

We can, however, prove a special case of this claim. Firstly we note that the map $res_{0}:S^{0}\to S^{0}$ is the identity. Combining this with Proposition \ref{topequivalence} and $\S$\ref{SplittingsviaCofibreSequences} proves the below theorem:

\begin{thm}\label{dimntwotowersplits} Let $V_{0}$ and $V_{1}$ be such that $V_{0}\leqslant V_{1}$ and $V_{0}$ is dimension $2$. Then we have the stable equivalences below and hence retrieve a dimension $2$ equivariant Miller
splitting:
$$X_{1}\simeq S^{0} \vee G_{1}(V_{0})^{\Hom(T,V_{1}-V_{0})\oplus s(T)}$$
$$\Ell_{\infty}\simeq X_{1}\vee G_{2}(V_{0})^{\Hom(T,V_{1}-V_{0})\oplus s(T)}$$
$$\Ell_{\infty}\simeq\bigvee_{r=0}^{2} G_{r}(V_{0})^{\Hom(T,V_{1}-V_{0})\oplus s(T)}.$$
\end{thm}

\section{An Equivalent Conjecture}\label{AnEquivalentConjecture}

We now state an alternate to Conjecture \ref{equivofmillersplittings}. The below claim is equivalent to the earlier conjecture; we demonstrate why in the next section. 

\begin{conjecture}\label{subrephomotopy} Let $\sigma_{k}$ denote the splitting map built by Miller, this notation matches that used in Appendix \ref{ch:appendixa}. Then the two maps below are homotopic:
$$res_{k}\circ\sigma_{k}:G_{k}(V_{0})^{\Hom(T,V_{1}-V_{0})\oplus s(T)}\to F_{k}(\Ell)_{\infty}\to X_{k}$$
$$\phi_{k}:G_{k}(V_{0})^{\Hom(T,V_{1}-V_{0})\oplus s(T)}\to X_{k}.$$
\end{conjecture}

The rest of this section is dedicated to providing evidence for this claim. Firstly, we consider the two maps unstably by suspending by $S^{\svo}$. For any fixed $W\in G_{k}(V_{0})$ and any $\gamma\in \Hom(W,V_{1})$ we can decompose $\gamma$ into $\alpha+\beta$, with $\alpha\in\Hom(W,W^{\bot})$ and $\beta\in \Hom(W,V_{1}\backslash W^{\bot})$. From
here it is a simple exercise to untangle the definitions to find that the two maps are explicitly given by the formulations below. We use the matrix notation as discussed in the proof of Lemma \ref{destabilizinglemma}. We also have $\rho$ and $\sigma$ defined previously in \ref{themaprho} and \ref{themapsigma}:
$$g_{0},g_{1}:G_{k}(V_{0})^{\Hom(T,V_{1})\oplus s(T^{\bot})}\to \tilde{X}_{k}$$
\small
$$g_{0}(W,\alpha+\beta,\psi)=\left\{\begin{array}{ll}\left(\delta:=\left(\begin{array}{cc}\log(\rho(\beta))&\alpha^{\dag}\\
\alpha&\psi\end{array}\right),(-\sigma(\beta)\oplus I_{W^{\bot}})|_{P_{k}(\delta)}\right)& \text{ $\beta$ injective}\\
\infty&\text{ otherwise}\end{array}\right.$$ \normalsize
$$g_{1}(W,\alpha+\beta,\psi)= ((\rho(\alpha+\beta)+e_{top}(\psi))|_{W}\oplus \psi|_{W^{\bot}},-\sigma(\alpha+\beta)).$$

We note that checking equivariance for $g_{0}$ is a potential issue here as $W$ is not necessarily going to be a subrepresentation. This doesn't cause any issues, however, as the group actions respect the matrix decomposition and the conjugation action on the maps $\alpha$, $\beta$ and $\gamma$ interacts well with the action on $W$ - it is easy to observe that $g_{0}(g.W,g.(\alpha+\beta)=g.\alpha+g.\beta,g.\psi)=(g.\delta,g.(-\sigma(\beta)\oplus I_{W^{\bot}})|_{P_{k}(g,\delta)})$ using similar techniques to equivariance arguments from earlier. We now consider the projections of both maps down to $\svo$, call these $g_{0}'$ and $g_{1}'$. We can build a number of proper homotopies between them using the functional calculus. Firstly let $\tilde{U}:=\{(x,t):x\geqslant 0, t\in [0,1], x+t>0\}$. Then let $f:\tilde{U}\to \Real$ be given by the below integral:
$$f(x,t):=\int_{1}^{x}u^{t-1}du=\left\{\begin{array}{ll}\frac{x^{t}-1}{t}& t\neq 0\\
\log (x)& t=0.\end{array}\right.$$

This is continuous and thus we can use it in conjunction with the functional calculus. Recall $\tilde{Z}_{k}$ from
\ref{factoringinsTperp} and define the space $U$ and map $h'$ as follows:
$$U:=\{(t,W,\alpha+\beta,\psi)\in [0,1]\times\tilde{Z}_{k}, t+e_{0}(\rho(\beta))>0\}$$
$$h':U\to \svo$$
$$(t,w,\alpha+\beta,\psi)\mapsto \left(\begin{array}{cc}f(\rho(t\alpha+\beta),t)+t(e_{top}(\psi)+1)&(1-t)\alpha^{\dag}\\
(1-t)\alpha&\psi\end{array}\right).$$

\begin{lem}\label{hprimeisproper} $h'$ is continuous and proper and hence the one-point compactification $h$ provides a homotopy between $g_{0}'$ and $g_{1}'$.
\end{lem}
\begin{proof} The only issue here is the question of properness. Let $C\subseteq \svo$ be a compact set. Then let
$(t,w,\alpha+\beta,\psi)\in h'^{-1}\{C\}$. We can trivially assume that $t\in(0,1)$. We need a bound on $\|\alpha+\beta\|$ and a bound on $\|\phi\|$. Firstly note that seeking a bound on $\|\alpha+\beta\|$ is equivalent when $t\in(0,1)$ to seeking a bound on $\|\rho(t\alpha+\beta)^{t}\|$. By assumption we have that there is a positive real number $R$ such that $\|\psi\|\leqslant R$, $\|(1-t)\alpha\|\leqslant R$ and $\|f(\rho(t\alpha+\beta),t)+t(e_{top}(\psi)+1)\|\leqslant R$. The first assumption immediately gives us the first bound we seek. For the second, let $e$ be an eigenvalue of $\rho(t\alpha+\beta)^{t}$. We are given the below inequality via the third assumption:
$$(e-1)/t+t(e_{top}(\psi)+1)\leqslant R.$$

We wish to show that there is a real number $S$ such that $e\leqslant S$. We first rearrange the above inequality:
$$e\leqslant 1+tR-t^{2}-t^{2}e_{top}(\psi).$$

Note that $1-t^{2}\leqslant 1$ for $t\in[0,1]$. Next note that as $R$ is positive we have $tR\leqslant R$. Finally, note that as $\|\psi\|\leqslant R$ we have $-R\leqslant e_{top}(\psi)\leqslant R$ and hence that $-R\leqslant -e_{top}(\psi)\leqslant R$. It follows that $-t^{2}e_{top}(\psi)\leqslant t^{2}R\leqslant R$. Combining the above inequalities gives us the below inequality:
$$e\leqslant 1+2R.$$

This is enough to show that $h'$ is proper and the claim follows.
\end{proof}

Thus the problem reduces to modifying $h'$ so that it lifts to $\tilde{X}_{k}$. We immediately remark that isometry factors of the maps $g_{0}$ and $g_{1}$ already look very similar. Maneuvering from $-\sigma(\beta)$ to $-\sigma(\alpha+\beta)$ can be done essentially instantly with the homotopy $t\alpha+\beta$. Both isometries, however, have their issues. For $-\sigma(\beta)\oplus I_{W^{\bot}}$ when we move away from $\beta$ being injective, as we do immediately in the above homotopy, then we have $-\sigma(\beta)\oplus I_{W^{\bot}}$ undefined on $\Ker(\beta)\oplus
0\subseteq W\oplus W^{\bot}$. For $\sigma(\alpha+\beta)$ the problem in $W$ is less of an issue as the only problem area is $\Ker(\alpha)\cap\Ker(\beta)$; however, $\sigma(\alpha+\beta)$ is undefined on $W^{\bot}$. We can circumvent this a little via extending $\sigma(\alpha+\beta)$ to $\sigma(\alpha+\beta)\oplus I_{\Ker(\alpha^{\dag})}$.

Hence to solve this problem we need to suitably modify $h'$ into producing a selfadjoint endomorphism $\delta_{t}$ that keeps one of the two below maps injective at all times:
\begin{enumerate}
\item $(-\sigma(\beta)\oplus I_{W^{\bot}})|_{P_{k}(\delta_{t})}$.
\item $(-\sigma(\alpha+\beta)\oplus I_{\Ker(\alpha^{\dag})})|_{P_{k}(\delta_{t})}$.
\end{enumerate}

We note that we only require injectivity as another application of $\sigma$ will turn our maps into isometries. We also note that we can easily move from $1$ to $2$ via the homotopy $-\sigma(t\alpha+\beta)\oplus I_{\Ker(t\alpha)^{\dag}}$. The required modification of $h'$ has, however, so far proved elusive.

\section{Follow-up to the Second Conjecture}\label{FollowuptotheSecondConjecture}

We now survey what happens should Conjecture \ref{subrephomotopy} hold. This is to demonstrate that it is equivalent to Conjecture \ref{equivofmillersplittings}.

\begin{prop}\label{thetowerisequivtothefiltration} Assume Conjecture \ref{subrephomotopy} is true. Then the maps $res_{k}$ are equivalences and the maps $\tau_{k}\circ h_{k}\circ res_{k}^{-1}$ provide one-sided inverses to each $\phi_{k}$.
\end{prop}

\begin{proof} We prove this by downward induction. We only state the inductive step, the initial step follows from \ref{topequivalence}. Assume that $res_{k}$ gives an equivalence with inverse $res_{k}^{-1}$. Then we first consider $\tau_{k}\circ h_{k}\circ res_{k}^{-1}\circ \phi_{k}$, from the assumption of Conjecture \ref{subrephomotopy} we have $\phi_{k}\simeq res_{k}\circ \sigma_{k}$. Thus we have the below equivalence:
$$\tau_{k}\circ h_{k}\circ res_{k}^{-1}\circ \phi_{k}\simeq\tau_{k}\circ h_{k}\circ res_{k}^{-1}\circ res_{k}\circ \sigma_{k}\simeq\tau_{k}\circ h_{k}\circ \sigma_{k}.$$

That this is the identity is given in the proof of the Miller splitting. We now prove that $res_{k-1}$ gives an equivalence. Using the shorthand $G_{k}$ from the proof of Lemma \ref{splittingmaps} we have the below diagram of cofibre sequences:
$$\xymatrix@C=0.8cm{&&\text{pt.}\ar[dr]|\bigcirc&&\\
&G_{k}\ar[ur]\ar[dl]|\bigcirc&&\ar[ll]_{\text{id.}}G_{k}\ar[dl]^{\phi_{k}}&\\
F_{k-1}(\Ell)_{\infty}\ar@/_1pc/[rrrr]_{res_{k-1}}\ar[rr]&&F_{k}(\Ell)_{\infty}\simeq
X_{k}\ar[ul]^{\tau_{k}\circ
h_{k}}\ar[rr]&&X_{k-1}\ar[ul]|\bigcirc}$$

That $res_{k-1}$ is an equivalence then follows from Proposition \ref{zerocofibre} and the result follows by downward induction.
\end{proof}

Thus Conjecture \ref{subrephomotopy} proves Conjecture \ref{equivofmillersplittings} and moreover proves Claim
\ref{thetowersplits}. It is trivial to see that Conjecture \ref{equivofmillersplittings} proves Conjecture
\ref{subrephomotopy}. Hence proving either one will exhibit the Miller splitting in the subrepresentation special case of Theorem \ref{themaintheorem}. This is a future goal of the author.
\chapter{The Equivariant K-Theory of the Tower}
\label{ch:ch8}

\section{Basic Results}\label{BasicResults}

We now begin to calculate algebraic invariants associated to our tower \ref{themaintheorem}. In particular, we aim to calculate the reduced equivariant complex topological $K$-theory of the diagram. We claim that the diagram in $K$-theory holds a certain nice algebraic structure that interacts well with the topological structure. We will not prove this, however, we will conjecture the result in $\S$\ref{TheKTheoryoftheTowerConjecture} after providing evidence for the claim. We open with an overview of the basic $K$-theory results we will need, we claim no originality for these results. Good surveys of the theory can be found in \cite{segalktheory} and Chapter $14$ of \cite{mayetal}, while \cite{AtiyahKTheory} gives an overview of the non-equivariant theory while alluding to the equivariant extensions. The basic building block of equivariant $K$-theory, and indeed of any cohomology theory, is the $K$-theory of a point:

\begin{defn}\label{therepresentationring} Let $R(G)$ denote the complex representation ring of $G$.
\end{defn}

This is the starting point of equivariant $K$-theory, it is easy to see that $\Kzero(\text{pt.})=R(G)$ as bundles over a point are representations. Further, we can switch to reduced $K$-theory $\redKstar$ and observe the below lemma:

\begin{lem}\label{reducedktheoryofsnought} $\redKzero(S^{0})=R(G)$ and $\redKone(S^{0})=0$.
\end{lem}

This is the $K$-theory of the bottom level of the tower. To calculate the next stage it is clear that we will need to know the equivariant $K$-theory of a projective representation. The result with proof is found as Proposition $(3.9)$ in \cite{segalktheory}:

\begin{prop}\label{thektheoryofprojectiverepresentations} Let $V$ be a representation of dimension $d$, and let $f_{V}(t)\in R(G)[t]$ be the polynomial given by:
$$f_{V}(t):= \sum_{k=0}^{d}t^{d-k}(-1)^{k}.\lambda^{k}(V).$$

Here $\lambda^{k}(V)$ is the $k^{\text{th}}$ exterior power of $V$. Further let $T_{V}$ be the tautological bundle over $PV$. Then the equivariant $K$-theory of $PV$ is as follows:
$$\Kzero(PV)=\frac{R(G)[T_{V},T_{V}^{-1}]}{f_{V}(T_{V})}$$
$$\Kone(PV)=0.$$
\end{prop}

For higher levels of the tower we will need to know the equivariant $K$-theory of Grassmannians. We do not provide the calculations here but we note that we can extend the standard non-equivariant theory detailed in many sources - for example result $2.7.14$ in \cite{AtiyahKTheory} or result $3.12$ in Chapter $4$ of \cite{KaroubiKTheory} - via the techniques laid out by Segal in Section $3$ of \cite{segalktheory}. This will yield results generalizing the proposition above. For the calculated portion of the tower, however, the explicit statement of the result for projective representations is all that is required. Noting that the theory is similar for Grassmannians, however, provides further motivation for the conjecture which follows in $\S$\ref{TheKTheoryoftheTowerConjecture}. 

Finally, we will have to deal with the $K$-theory of Thom spaces over vector bundles. In particular wish to know how to deal with sums and differences of these bundles. The two results below deal with this problem; the first calculates the $K$-theory of sums of bundles while the second covers the $K$-theory of virtual vector bundles.

\begin{lem}\label{addingvectorbundlesinKtheory} Let $V$ and $W$ be $G$-vector bundles over compact base $X$ such
that $\Kone(X)=0$. Then:
\begin{itemize}
\item $\redKzero(X^{V\oplus W})\cong \redKzero(X^{V})\otimes_{\Kzero(X)} \redKzero(X^{W})$.
\item $\redKone(X^{V\oplus W})\cong 0$.
\end{itemize}
\end{lem}
\begin{proof} The first result is built via the Thom homomorphism, details of which are found in the sourced text. We first note that we have the following Thom homomorphisms:
$$\phi_{*}^{V}:\Kzero(X)\to \Kzero (V)$$
$$\phi_{*}^{W}:\Kzero(X)\to \Kzero (W).$$
As $X$ is compact we have that $\Kzero(V)\cong \redKzero (X^{V})$ and $\Kzero(W)\cong\redKzero(X^{W})$. Thus we have the below pushout:
$$\xymatrix{\Kzero(X)\ar[d]_{\phi_{*}^{W}}\ar[r]^{\phi_{*}^{V}}&\redKzero(X^{V})\ar[d]\\
\redKzero(X^{W})\ar[r]&\redKzero(X^{V})\otimes_{\Kzero(X)} \redKzero(X^{W})}$$
We also note that $V\oplus W$ is a bundle over both $V$ and $W$ so we have the following Thom homomorphisms:
$$\phi_{*}^{VW}:\Kzero(V)\cong\redKzero(X^{V})\to \Kzero(V\oplus W)\cong\redKzero(X^{V\oplus W})$$
$$\phi_{*}^{WV}:\Kzero(W)\cong\redKzero(X^{W})\to \Kzero(V\oplus W)\cong\redKzero(X^{V\oplus W}).$$
If $\phi_{*}:\Kzero(X)\to \redKzero(X^{V\oplus W})$ is the Thom homomorphism associated to the Whitney sum $V\oplus W$ of $V$ and $W$ then by result $(3.4)$ in \cite{segalktheory} we have that $\phi_{*}$ commutes with $\phi_{*}^{VW}\circ\phi_{*}^{V}$ and $\phi_{*}^{WV}\circ\phi_{*}^{W}$. Thus we have the below commutative diagram:
$$\xymatrix{\Kzero(X)\ar[dr]^{\phi_{*}}\ar[d]_{\phi_{*}^{W}}\ar[r]^{\phi_{*}^{V}}&\redKzero(X^{V})\ar[d]^{\phi_{*}^{VW}}\\ 
\redKzero(X^{W})\ar[r]_{\phi_{*}^{WV}}&\redKzero(X^{V\oplus W})}$$

Via the universal property of pushouts we have a map $\redKzero(X^{V})\otimes_{\Kzero(X)} \redKzero(X^{W})\to
\redKzero(X^{V\oplus W})$. This becomes an isomorphism if we can show that there exists a unique map $j$ which makes the diagram below commute:
$$\xymatrix{\Kzero(X)\ar[rr]^{\phi^{V}_{*}}\ar[dr]_{\phi_{*}}\ar[dd]_{\phi^{W}_{*}}&&\redKzero(X^{V})\ar[dl]_{\phi_{*}^{VW}}\ar[dd]\\
&\redKzero(X^{V\oplus W})\ar[dr]_{j}&\\
\redKzero(X^{W})\ar[ur]^{\phi_{*}^{WV}}\ar[rr]&&\redKzero(X^{V})\otimes_{\Kzero(X)}
\redKzero(X^{W})}$$
We have a unique map $j$, however - such a map can be built from the two bundle projections $V\oplus W\to V$ and $V\oplus W\to W$. Further, this will be unique. That it also makes the above diagram commute is implicitly covered in $(3.4)$ of \cite{segalktheory}. Thus the first claim follows from the universal property of pushouts.

For the second claim we note that the Thom homomorphism $\phi_{*}^{1}:\Kone(X)\to \redKone(X^{V\oplus W})$ gives an
isomorphism via multiplication by some element, as detailed in the literature. As $\Kone(X)$ is zero, however, this gives
$\redKone(X^{V\oplus W})=0$ as required. 
\end{proof}

\begin{lem}\label{virtualbundlesinKtheory} With the same set-up as \ref{addingvectorbundlesinKtheory} we can calculate the $K$-theory of the virtual bundle $X^{V-W}$:
\begin{itemize}
\item $\redKzero(X^{V- W})\cong\Hom_{\Kzero(X)}( \redKzero(X^{W}), \redKzero(X^{V}))$.
\item $\redKone(X^{V-W})\cong 0$.
\end{itemize}
\end{lem}
\begin{proof} Firstly we note that although we did not explicitly mention it the above Lemma \ref{addingvectorbundlesinKtheory} holds for virtual vector bundles - in the stable case we still have a Thom homomorphism which has the factoring properties used in the above proof. Note that $X^{V}\cong X^{V-W\oplus W}$, then via an application of Lemma \ref{addingvectorbundlesinKtheory} we have:
$$\redKzero(X^{V})\cong\redKzero(X^{V-W})\otimes_{\Kzero(X)}\redKzero(X^{W}).$$
Taking adjoints then gives us the first claim. The second claim is standard and follows as per above from the Thom homomorphism actually being an isomorphism.
\end{proof}

\section{The K-Theory of the Bottom of the Tower}\label{TheKTheoryoftheBottomoftheTower}

We now partially calculate the equivariant $K$-theory of the bottom of the tower; we wish to provide evidence for a detailed conjecture on the $K$-theory of the tower. We first consider general Thom spaces of the form $PV^{\Hom(T_{V},W)}$. We have the below equivariant identification:
$$PV^{\Hom(T_{V},W)}\cong \frac{P(V\oplus W)}{PW}.$$

Let $i_{VW}:PV\to P(V\oplus W)$ be the inclusion. Then the following identity is easily observable:
$$\redKstar(PV^{\Hom(T_{V},W)})\cong \Ker(i_{VW}^{*}:\Kstar(P(V\oplus W))\to\Kstar(PW)).$$

It is clear that $i^{*}_{VW}$ is the evident restriction. We note one further detail before we begin the calculations:

\begin{lem}\label{ktheorypolyssplitup} Recall the polynomials $f_{V}(t)$ defined in Proposition \ref{thektheoryofprojectiverepresentations}. Then we have:
$$f_{V\oplus W}(t)= f_{V}(t).f_{W}(t).$$
\end{lem}
\begin{proof} We state a single simple fact from the theory of exterior algebras, that $\lambda^{*}(V\oplus W)\cong \lambda^{*}(V)\otimes \lambda^{*}(W)$. The result then trivially follows from this.
\end{proof}

We can now calculate $\redKstar(PV_{0}^{\Hom(T,V_{1})})$ and $\redKstar(PV_{0}^{\Hom(T,V_{0})})$. The first calculation goes as follows:
$$\begin{array}{rcl}
\redKzero(PV_{0}^{\Hom(T,V_{1})})&=&\Ker\left(\frac{R(G)[T_{V_{0}\oplus V_{1}},T_{V_{0}\oplus V_{1}}^{-1}]}{f_{V_{0}\oplus V_{1}}(T_{V_{0}\oplus V_{1}})}\overset{i_{V_{0}V_{1}}^{*}}{\to}\frac{R(G)[T_{V_{1}},T_{V_{1}}^{-1}]}{f_{ V_{1}}(T_{ V_{1}})}\right)\\
&=&\Ker\left(\frac{R(G)[T_{V_{0}\oplus V_{1}},T_{V_{0}\oplus V_{1}}^{-1}]}{f_{V_{0}}(T_{V_{0}\oplus V_{1}}).f_{V_{1}}(T_{V_{0}\oplus V_{1}})}\overset{i_{V_{0}V_{1}}^{*}}{\to}\frac{R(G)[T_{V_{1}},T_{V_{1}}^{-1}]}{f_{ V_{1}}(T_{ V_{1}})}\right)\\
&=&\frac{R(G)[T,T^{-1}].f_{V_{1}}(T)}{R(G)[T,T^{-1}].f_{V_{0}}(T).f_{V_{1}}(T)}.
\end{array}
$$

The last step follows from observing that $i^{*}_{V_{0}V_{1}}$ is the restriction. In the final step here we use $T$ as a dummy variable; under the projection onto $PV_{i}$ we have $T_{V_{0}\oplus V_{1}}$ mapping to $T_{V_{i}}$ so we switch to $T$ to represent both $T_{V_{0}\oplus V_{1}}$ and its projection. Thus we have that $\redKzero(PV_{0}^{\Hom(T,V_{1})})$ is a free module of rank $1$ over $\Kzero(PV_{0})$. It is also clear to see that
$\redKone(PV_{0}^{\Hom(T,V_{1})})$ is zero as $\Kone(P(V_{0}\oplus V_{1}))$, $\Kone(PV_{0})$ and $\Kone(PV_{1})$ are all zero. We now perform a similar calculation for $\redKzero(PV_{0}^{\Hom(T,V_{0})})$, noting that $\redKone(PV_{0}^{\Hom(T,V_{0})})$ will again be zero:
$$\begin{array}{rcl}
\redKzero(PV_{0}^{\Hom(T,V_{0})})&=&\Ker\left(\frac{R(G)[T_{V_{0}\oplus V_{1}},T_{V_{0}\oplus V_{1}}^{-1}]}{f_{V_{0}\oplus V_{1}}(T_{V_{0}\oplus V_{1}})}\overset{i_{V_{0}V_{0}}^{*}}{\to}\frac{R(G)[T_{V_{0}},T_{V_{0}}^{-1}]}{f_{ V_{0}}(T_{ V_{0}})}\right)\\
&=&\Ker\left(\frac{R(G)[T_{V_{0}\oplus V_{1}},T_{V_{0}\oplus V_{1}}^{-1}]}{f_{V_{0}}(T_{V_{0}\oplus V_{1}}).f_{V_{1}}(T_{V_{0}\oplus V_{1}})}\overset{i_{V_{0}V_{0}}^{*}}{\to}\frac{R(G)[T_{V_{0}},T_{V_{0}}^{-1}]}{f_{ V_{0}}(T_{ V_{0}})}\right)\\
&=&\frac{R(G)[T,T^{-1}].f_{V_{0}}(T)}{R(G)[T,T^{-1}].f_{V_{0}}(T)^{2}}.
\end{array}
$$

Again we switch to $T$ as a dummy variable. We now combine the two calculations with Lemma \ref{virtualbundlesinKtheory} to give:
$$\redKzero(PV_{0}^{\Hom(T,V_{1}-V_{0})})\cong \frac{R(G)[T,T^{-1}]}{f_{V_{0}}(T)}\otimes \frac{f_{V_{1}}(T)}{f_{V_{0}}(T)}$$
$$\redKone(PV_{0}^{\Hom(T,V_{1}-V_{0})})\cong 0.$$

Finally, we note that over $PV_{0}$ the bundle $s(T)$ is just the trivial bundle $\Real$. Now recall the map $\delta_{1}$ in the tower of Theorem \ref{themaintheorem}:
$$\delta_{1}:S^{0}\to \Sigma PV_{0}^{\Hom(T,V_{1}-V_{0})\oplus s(T)}\cong \Sigma^{2}PV_{0}^{\Hom(T,V_{1}-V_{0})}.$$

We have the map $\delta_{1}^{*}$ on $K$-theory. Using Bott periodicity, this map has domain and codomain as follows:
$$\delta_{1}^{*}:\redKstar(\Sigma^{2}PV_{0}^{\Hom(T,V_{1}-V_{0})})\cong\redKstar(PV_{0}^{\Hom(T,V_{1}-V_{0})})\to \redKstar (S^{0}).$$

This map involves already calculated $K$-theory groups. We first note that all $\redKone$ groups have been zero, thus $\delta^{*}_{1}$ is the trivial map $0\to 0$ on $\redKone$. We can also observe that on $\redKzero$ the map $\delta^{*}_{1}$ has the following domain and codomain:
$$\delta_{1}^{*}:\redKzero(PV_{0}^{\Hom(T,V_{1}-V_{0})})\to \redKzero (S^{0})$$
$$\delta_{1}^{*}:\frac{R(G)[T,T^{-1}]}{f_{V_{0}}(T)}\otimes\frac{ f_{V_{1}}(T)}{f_{V_{0}}(T)}\to R(G).$$

Note we are implicitly suppressing here that $\redKzero(\Sigma^{2}PV_{0}^{\Hom(T,V_{1}-V_{0})})$ is not strictly
$\redKzero(PV_{0}^{\Hom(T,V_{1}-V_{0})})$, rather it is more correctly the below module:
$$\frac{R(G)[T,T^{-1}]}{f_{V_{0}}(T)}\otimes\frac{f_{V_{1}}(T)}{f_{V_{0}}(T)}.\omega.$$

Here $\omega$ is the Bott element in $\tilde{K}_{G}^{2}(S^{2})$; Bott periodicity gives an isomorphism $\redKzero(PV_{0}^{\Hom(T,V_{1}-V_{0})})\cong\redKzero(\Sigma^{2}PV_{0}^{\Hom(T,V_{1}-V_{0})})$ via multiplication with this $\omega$. We are now in a position to describe how the map $\delta^{*}_{1}$ acts in certain special cases:

\begin{prop}\label{itsaresiduemap} For $G$ a finite abelian group the map $\delta_{1}^{*}$ is the residue map.
\end{prop}
\begin{proof} Recalling Section \ref{TheBottomoftheTower} we observe that $\delta_{1}^{*}$ is the composition of the stable Pontryagin-Thom collapse map $\Sigma^{-\svo}j^{!}:S^{0}\to PV_{0}^{-\tau_{PV_{0}}}$ with the twisting zero-section map $g:PV_{0}^{-\tau_{PV_{0}}}\to \Sigma^{2}PV_{0}^{\Hom(T,V_{1}-V_{0})}$. Hence in cohomology we have a Gysin map $(\Sigma^{-\svo}j^{!})^{*}:\redKstar(PV_{0}^{-\tau_{PV_{0}}})\to\redKstar(S^{0})$. As we are working over a finite abelian group this is the residue map by a result of Strickland, Theorem 21.35 in \cite{StricklandMulticurves}. That the composition is the residue map easily follows from here.
\end{proof}

We conjecture that this result holds in the more general setting of compact Lie groups. Non-equivariantly, Quillen in
\cite{QuillenCobordismFGL} stated that Gysin maps of the form above are residue maps in complex cobordism. From this it is not a far stretch to hypothesize that the above proposition holds in more generality. A toric version of the result is an immediate future goal of the author and a generalization to all compact Lie groups would produce a complete calculation of the $K$-theory of the bottom of the tower. We can offer no proof to this conjecture, however.

The result does, however, construct the bottom of the tower in some generality. We now check one further result, that of the subrepresentation case. Similar to our conjecture in Chapter \ref{ch:ch7}, we expect the algebraic invariants to suitably degenerate to mirror the non-equivariant case when both theories are applicable. Non-equivariantly we have the result \ref{kitchloosresult} of Kitchloo from \cite{Kitchloo}. We conjecture, as stated briefly in \ref{introconjecture} and to be stated in more detail in the next section, that the tower is a Koszul complex. In order for the result
to suitably retrieve the Miller splitting on the cohomological level, we require that the following result holds:

\begin{prop}\label{bottomdifferentialiszeroinsubrep} For $G$ a finite abelian group the map $\delta_{1}^{*}$ is zero if and only if $V_{0}\leqslant V_{1}$.
\end{prop}
\begin{proof} Let $V_{0}\leqslant V_{1}$, we claim that the meromorphic function $f_{V_{1}}(T)/f_{V_{0}}(T)$ has no singularities. As $G$ is abelian, note that any representation $V$ of dimension $d$ decomposes into a sum of $1$-dimensional representations $L_{1}\oplus\ldots\oplus L_{d}$; this follows from (3.7) of \cite{segalktheory}. Under this decomposition the polynomial $f_{V}(t)$ can take the below form, this is covered in detail in $2.5.3$ and $2.7.1$ of \cite{AtiyahKTheory}:
$$f_{V}(t)=\prod_{i=1}^{d}(t-[L_{i}]).$$

With this identification it is clear that $f_{V_{1}}(T)/f_{V_{0}}(T)$ will have no singularities - it will take the form of $\prod(t-[L_{i}])$ for each line that is part of the decomposition of $V_{1}$ but not part of the corresponding
decomposition of $V_{0}$. This is enough to show that $\delta_{1}^{*}$ is zero.

For the converse, if $\delta_{1}^{*}$ is zero then it follows that $f_{V_{1}}(T)/f_{V_{0}}(T)$ has no singularities. Using the line decomposition above it is easy to see that $f_{V_{1}}(T)/f_{V_{0}}(T)$ having no singularities implies that for every $(T-[L])$-component of the polynomial $f_{V_{0}}(T)$ there is a corresponding component of $(T-[L])$ in $f_{V_{1}}(T)$ to cancel and make $f_{V_{1}}(T)/f_{V_{0}}(T)$ nonsingular. This gives $V_{0}$ decomposed into a sum of lines which all lie in the decomposition of $V_{1}$, thus $V_{0}\leqslant V_{1}$ as required.
\end{proof}

This gives a partial description of the equivariant $K$-theory of the bottom of the tower, suggesting a pattern which we now proceed to hypothesize.

\section{The K-Theory of the Tower - Conjecture}\label{TheKTheoryoftheTowerConjecture}

After detailing the $K$-theory of the bottom of the tower, the next logical goal would be to extend the $K$-theory results up the tower to try and build meaningful structure. An obstruction to this generalization following through, however, is the behaviour of the $K$-theory of the $G_{k}(V_{0})^{s(T)}$-part.

Previously, we built the $K$-theory of $PV_{0}^{\Hom(T,V_{1})}$ and similar out of equivariant identities; we then used the lemmas detailed at the end of $\S$\ref{BasicResults} to combine everything together. We can somewhat repeat this process up the tower - we have already noted that there will be a resemblance between the constructions of $\Kstar(PV_{0})$ and $\Kstar (G_{k}(V_{0}))$. Furthermore, we can modify the techniques used throughout
$\S$\ref{TheKTheoryoftheBottomoftheTower} in ways that will hopefully follow through towards a more general result. However, in the detailed case we completely circumvented the issue of the $s(T)$-bundle by noting that if $T$ is a line bundle then $s(T)\cong \Real$. For higher tautological bundles, however, we simply cannot do this. Thus one of the issues becomes detailing $\redKstar (G_{k}(V_{0})^{s(T)})$ and combining it into our calculations above via Lemma \ref{addingvectorbundlesinKtheory}.

The other obstruction to our progress is the reliance on Gysin maps for our dimension $2$ proof. Further up the tower we do not have access to such a nice geometric structure. Thus we would have to find another way to suitably describe the maps.

For reasons of time, we are unable to include these calculations within the document. We can, however, note certain things about the issues at hand and disseminate the evidence towards producing a conjecture. Firstly, we note that what we wish to regard are the groups and maps below, recalling $\delta_{k}$ and $\phi_{k-1}$ from Theorem \ref{themaintheorem}:
$$\redKstar(G_{k}(V_{0})^{s(T)\oplus\Hom(T,V_{1}-V_{0})})\overset{(\delta_{k}\circ\phi_{k-1})^{*}}{\longrightarrow}\redKstar(G_{k-1}(V_{0})^{s(T)\oplus\Hom(T,V_{1}-V_{0})}).$$

We immediately note that we have suppressed that $\delta_{k}^{*}$ is actually a map to suspension. It is clear that from
$\Kone(G_{k}(V_{0}))$ being zero we will have one of the two groups of $\redKstar(G_{k}(V_{0})^{s(T)\oplus\Hom(T,V_{1}-V_{0})})$ also being zero. The suspension in the $\delta_{k}^{*}$ will make these zeroes match up - the addition of the bundle $s(T)$ will amongst other things add in a $K$-theory dimension shift for $k$ odd, this is easy to see by recalling that for $V$ a representation of dimension $n$ we have $s(V)$ as a real representation of dimension $n^{2}$. Thus we know that what we will have are maps between the `interesting' $K$-groups in this case, which matches with the behaviour at the bottom.

We also recall one other piece of evidence we have, that of the subrepresentation case. Geometrically, we conjecture that we can link in and retrieve the work of Miller from our construction. Moreover non-equivariantly we know what happens regarding the cohomology of the picture, Kitchloo detailed in \cite{Kitchloo} that it becomes an exterior algebra with interesting additional structure. Passing to an equivariant theory in this case should also behave in a similar fashion; if the subrepresentation conjecture were to hold as we believe it does then we would expect to retrieve
the algebraic equivalent in $K$-theory.

Combining the evidence regarding the subrepresentation case with the bottom of the tower results allows us to confidently conjecture an algebraic structure to be held by the $K$-theory of the tower. We take the definition of a Koszul complex to be as follows, this will be the structure our tower will hold. One of a number of references
covering the structure is Chapter XXI, $\S$4 of \cite{LangAlgebra}:

\begin{defn}\label{AKoszulComplex} Let $R$ be a ring and $(x_{1},\ldots,x_{r})$ a sequence of elements in $R$. Then a Koszul complex is given by the following information:
\begin{enumerate}
\item Modules $K_{i}$ for $i=0,\ldots,r$ given by setting $K_{0}$ to be $R$, $K_{1}$ to be the free $R$-module generated by elements $e_{1},\ldots,e_{r}$ and each $K_{i}$ to be $\lambda^{i}_{R}(K_{1})$, the $i^{th}$ exterior power of $K_{1}$ over $R$.
\item Differential maps $d_{i}:K_{i}\to K_{i-1}$ given by $d_{1}(e_{j})=x_{j}$ and $d_{i}$ sending $(e_{j_{1}}\wedge \ldots\wedge e_{j_{i}})$ to $\sum_{k=1}^{i}(-1)^{k-1}x_{j_{k}}(e_{j_{1}}\wedge\ldots\wedge\widehat{e_{j_{k}}}\wedge\ldots\wedge e_{j_{i}})$.
\end{enumerate}
The notation $e_{j_{1}}\wedge\ldots\wedge\widehat{e_{j_{k}}}\wedge\ldots\wedge e_{j_{i}}$ refers to the exterior product achieved by removing the $e_{j_{k}}$-element from the product. We note that the above definition gives a complex with $d^{2}=0$ and refer to the cited text for the theory behind the construction.
\end{defn}

We conclude by stating in detail what we believe to be true for the equivariant $K$-theory of the tower built in \ref{themaintheorem}. A proof of the conjecture below is one of the goals of the author in follow-up to this document.

\begin{conjecture}\label{TheFinalConjecture} Applying $\redKstar$ to the tower detailed in \ref{themaintheorem} will produce the structure of a Koszul complex as follows:
\begin{itemize}
\item The base-ring will be $R(G)$.
\item $K_{1}$ will be up to isomorphism $\redKzero(PV_{0}^{\Hom(T,V_{1}-V_{0})})$, the isomorphism arising
from a double suspension and Bott periodicity.
\item Each $K_{r}\cong\redKstar(G_{r}(V_{0})^{s(T)\oplus \Hom(T,V_{1}-V_{0})})$ will be isomorphic to the $k^{th}$ exterior power of $K_{1}$.
\item The differential $d_{1}:K_{1}\to K_{0}=R(G)$ will be $\delta_{1}^{*}$, a residue map.
\item The differentials $d_{n}:K_{n}\to K_{n-1}$ will be $(\delta_{n}\circ \phi_{n-1})^{*}$.
\end{itemize}
Moreover, $V_{0}\leqslant V_{1}$ if and only if each differential is zero. Thus if $V_{0}\leqslant V_{1}$ the complex has zero differentials and we retrieve that $\Kstar(\Ell)$ is an exterior algebra generated over $R(G)$ by $\redKstar(PV_{0}^{s(T)\oplus \Hom(T,V_{1}-V_{0})})$ with each exterior power being the equivariant $K$-theory of each component in the splitting.
\end{conjecture}

\appendix
\chapter{A Proof of the Miller Splitting}
\label{ch:appendixa}

As an appendix we present details of a proof of the Miller splitting of $\Ell$ as first proved in \cite{Miller}. The mathematics is unoriginal but we choose to present the result in terms of our language and style so we can refer when needed in the main document to parts of the proof. In particular we will be explicit about the change from $u(T)$ to $s(T)$. Due to this focus we will only briefly address aspects of the proof not pertinent to our work; we refer the reader to the previous literature where needed. Familiarity with some of the notation and results in the main document is assumed. Other proofs of this result can be found in \cite{Crabb} and \cite{Kitchloo}.

Let $V_{0}\leqslant V_{1}$ via the map $I:V_{0}\to V_{1}$, i.e. $V_{1}\cong V_{0}\oplus V_{2}$ for some representation $V_{2}$ and $I$ is the inclusion. We define a filtration on $\Ell$ as follows:
$$F_{k}(\Ell):=\{\alpha\in\Ell:\text{rank}(\alpha-I)\leqslant k\}.$$

We first note that $F_{k}(\Ell)$ is a submanifold of $\Ell$. Moreover, it is clear that when based the inclusion $i_{k}:F_{k-1}(\Ell)_{\infty}\to F_{k}(\Ell)_{\infty}$ is a cofibration; it is a closed inclusion between Hausdorff spaces. Let $h_{k}$ be the corresponding collapse map, then we have the below cofibre sequence by Proposition \ref{classicalcofibrationcofibres}:
$$\xymatrix{\frac{F_{k}(\Ell)}{F_{k-1}(\Ell)}_{\infty}\ar[dr]|\bigcirc_{\eta_{k}}&F_{k}(\Ell)_{\infty}\ar[l]_{h_{k}}\\
&F_{k-1}(\Ell)_{\infty}\ar[u]_{i_{k}}}$$

Two results then imply the splitting:

\begin{aprop}\label{thequotientisthethomspace} There is a homeomorphism:
$$\tau_{k}:\frac{F_{k}(\Ell)}{F_{k-1}(\Ell)}_{\infty}\cong G_{k}(V_{0})^{\Hom(T,V_{1}-V_{0})\oplus s(T)}.$$
\end{aprop}

\begin{aprop}\label{thecompositionsplits} There are stable maps $\sigma_{k}$ as below:
$$\sigma_{k}:G_{k}(V_{0})^{\Hom(T,V_{1}-V_{0})\oplus s(T)}\to F_{k}(\Ell)_{\infty}.$$
These are such that the composition $\tau_{k}\circ h_{k}\circ\sigma_{k}$ is fibrewise homotopic to the identity.
\end{aprop}

These two results are enough to imply the Miller Splitting Theorem via the information detailed in $\S$\ref{SplittingsviaCofibreSequences}:

\begin{athm}[Miller]\label{themillersplittingthm} $\Ell_{\infty}$ splits stably as:
$$\Ell_{\infty}\simeq\bigvee_{k=0}^{n}G_{k}(V_{0})^{\Hom(T,V_{1}-V_{0})\oplus s(T)}.$$
\end{athm}

We first prove Proposition \ref{thequotientisthethomspace}. Let $\Gamma_{k}$ be the submanifold of $\Ell\times G_{k}(V_{0})$ given by:
$$\{(\psi,W):\psi\in\Ell,W\in G_{k}(V_{0}),\psi|_{W^{\bot}}=I|_{W^{\bot}}\}.$$

Now setting $\pi_{0}$ and $\pi_{1}$ to be the two evident projections it is clear that the image of $\pi_{0}$ is
$F_{k}(\Ell)$ and that $\pi_{1}$ makes $\Gamma_{k}$ into a fibre bundle over $G_{k}(V_{0})$. Moreover, there is a section $\iota_{k}$ sending $W\mapsto (-I|_{W}\oplus I|_{W^{\bot}},W)$ and thus we have a concept of basepoint for $\Gamma_{k}$, we refer to this based version as $\Gamma_{k}^{\infty}$.

We now claim the following:

\begin{alem}\label{diffeomorphismthingy} There is a diffeomorphism of bundles over $G_{k}(V_{0})$ between $\Gamma_{k}$ and the below bundle, we leave the details as to why this is a bundle to the previously cited papers:
$$\{(W,\phi):W\in G_{k}(V_{0}),\phi\in\mathcal{L}(W,W\oplus V_{2})\}.$$

We note here that as $V_{1}\cong V_{0}\oplus V_{1}$ and $V_{0}\cong W\oplus W^{\bot}$ we can also denote $W\oplus V_{2}$ as $V_{1}\backslash W^{\bot}$. Using this notation we refer to the above as $\mathcal{L}(T,V_{1}\backslash T^{\bot})$. This diffeomorphism is given by the map below:
$$\Gamma_{k}\to\mathcal{L}(T,V_{1}\backslash T^{\bot})$$
$$(W,\psi)\mapsto (W,\psi- (0_{W}\oplus I_{W^{\bot}})).$$
\end{alem}

We now note that this diffeomorphism is filtration preserving when composed with $\pi_{0}$. Thus we have the below relative diffeomorphism:
$$(\mathcal{L}(T,V_{1}\backslash T^{\bot}),F_{k-1}(\mathcal{L}(T,V_{1}\backslash T^{\bot})))\cong(
F_{k}(\Ell),F_{k-1}(\Ell)).$$

We partner this result with the following lemma:

\begin{alem}\label{cayleytransform} Let $W\in G_{k}(V_{0})$. We have the following Cayley transform:
$$s(W)\oplus \Hom(W,V_{1}-V_{0})\cong \mathcal{L}(W,V_{1}\backslash W^{\bot})\backslash F_{k-1}(\mathcal{L}(W,V_{1}\backslash W^{\bot})).$$
This is given in the case $V_{0}=V_{1}$ by the map and inverse as follows - the generalization is detailed in \cite{Crabb}:
$$\delta\mapsto\frac{((-i\delta/2)-1)}{((-i\delta/2)+1)}$$
$$\frac{2}{i}\frac{(\phi+1)}{(\phi-1)}\mapsfrom \phi.$$
\end{alem}

Combining the lemma with the relative diffeomorphism gives the homeomorphisms $\tau_{k}$ as required and proves Proposition \ref{thequotientisthethomspace}.

We now construct the maps of Proposition \ref{thecompositionsplits}. For the top map $\sigma_{d_{0}}:G_{d_{0}}(V_{0})^{\Hom(T,V_{1}-V_{0})\oplus s(T)}\to \Ell_{\infty}$ we take as an unstable representative the collapse map $\kappa^{!}:S^{\aich}\to S^{\svo}\wedge\Ell_{\infty}$ from Proposition \ref{thetaEalpha}. We have in the main document observed the below bundle identity: 
$$S^{\svo}\wedge G_{d_{0}}(V_{0})^{\Hom(T,V_{1}-V_{0})\oplus s(T)}\cong S^{\aich}.$$

Thus we look at the composition $(1\wedge(\tau_{d_{0}}\circ h_{d_{0}}))\circ \kappa^{!}$, a self map of $S^{\aich}$. This map is homotopic to the identity, which is easiest to see if $V_{0}=V_{1}$. In that case this map is the collapse corresponding to the embedding:
$$S^{2\svo}\to S^{2\svo}\cong S^{\End(V_{0})}$$
$$(\delta,\alpha)\mapsto -\frac{((-i\delta/2)-1)}{((-i\delta/2)+1)}\Exp(\alpha).$$
That the composition is the identity then follows from the above map having the identity as its derivative at $(0,0)$. Thus we have constructed a map $\sigma_{d_{0}}$ with the required properties.

For $\sigma_{k}$ we again note an observation made in the main document:
$$S^{\svo}\wedge G_{k}(V_{0})^{\Hom(T,V_{1}-V_{0})\oplus s(T)}\cong G_{k}(V_{0})^{\Hom(T,V_{1})\oplus s(T^{\bot})}.$$

Thus we can build $\sigma_{k}$ unstably, mapping out of $G_{k}(V_{0})^{\Hom(T,V_{1})\oplus s(T^{\bot})}$. We introduce the space $\mathcal{K}_{k}'$ to be the fibre bundle given below, with one-point compactification $\mathcal{K}_{k}$:
$$\{W\in G_{k}(V_{0}),\gamma\in \inj(W,V_{1}\backslash W^{\bot}),\beta\in\Hom(W,W^{\bot}),\psi\in s(W^{\bot})\}.$$

We define a map $\tilde{K}:G_{k}(V_{0})^{\Hom(T,V_{1})\oplus s(T^{\bot})}\to \mathcal{K}_{k}$ as follows. First we note that when fixing $W$ we can standardly decompose $f\in\Hom(W,V_{1})\cong \Hom(W,W\oplus W^{\bot}\oplus V_{2})$ into $(g,h)$ with $g\in\Hom(W,W\oplus V_{2})\cong \Hom(W,V_{1}\backslash W^{\bot})$ and $h\in\Hom(W,W^{\bot})$. This gives us a homeomorphism from $G_{k}(V_{0})^{\Hom(T,V_{1})}$ to $G_{k}(V_{0})^{\Hom(T,V_{1}\backslash T^{\bot})\oplus
\Hom(T,T^{\bot})}$. We then compose fibrewise with collapse maps built via Proposition \ref{thetaEalpha}; we have collapses which we apply on each fibre given by $\kappa_{W}^{!}:S^{\Hom(W,V_{1}\backslash W^{\bot})}\to
\inj(W,V_{1}\backslash W^{\bot})_{\infty}$ for each fixed $W$. We then take the identity on each $s(W^{\bot})$ and $\Hom(W,W^{\bot})$ to build the map $\tilde{K}$.

We now note that for $W\in G_{k}(V_{0})$ we have fibrewise identifications from Proposition \ref{thetaEalpha} given as follows:
$$\mathcal{L}(W,V_{1}\backslash W^{\bot})\times s(W)\cong \inj(W,V_{1}\backslash W^{\bot}).$$

We can apply these fibrewise on $\mathcal{K}_{k}$. This allows us to map $\mathcal{K}_{k}$ homeomorphically onto $S^{\svo}\wedge \Gamma_{k}^{\infty}$ via the bundle identity $s(T)\oplus s(T^{\bot})\oplus \Hom(T,T^{\bot})\cong \svo$ from the proof of Lemma \ref{destabilizinglemma} and via the diffeomorphism \ref{diffeomorphismthingy}, call this map $j_{k}$. We then take as an unstable representative of $\sigma_{k}$ the following composite:
$$G_{k}(V_{0})^{\Hom(T,V_{1})\oplus s(T^{\bot})}\overset{\tilde{K}}{\to}\mathcal{K}_{k}\overset{j_{k}}{\cong} S^{\svo}\wedge \Gamma_{k}^{\infty}\overset{1\wedge \pi_{0}}{\to}S^{\svo}\wedge F_{k}(\Ell)_{\infty}.$$

When composed with $(1\wedge(\tau_{k}\circ h_{k}))$ this is fibrewise homotopic to the identity for similar reasons to why $(1\wedge(\tau_{d_{0}}\circ h_{d_{0}}))\circ \kappa^{!}$ is homotopic to the identity. Thus we have built stable $\sigma_{k}$ proving Proposition \ref{thecompositionsplits} - from here it is simple to retrieve a proof of The Miller Splitting Theorem \ref{themillersplittingthm}.

\bibliography{thesisbib}
\addcontentsline{toc}{chapter}{Bibliography}
\bibliographystyle{plain}
\end{document}